\documentclass[11pt]{amsart}

\usepackage{amssymb,amsfonts,theoremref,mathrsfs}
\usepackage{hyperref}
\usepackage{mdwlist}
\usepackage{thm-restate}

\usepackage{amssymb,textcomp}
\usepackage{enumerate}

\usepackage[utf8]{inputenc}
\usepackage[T1]{fontenc}

\usepackage[mathscr]{eucal}
\usepackage{xcolor}
\usepackage{tikz}
\tikzstyle{startstop} = [rectangle, rounded corners, minimum width=3cm, minimum height=1cm,text centered, draw=black, fill=red!30]

\usepackage{graphicx}

\usepackage{hyperref}
 \hypersetup{
     colorlinks=true,
     linktocpage=true,
     linkcolor=red,
     filecolor=blue,
     citecolor = blue,
     urlcolor=cyan,
     }

\usepackage[a4paper, twoside=false, vmargin={2cm,3cm}, includehead]{geometry}

\usepackage{comment}
\newtheorem{lemma}{Lemma}[section]
\newtheorem{theorem}[lemma]{Theorem}
\newtheorem*{theorem*}{Theorem}
\newtheorem{corollary}[lemma]{Corollary}

\newtheorem{proposition}[lemma]{Proposition}
\newtheorem*{proposition*}{Proposition}

\newtheorem*{problem*}{Problem}

\theoremstyle{definition}
\newtheorem*{claim*}{Claim}

\newtheorem{definition}[lemma]{Definition}

\newtheorem{example}[lemma]{Example} 

\newtheorem{remark}[lemma]{Remark}

\DeclareMathOperator*{\E}{\mathbb{E}}

\newcommand{\C}{{\mathbb C}}

\newcommand{\F}{{\mathbb F}}
\newcommand{\N}{{\mathbb N}}

\newcommand{\R}{{\mathbb R}}

\newcommand{\Z}{{\mathbb Z}}

\newcommand{\CC}{{\mathcal C}}

\newcommand{\CH}{{\mathcal H}}

\newcommand{\CK}{{\mathcal K}}

\newcommand{\CN}{{\mathcal N}}
\newcommand{\CP}{{\mathcal P}}
\newcommand{\CQ}{{\mathcal Q}}

\newcommand{\ba}{{\mathbf{a}}}
\renewcommand{\b}{{\mathbf{b}}}

\newcommand{\bc}{{\mathbf{C}}}

\newcommand{\bn}{{\mathbf{n}}}
\newcommand{\be}{{\mathbf{e}}}

\newcommand{\bp}{{\mathbf{P}}}
\newcommand{\p}{{\mathbf{P}}}
\newcommand{\bq}{{\mathbf{Q}}}
\newcommand{\q}{{\mathbf{Q}}}

\newcommand{\bv}{{\mathbf{v}}}
\newcommand{\x}{{\mathbf{x}}}
\newcommand{\bx}{{\mathbf{x}}}

\newcommand{\bs}{{\mathbf{s}}}

\newcommand{\balpha}{{\boldsymbol{\alpha}}}

\newcommand{\bgamma}{{\boldsymbol{\gamma}}}

\newcommand{\bbeta}{{\boldsymbol{\beta}}}

\newcommand{\uh}{{\underline{h}}}
\newcommand{\uk}{{\underline{k}}}

\newcommand{\un}{{\underline{n}}}
\newcommand{\um}{{\underline{m}}}
\newcommand{\uu}{{\underline{u}}}

\newcommand{\eps}{\varepsilon}
\newcommand{\ueps}{{\underline{\epsilon}}}

\DeclareMathOperator{\supp}{supp}

\newcommand{\norm}[1]{\left\Vert #1\right\Vert}

\newcommand{\inv}{^{-1}}

\newcommand{\veps}{\varepsilon}

\newcommand{\abs}[1]{\mathopen{}\left| #1\mathclose{}\right|}

\newcommand{\Bigabs}[1]{\Bigl| #1 \Bigr|}
\newcommand{\brac}[1]{\mathopen{}\left( #1 \mathclose{}\right)}

\newcommand{\Bigbrac}[1]{\Bigl( #1 \Bigr)}
\newcommand{\floor}[1]{\left \lfloor #1 \right \rfloor}
\newcommand{\ceil}[1]{\left \lceil #1 \right \rceil}

\newcommand{\BK}[1]{{\color{red}{#1}}}

\title[]{Quantitative concatenation\\ for polynomial box norms}
\author{Noah Kravitz, Borys Kuca, James Leng}
\date{}

\address[Noah Kravitz]{Department of Mathematics, Princeton University, Princeton, NJ 08540, USA}
\email{nkravitz@princeton.edu}

\address[Borys Kuca]{Faculty of Mathematics and Computer Science, Jagiellonian University, 30-348 Krak\'ow, Poland}
\email{borys.kuca@uj.edu.pl}

\address[James Leng]{Department of Mathematics, UCLA, Los Angeles, CA 90095, USA}
\email{jamesleng@math.ucla.edu}

\begin{document}

\begin{abstract}
Using PET and quantitative concatenation techniques, we establish box-norm control with the ``expected'' directions for counting operators for general multidimensional polynomial progressions, with at most polynomial losses in the parameters.  Such results are often useful first steps towards obtaining explicit upper bounds on sets lacking instances of given such progressions.  In the companion paper \cite{KKL24b}, we complete this program for sets in $[N]^2$ lacking nondegenerate progressions of the form $(x, y), (x + P(z), y), (x, y + P(z))$, where $P \in \mathbb{Z}[z]$ is any fixed polynomial with an integer root of multiplicity $1$.
\end{abstract}

\maketitle

\tableofcontents

\section{Introduction}

One of the starting points of arithmetic combinatorics is Roth's Theorem \cite{R53}, which says that every subset of the natural numbers with positive upper Banach density contains a nontrivial $3$-term arithmetic progression $x,\; x+z,\; x+2z$.  Szemer\'edi \cite{Sz75} famously extended this result to $\ell$-term arithmetic progressions for all $\ell$.  The furthest-reaching generalization of this line of work is the Polynomial Szemer\'edi Theorem of Bergelson and Leibman \cite{BL96}, which asserts that positive-upper-Banach-density subsets of $\Z^D$ always contain nontrivial ``polynomial progressions'' under certain weak local conditions on the polynomials involved.

The quantitative bounds from Roth's original proof of his theorem have been improved many times, most recently by Kelley and Meka \cite{KM23}.  Szemer\'edi provided effective, albeit enormous, bounds for his theorem, and Gowers' work \cite{Go01} on obtaining more ``reasonable'' bounds inaugurated the field of higher-order Fourier analysis, where recent developments have led to improved bounds for Szemer\'edi's Theorem \cite{LSS24b}.  The proof of Bergelson and Leibman \cite{BL96}, by contrast, uses ergodic-theoretic methods and is completely ineffective.  Obtaining ``reasonable'' bounds for the Bergelson--Leibman Theorem is a difficult problem that was highlighted by Gowers e.g. in \cite{Go98b}.

In an early development, S\'ark\"ozy \cite{Sa78a, Sa78b} obtained bounds for nonlinear polynomial progressions of length two. {Following on Gowers' breakthrough, Green \cite{Gr02} and Prendiville \cite{Pre17} later studied arithmetic progressions in which the difference is restricted to an image set of a fixed homogeneous polynomial.
In parallel to these developments, Bourgain and Chang \cite{BC17}, Peluse \cite{Pel18}, and Dong, Li, and Sawin \cite{DLS17} examined $x,\; x + z,\; x + z^2$ and other length-3 linearly independent polynomial patterns over $\F_p$. This culminated in a spectacular work of Peluse \cite{Pel19} that introduced the by-now standard degree-lowering technique to obtain bounds for arbitrarily long polynomial progressions with linearly independent polynomials over the finite fields.}  Building on that, Peluse and Prendiville \cite{PP19} gave an effective polynomial Szemer\'edi theorem for the progression $x,\; x + z,\; x + z^2$. Peluse \cite{Pel19} then extended this method to obtain reasonable bounds for polynomial progressions involving polynomials with distinct degrees. Quite a few quantitative polynomial Szemer\'edi theorems \cite{BB23, HLY21, Kuc21a, Kuc23, Kuc22b, Leng23a, Leng22, PSS23, SW25} have been proven in recent years, and our work \cite{KKL24b} can be seen as progress in this direction. 

Most arguments in this line of research have begun by using the PET\footnote{Some sources record that ``PET'' stands for ``polynomial ergodic theorem'' or ``polynomial exhaustion technique''.} induction scheme of Bergelson \cite{Ber87} to show that the counting operator for the associated polynomial progression is bounded by an appropriate Gowers norm or, more generally, an appropriate box norm.\footnote{For the definitions of the concepts appearing in the introduction, including Gowers norms, box norms and multiplicative derivatives, we refer the reader to Section \ref{S: definitions}. The details of the PET induction scheme are elucidated in Section \ref{S: PET}.}  Green and Tao \cite{GT10b} showed that iterated applications of the Cauchy--Schwarz inequality in the style of Gowers \cite{Go98a, Go01} suffice to give Gowers-norm-control for positive-complexity linear configurations in any number of dimensions.  For $1$-dimensional polynomial progressions in the finite-field setting, a na\"ive application of the PET argument produces the required norm-control.  For example, if the counting operator
\begin{align*}
\sum_{x,z\in\Z/N\Z} f_0(x)f_1(x+z)f_2(x+z^2) 
\end{align*}
is large for some $1$-bounded functions $f_0, f_1, f_2:\Z/N\Z\to\C$ (where $N$ is a large prime), then a short PET argument gives that the Gowers norm $\norm{f_2}_{U^3(\Z/\N\Z)}$ is also large, with polynomial dependences.  In the integer setting, by contrast, the PET argument alone does not suffice.  If the counting operator
\begin{align*}
    \sum_{x, z\in\Z}f_0(x)f_1(x+z)f_2(x+z^2)
\end{align*}
is large for some $1$-bounded functions $f_0, f_1, f_2:\Z\to\C$ supported on $[N]:=\{1, \ldots, N\}$ for some large $N\in\N$, then PET gives only the largeness of the quantity
\begin{align*}
    \sum_{h_1, h_2\in[N^{1/2}]}\sum_{m_1,m_2,m_3\in[N^{1/2}]}\sum_x\Delta_{h_1 m_1, h_2 m_2, (h_1+h_2) m_3}f_2(x).
\end{align*}
This expression is essentially an average of many ``arithmetic box norms'' along progressions with common differences 
\begin{align}\label{E: single dim directions}
    h_1, h_2, h_1 + h_2.    
\end{align}
For more complicated $1$-dimensional polynomial progressions, the common differences arising are \textit{multilinear} polynomials in parameters $h_1, h_2, \ldots$ {(meaning that they are linear combinations of monomials $h_1^{\eps_1}\cdots h_r^{\eps_r}$ with $\eps_1, \ldots, \eps_r\in\{0,1\}$; in other words, they are of degree at most 1 in each variable)}.  At this point, it is desirable to ``concatenate'' these norms into a single Gowers norm of higher degree, so that one can use the numerous techniques involving the inverse theory of the Gowers norms.

The first such concatenation result was proven by Tao and Ziegler \cite{TZ16}.  Their approach was infinitary, however, and gave no quantitative bounds on the losses incurred while passing from an average of box norms to a single concatenated Gowers norm. In \cite{Pel20, PP19}, Peluse and Prendiville developed a different approach that could provide quantitative bounds\footnote{See also \cite{PPS24}, in which Peluse, Prendiville and Shao extend the quantitative concatenation approach from \cite{PP19} to obtain Gowers-norm-control for the configuration $(x_1, x_2),\; (x_1+z, x_2),\; (x_1, x_2+z^2)$.}. They showed, among other things, that the averages of box norms arising from applying PET to $1$-dimensional polynomial progressions
\begin{align}\label{E: single dim progressions}
    x, \; x + P_1(z),\; \ldots,\; x + P_\ell(z)
\end{align}
can be concatenated into a single Gowers norm that control the counting operator for \eqref{E: single dim progressions} with only polynomial losses. This concatenation result was instrumental in their proof of quantitative bounds for subsets of $[N]$ lacking progressions \eqref{E: single dim progressions} in the case where $P_1, \ldots, P_\ell\in\Z[z]$ have distinct degrees.

The situation is even more complicated in higher dimensions.  Consider, for example, the \textit{multidimensional} polynomial pattern
\begin{equation}\label{E: 2-dim prog}
(x_1, x_2), (x_1 + z^2+z, x_2), (x_1, x_2 + z^2 + z).    
\end{equation}
This time, we 
would like to show that if
\begin{align}\label{E: count of 2-dim}
    \sum_{x_1,x_2,z\in\Z}f_0(x_1, x_2)f_1(x_1 + z^2 + z, x_2)f_2(x_1, x_2 + z^2 + z)
\end{align}
is large 
for some $1$-bounded functions $f_0, f_1, f_2:\Z\to\C$ supported on $[N]^2$, then so is\footnote{Naturally, we also want to obtain the largeness of certain averages of box norms of $f_0, f_1$, with $\{\be_2,\be_2-\be_1\}$ is replaced by $\{\be_1, \be_2\}$ or $\{\be_1, \be_1-\be_2\}$, respectively.}
\begin{align*}
    \sum_{h_1, \ldots, h_{2t}\in\Z}\sum_{x_1,x_2\in\Z}\Delta_{h_1 \be_2, \ldots, h_t \be_2, h_{t+1}(\be_2 - \be_1), \ldots, h_{2t}(\be_2 - \be_1)}f_2(x_1, x_2)
\end{align*}
for some positive integer $t=O(1)$, where $\be_1, \be_2$ are the usual coordinate vectors in $\mathbb{Z}^2$, and $\Delta$ refers to the symmetric multiplicative derivative (defined in \eqref{E: symmetric derivative} below).  (From here, one can hope to use degree-lowering arguments to decrease $t$ to  its optimal value.)  The usual PET argument gives control of \eqref{E: count of 2-dim} by an average of degree-$7$ box norms of $f_2$ in the directions
\begin{gather}\label{E: examplealgebraicdirections}
    2(h_2+h_3)(\be_2 - \be_1) + 2h_1 \be_2,\ 2h_2(\be_2 - \be_1) + 2h_1 \be_2,\ 2 h_3(\be_2 - \be_1) + 2h_1 \be_2,\\ 
    \nonumber 2h_1 \be_2,\ 2(h_2+h_3)(\be_2 - \be_1),\ 2h_2(\be_2 - \be_1),\  2h_3(\be_2 - \be_1).
\end{gather}
In contrast to the single-dimensional examples described above, this example requires us to concatenate an average of many box norms involving {linear combinations} of \emph{two} principal directions $\be_2, \be_2-\be_1$ into a single box norm involving  these {two} principal directions. 
More generally, in the study of multidimensional polynomial progressions
\begin{align}\label{E: multi dim progressions}
    \bx, \; \bx + \p_1(z),\; \ldots,\; \bx + \p_\ell(z)
\end{align} 
for some choice of $\p_1, \ldots, \p_\ell\in\Z^D[z]$, one is confronted with the task of concatenating averages of box norms along
\begin{align}\label{E: multi dim directions general}
    \bc_1(h_1, \ldots, h_r), \ldots,\bc_s(h_1, \ldots, h_r),
\end{align}
for some \emph{multilinear} polynomials $\bc_1, \ldots, \bc_s\in\Z^D[h_1, \ldots, h_r]$ that arise from the application of PET to \eqref{E: multi dim progressions}; the goal is to obtain control by a single box norm, which usually involves several principal directions.

Let us compare the directions in \eqref{E: single dim directions} with the directions in \eqref{E: examplealgebraicdirections}.  The polynomial directions in \eqref{E: single dim directions}, as examined by Peluse and Prendiville, are all multiples of the same principal direction in $\Z$.  The polynomial directions in \eqref{E: multi dim progressions}, by contrast, are linear combinations of the two principal directions $\be_2-\be_1, \be_2$ that come from the leading coefficients of the polynomials in \eqref{E: 2-dim prog} and their pairwise differences.  Thus the problem of quantitative concatenation of box norms gains a layer of difficulty when one passes from the $1$-dimensional setting of Peluse and Prendiville to the multidimensional setting.

The goal of this paper is to prove a new quantitative concatenation result that is robust enough to handle the averages of polynomial box norms output by the PET argument for multidimensional polynomial progressions. {The second author~\cite{Kuc23} recently obtained a similar result in the finite-field setting and used it to give bounds for subsets of $\F_p^D$ avoiding (multidimensional) polynomial progressions.  Our work can be understood as an extension of his methods to the integer setting, which presents substantial new difficulties coming from the poor equidistribution behavior of polynomials over the integers.}  The main output of our work is the following box-norm-control for general multidimensional polynomial progressions.  See \eqref{E: box norm} below for our definition of box norms.

\begin{theorem}\label{T: single box norm bound simplified}
Let $d, \ell,D \in \N$ and $T>0$.  There exist a positive integer $t=O_{d, \ell}(1)$ and a positive real $C=O_{d,\ell,D,T}(1)$ such that the following holds.  Let $N \in \N$ and $\delta>0$ satisfy $N\geq C\delta^{-C}$,  and let $\bp_1, \ldots, \bp_\ell\in\Z^D[z]$ be  polynomials of the form $\bp_j(z) = \sum_{i=0}^d \bbeta_{ji}z^i$ with coefficients of size at most $T$. 
Then for all $1$-bounded functions $f_0, \ldots, f_\ell:\Z^D\to\C$ supported on $[N]^D$, the bound
    \begin{align}\label{E: concatenation of polynomials assumption -1}
            \abs{\sum_{\bx\in\Z^D}\E_{z\in[N^{1/d}]}
            f_0(\bx)\cdot f_1(\bx+\p_1(z))\cdots f_\ell(\bx+\p_\ell(z))} \geq \delta N^D
    \end{align}
    implies that
    \begin{align}\label{E: the box norm -1}
        \norm{f_j}_{\{(\bbeta_{jd_{jj'}}-\bbeta_{j'd_{jj'}})\cdot[\pm N^{d_{jj'}/d}]:\; j'\in[0,\ell],\; j'\neq j\}^{ t}}^{2^{\ell t}}\geq \frac{1}{C}\delta^{C}N^{D}
    \end{align}
    for all $j\in[0,\ell]$, where $d_{jj'} = \deg (\bp_j - \bp_{j'})$ and $\bp_0 = \mathbf{0}$. 
\end{theorem}


Informally, Theorem \ref{T: single box norm bound simplified} asserts that every multidimensional polynomial progression can be controlled, with polynomial losses, by a single box norm of bounded degree along the ``correct'' boxes. Let us explain in words why the box norm in \eqref{E: the box norm -1} is what one ``should'' expect, and specifically why it involves intervals of length $\asymp N^{d_{jj'}/d}$ in the directions $\bbeta_{jd_{jj'}}-\bbeta_{j'd_{jj'}}$.   
Consider the special case where
all of the $f_i$'s except for $f_j, f_{j'}$ are the constant function $1$.\footnote{Obviously, such functions do not satisfy the assumption of being supported on $[N]^D$, but this example can be modified slightly to fit the assumptions of the more general Theorem \ref{T: single box norm bound} below.}  Shifting $\bx\mapsto \bx - \bp_{j'}(z)$, we find that \eqref{E: concatenation of polynomials assumption -1} simplifies to 
\begin{align*}
\sum_{\bx\in\Z^D}\E_{z\in[N^{1/d}]}f_{j'}(\bx)f_j(\bx + \bp_j(z) - \bp_{j'}(z)) \geq \delta N^D.
\end{align*}
The polynomial $\bp_j - \bp_{j'}$ has degree $d_{jj'}$ and leading coefficient $\bbeta_{jd_{jj'}}-\bbeta_{j'd_{jj'}}$.
Applying the classical van der Corput inequality $d_{jj'}$ times removes the lower-order terms of $\bp_j - \bp_{j'}$
and (more or less) results in the inequality
\begin{multline*}
    \sum_{\bx\in\Z^D}\E_{h_1, \ldots, h_{d_{jj'}}\in [N^{1/d}]}f_j(\bx)f_j(\bx + (d!/{(d-d_{jj'})!})(\bbeta_{jd_{jj'}} -\bbeta_{j'd_{jj'}})h_1\cdots h_{d_{jj'}})\\
    \gg_{d_{jj'}}\delta^{O_{d_{jj'}}(1)}N^D.
\end{multline*}
One can use iterated applications of the Cauchy--Schwarz inequality to replace the monomial $h_1\cdots h_{d_{jj'}}$ with a single variable ranging over $[N^{d_{jj'}/d}]$, and this morally gives the largeness of a box norm of $f_j$ along the progression $(\bbeta_{jd_{jj'}}-\bbeta_{j'd_{jj'}})\cdot[\pm N^{d_{jj'}/d}]$.  The large multiplicity $t$ appearing in Theorem \ref{T: single box norm bound simplified} is an artifact of the large number of applications of the Cauchy--Schwarz inequality in both the PET induction scheme and the concatenation arguments.

For $D>1$, the results of Theorem \ref{T: single box norm bound simplified} and the more general Theorem \ref{T: single box norm bound} below are completely novel (in the integer setting).  For $D=1$, they recover Peluse's concatenation result \cite[Theorem 6.1]{Pel20} and also address some configurations that are not covered by Peluse's work.
\begin{example}\label{Ex: Peluse}
    As an example of a 1-dimensional configuration not covered by \cite[Theorem 6.1]{Pel20}, take $x,\; x + z^2,\; x+ 2z^2,\; x + 2z^2 + z$. If 
    \begin{align*}
        \Bigabs{\sum_{x\in\Z}\E_{z\in[N^{1/2}]}f_0(x)f_1(x + z^2) f_2(x + 2z^2)f_3(x + 2z^2 + z)} \geq \delta N,
    \end{align*}
    then Theorem \ref{T: single box norm bound simplified} (combined with standard properties of box norms listed in Lemma \ref{L: properties of box norms}) gives a positive integer $t=O(1)$ for which
    \begin{align*}
        \norm{f_0}_{U^t([\pm N])}^{2^t}, \norm{f_1}_{U^t([\pm N])}^{2^t}\gg \delta^{O(1)}N \quad \text{and}\quad \norm{f_2}_{U^t([\pm N^{1/2}])}^{2^t}, \norm{f_3}_{U^t([\pm N^{1/2}])}^{2^t}\gg \delta^{O(1)}N.
    \end{align*}
    The disparity in the scales of these Gowers norms comes from the polynomials $2z^2,\; 2z^2 + z$ having the same leading coefficient: Since the difference of these polynomials is a linear polynomial that ranges over an interval of length $\asymp N^{1/2}$ rather than $\asymp N$, the shifting parameters in the Gowers norms controlling $f_2$ and $f_3$ must be at the smaller scale $N^{1/2}$. For configurations like this in which several polynomials have the same leading coefficients, \cite[Theorem 6.1]{Pel20} does not give any nontrivial box-norm-control, so our result is novel in this case. We remark as a side note that such restricted Gowers norms are often difficult to use in applications because the inverse theorem for the $U^t([\pm N^{1/2}])$-norm involves correlations with $(t-1)$-step nilsequences on short intervals of length $\asymp N^{1/2}$, with no discernible global (i.e., scale-$N$) structure.
\end{example}

When the polynomials $\p_j$ have distinct degrees, the conclusion of Theorem \ref{T: single box norm bound simplified} radically simplifies.
\begin{corollary}\label{C: distinct degrees}
    Let $d, \ell, D \in \N$ and $T>0$.  There exist a positive integer $t=O_{d, \ell}(1)$ and a positive real $C=O_{d,\ell,D,T}(1)$ such that the following holds.  Let $N\in\N$ and $\delta>0$ satisfy $N\geq C\delta^{-C}$, let $\bv_1, \ldots, \bv_\ell\in\Z^{{D}}$ be nonzero vectors of size at most $T$,  and let $P_1, \ldots, P_\ell\in\Z[z]$ be polynomials with coefficients of size at most $T$ and degrees
    \begin{align*}
        1\leq \deg P_1 < \cdots < \deg P_\ell =d.
    \end{align*}
     Then for all $1$-bounded functions $f_0, \ldots, f_\ell:\Z^D\to\C$ supported on $[N]^D$, the bound
    \begin{align*}
            \abs{\sum_{\bx\in\Z^D}\E_{z\in[N^{1/d}]}
            f_0(\bx)\cdot f_1(\bx+\bv_1 P_1(z))\cdots f_\ell(\bx+\bv_\ell P_\ell(z))} \geq \delta N^D
    \end{align*}
    implies that
    \begin{align*}
        \norm{f_\ell}_{U^t(\bv_\ell\cdot[\pm N])}^{2^{t}} \geq \frac{1}{C}\delta^{C}N^{D}.
    \end{align*}
\end{corollary}


\begin{example}
    In particular, Corollary \ref{C: distinct degrees} recovers the Gowers-norm-control obtained by Peluse and Prendiville in \cite{PP19}: It shows that if
    \begin{align*}
        \abs{\sum_{x\in\Z}\E_{z\in[N^{1/2}]}f_0(x)f_1(x+z)f_2(x+z^2)} \geq \delta N,
    \end{align*}
    then $\norm{f_2}_{U^t([\pm N])}^{2^t}\gg \delta^{O(1)}N$ for some positive integer $t=O(1)$.
\end{example}
\begin{example}
    Corollary \ref{C: distinct degrees} also recovers the Gowers-norm-control obtained by Peluse, Prendiville and Shao in \cite{PPS24}: It shows that if
    \begin{align*}
        \abs{\sum_{x_1,x_2\in\Z}\E_{z\in[N^{1/2}]}f_0(x_1,x_2)f_1(x_1+z,x_2)f_2(x_1, x_2+z^2)} \geq \delta N,
    \end{align*}
    then $\norm{f_2}_{U^t(\be_2 \cdot [\pm N])}^{2^t}\gg \delta^{O(1)}N^2$ for some positive integer $t=O(1)$.
\end{example}

It is often desirable for technical reasons to consider polynomial progressions with ``$W$-tricked'' polynomials whose coefficients grow (in a controlled way) with the parameter $N$.  We in fact prove the following more general version of Theorem \ref{T: single box norm bound simplified} which is flexible enough to accommodate $W$-tricks.

\begin{theorem}\label{T: single box norm bound}
Let $d, \ell,D \in \N$ and $T>0$. There exist a positive integer $t=O_{d, \ell}(1)$ and a positive real $C=O_{d,\ell,D,T}(1)$ such that the following holds.  Let $K, N, V\in\N$ and $\delta>0$ satisfy $$C\delta^{-C}\leq  K\leq T (N/V)^{1/d},$$  and let $\bp_1, \ldots, \bp_\ell\in\Z^D[z]$ be polynomials of the form $\bp_j(z) = \sum_{i=0}^d \bbeta_{ji}z^i$ with coefficients  at most $V$.\footnote{In the sense that for every $j,i$, each coordinate of $\bbeta_{ji}$ is at most $V$ in absolute value, or equivalently that $\bbeta_{ji}\in[\pm V]^D$.}
     Then for all $1$-bounded functions $f_0, \ldots, f_\ell:\Z^D\to\C$ supported on $[N]^D$, the bound
    \begin{align*}
            \abs{\sum_{\bx\in\Z^D}\E_{z\in[K]}
            f_0(\bx)\cdot f_1(\bx+\p_1(z))\cdots f_\ell(\bx+\p_\ell(z))} \geq \delta N^D
    \end{align*}
    implies that
    \begin{align}\label{E: the box norm 0}
        \norm{f_j}_{\{(\bbeta_{jd_{jj'}}-\bbeta_{j'd_{jj'}})\cdot[\pm K^{d_{jj'}}]:\; j'\in[0,\ell],\; j'\neq j\}^{ t}}^{2^{\ell t}} \geq \frac{1}{C}\delta^{C}N^{D}
    \end{align}
    for all $j\in[0,\ell]$, where $d_{jj'} = \deg (\bp_j - \bp_{j'})$ and $\bp_0 = \mathbf{0}$.
\end{theorem}
 
When $\p_1, \ldots, \p_\ell$ are all $\Z^D$-multiples of a single integer polynomial, Theorem \ref{T: single box norm bound} reduces to the following.
\begin{corollary}\label{C: main theorem, same polynomial}
Let $d, \ell,D \in \N$ and $T>0$.  There exist a positive integer $t = O_{d, \ell}(1)$ and a positive real $C=O_{d,\ell,D,T}(1)$ such that the following holds. Let $\bv_0, \ldots, \bv_\ell\in\Z^{{D}}$ be nonzero vectors of size at most $T$ (with $\bv_0 = \mathbf{0}$), and let $P\in\Z[z]$ be a degree-$d$ polynomial with leading coefficient $\beta_d$. 
    Then for all $N\in\N$ and $1$-bounded functions $f_0, \ldots, f_\ell:\Z^D\to\C$ supported on $[N]^D$, the bound
    \begin{align*}
            \abs{\sum_{\bx\in\Z^D}\E_{z\in[(N/\beta_d)^{1/d}]}
            f_0(\bx)\cdot f_1(\bx+\bv_1 P(z))\cdots f_\ell(\bx+\bv_\ell P(z))} \geq \delta N^D
    \end{align*}
    implies that
    \begin{align*}
        \norm{f_j}_{\{\beta_d (\bv_j-\bv_{j'})\cdot[\pm N/\beta_d]:\; j'\in[0,\ell],\; j'\neq j\}^{ t}}^{2^{\ell t}} \geq \frac{1}{C}\delta^{C}N^{D}
    \end{align*}
    for all $j\in[0,\ell]$.
\end{corollary}
\begin{example}
    If $D=1$ and $\bv_j = j$, Corollary \ref{C: main theorem, same polynomial} tells us that
    \begin{align*}
        \abs{\sum_{x\in\Z}\E_{z\in[(N/\beta_d)^{1/d}]} f_0(x)f_1(x+P(z))\cdots f_\ell(x+\ell P(z))}\geq \delta N
    \end{align*}
    implies that
    \begin{align*}
        \min_{j\in[0,\ell]}\norm{f_j}^{2^t}_{U^t(\beta_d\cdot[\pm N/\beta_d])}\gg_{d, \ell}\delta^{O_{d,\ell}(1)}N
    \end{align*}
    for some positive integer $t=O_{d,\ell}(1)$, where we recall that $\beta_d$ is the leading coefficient of $P$. This improves on the Gowers-norm-control obtained by Prendiville in \cite{Pre17} and recovers the global Gowers-norm-control used by Peluse, Sah and Sawhney in \cite{PSS23}.
\end{example}
\begin{example}
    If $\ell=D=2$ and $\bv_j = \be_j$, Corollary \ref{C: main theorem, same polynomial} tells us that
    \begin{align*}
        \abs{\sum_{x_1,x_2\in\Z}\E_{z\in[(N/\beta_d)^{1/d}]} f_0(x_1, x_2)f_1(x_1+P(z), x_2)f_2(x_1, x_2+P(z))}\geq \delta N^2
    \end{align*}
    implies that
    \begin{align*}
        \norm{f_0}_{(\be_1 \beta_d\cdot[\pm N/\beta_d])^t, (\be_2 \beta_d\cdot[\pm N/\beta_d])^t}^{2t}&\gg_{d}\delta^{O_{d}(1)}N^2\\
        \norm{f_1}_{(\be_1 \beta_d\cdot[\pm N/\beta_d])^t, ((\be_1-\be_2) \beta_d\cdot[\pm N/\beta_d])^t}^{2t}&\gg_{d}\delta^{O_{d}(1)}N^2\\
        \norm{f_2}_{(\be_2 \beta_d\cdot[\pm N/\beta_d])^t, ((\be_2-\be_1) \beta_d\cdot[\pm N/\beta_d])^t}^{2t}&\gg_{d}\delta^{O_{d}(1)}N^2
    \end{align*}
    for some positive integer $t=O(1)$.  This result serves as the starting point for our companion paper \cite{KKL24b}.
\end{example}

Our new concatenation result (the key ingredient in the proof of Theorem \ref{T: single box norm bound}) is the following.
\begin{theorem}\label{T: concatenation of polynomials}
Let $d, r,s,D \in \N$ with $d\leq r$.  There exist a positive integer $t=O_{r, s}(1)$ and a positive real $C=O_{r,s,D}(1)$ such that the following holds.  Let $H, M, N, V\in\N$ and $\delta>0$ satisfy
    \begin{align*}
        C^{-1}\delta^{-C}\leq H\leq M\quad \textrm{and}\quad H^d M V\leq N,
    \end{align*}
    and let $\bc_1, \ldots, \bc_s\in\Z^D[h_1, \ldots, h_r]$ be multilinear polynomials $\bc_j(\uh) = \sum\limits_{\uu\in\{0,1\}^r} \bgamma_{j\uu}\uh^\uu$ of degree at most $d$ and with coefficients at most $V$.  Then for all $1$-bounded functions $f:\Z^D\to\C$ supported on $[N]^D$, the bound
    \begin{align*}
        \E_{\uh\in[\pm H]^r}\norm{f}_{\bc_1(\uh)\cdot[\pm M], \ldots, \bc_s(\uh)\cdot[\pm M]}^{2^s}\geq \delta N^D
    \end{align*}
        implies that
    \begin{align*}
        \norm{f}_{E_1^{ t}, \ldots, E_s^{ t}}^{2^{st}}\geq \frac{1}{C}\delta^{C}N^{D},
    \end{align*}
    where
    \begin{align*}
        E_j = \sum_{\uu\in\{0,1\}^d}\bgamma_{j\uu}\cdot [\pm H^{|\uu|+1}].
    \end{align*} 
\end{theorem}

\subsection{Further applications}
The original motivation for this work was obtaining quantitative bounds for sets avoiding polynomial corners (see the companion paper \cite{KKL24b}).  In the time since this paper appeared as a preprint, however, it has already found rather unexpected applications to other problems in number theory, ergodic theory, and harmonic analysis. First, Green and Sawhney~\cite{GS25} used our concatenation results to obtain an asymptotic for the number of pairs of primes $(p,q)$ such that $p^2 + nq^2$ is also prime (for any fixed $n \equiv 0,4 \pmod 6$). In the course of their work, they extended our arguments from Section \ref{S: general concatenation} to a variant of box norms defined with general probability measures (instead of multisets) on $\Z^D$. Second, Donoso, Koutsogiannis, Sun, Tsinas, and the second author~\cite{DKKST25, DKKST24} adapted our concatenation argument to an ergodic-theoretic setting and used it to obtain new seminorm estimates, limiting formulas, and joint ergodicity and recurrence results for multiple ergodic averages of commuting transformations along Hardy sequences of polynomial growth. Lastly, Kosz, Mirek, Peluse, and Wright~\cite{KMPW24} applied our result in their work on establishing pointwise almost-everywhere convergence of multiple ergodic averages along distinct-degree polynomials.

\subsection{Acknowledgments} NK was supported in part by the NSF Graduate Research Fellowship Program under grant DGE–203965. BK is supported by the NCN Polonez Bis 3 grant No. 2022/47/P/ST1/00854 (H2020 MSCA GA No. 945339). For the initial stages of the project, BK was supported by the ELIDEK grant No: 1684.  JL is suppored by the NSF Graduate Research Fellowship Grant No. DGE-2034835. 

For the purpose of Open Access, the authors have applied a CC-BY public copyright licence to any Author Accepted Manuscript (AAM) version arising from this submission.

We would like to thank Sarah Peluse, Mehtaab Sawhney, and Terry Tao for helpful discussions and for comments on a draft of this paper. We also extend our gratitude to the anonymous referees for thorough feedback.

\section{Proof strategy}\label{S: strategy}
Because the proofs of our main results are rather technical, we will begin with an informal discussion of how our techniques apply to a few simplified examples.

The proof of Theorem \ref{T: single box norm bound} begins with PET induction, a complexity-reduction argument that allows us to control the counting operator of a polynomial progression by an average of box norms.  While PET arguments are now standard and have been carried out in many places (see, e.g., \cite{Ber87, BL96, DFKS22, Pel20, Pre17}), we need to keep careful track of the structure of the polynomial differencing parameters $\bc_j$ that are output by PET.  In general, the input of PET is the counting operator lower bound
\begin{align*}
    {\sum_{\bx\in\Z^D}\E_{z\in[N^{1/d}]} f_0(\bx)\cdot f_1(\bx+\p_1(z))\cdots f_\ell(\bx+\p_\ell(z))} \geq \delta N^D,
\end{align*}
for some polynomials $\p_1, \ldots, \p_\ell\in\Z^D[z]$ of degree at most $d$ and some $1$-bounded functions $f_0, \ldots, f_\ell$ supported on $[N]^D$.  The output of PET is then that the average of box norms 
\begin{align}\label{E: average of box norms of polynomials}
	 \E_{\uh\in[\pm N^{1/d}]^{r}}\norm{f}_{{\bc_1(\uh)\cdot[\pm N^{1/d}]}, \ldots, {\bc_s(\uh)\cdot[\pm N^{1/d}]}}^{2^s} \gg \delta^{O(1)} N^D
\end{align}
is also large. (Here and elsewhere, we allow all $O(1)$ terms to depend on $d,\ell$.)  Two crucial properties of this output are that the polynomials $\bc_1, \ldots, \bc_s\in\Z^D[h_1, \ldots, h_r]$ appearing above are multilinear and that their coefficients are $O(1)$-multiples of the leading coefficients of $\p_1, \ldots, \p_\ell$ and of their pairwise differences.  This PET argument is carried out in detail in Section \ref{S: PET}.

Sections \ref{S: model}--\ref{S: concatenation along polys} are then devoted to the proof of Theorem \ref{T: concatenation of polynomials}, our main concatenation result, which we will focus on in the remainder of this sketch.  To differentiate between the ranges of the various parameters, let us rewrite \eqref{E: average of box norms of polynomials} as
\begin{align}\label{E: PET output}
    \E_{\uh\in[\pm H]^{r}}\norm{f}_{{\bc_1(\uh)\cdot[\pm M]}, \ldots, {\bc_s(\uh)\cdot[\pm M]}}^{2^s} \gg \delta N^D,
\end{align}
where we have set $H:=M:=N^{1/d}$ and replaced $\delta^{O(1)}$ with $\delta$.  (It will later be convenient to be able to consider slightly different values of $H,M$.)

We begin with the very simplest example, namely, where $d=2$, $D=s=r=1$, and $\bc_1(h)=h$, so that our hypothesis is
$$\E_{h \in [\pm H]} \norm{f}^2_{h \cdot [\pm M]} \gg \delta N.$$
Ignoring smoothing terms, we can expand this as
$$\E_{h \in [ \pm H]} \E_{m \in [\pm M]} \sum_{x \in \mathbb{Z}} f(x)\overline{f(x+hm)}.$$
If the quantity $hm$ ranged over $[\pm HM]$ uniformly as $h,m$ ranged, then we could replace $hm$ by a new variable $t$ and conclude that $\norm{f}_{[\pm HM]}^2 \gg \delta N$.  This is of course not the case, but our goal will be to reach a similar situation.  Exchanging the order of summation and applying the Cauchy--Schwarz inequality twice, we obtain
$$\E_{h_1,h_2,h_3,h_4 \in [\pm H]}\E_{m_1,m_2,m_3,m_4 \in [\pm M]} \sum_{x \in \Z}f(x) \overline{f(x+h_1m_1+h_2m_2+h_3m_3+h_4m_4)} \gg \delta^4 N.$$
An elementary number-theoretic argument (see Corollary \ref{C: linear congruences} below) tells us that the quantity $h_1m_1+h_2m_2+h_3m_3+h_4m_4$ assumes each value $t \in [\pm 4HM]$ at most $O(H^3M^3)$ times as $h_1,h_2,h_3,h_4,m_1,m_2,m_3,m_4$ range; in other words, this quantity is nearly uniformly distributed, in the sense of being very anti-concentrated.  The idea behind the lemma is that for most fixed choices of the $h_i$'s, the expression is well-distributed as the $m_i$'s range; more precisely, this is the case unless $\gcd(h_1,h_2,h_3,h_4)$ is big or all of the $h_i$'s are small, and there are very few such ``bad'' choices.  
Together with the triangle inequality, this fact implies (after some work) that
$$\E_{t \in [\pm 4HM
]}\abs{\sum_{x \in \Z} f(x)\overline{f(x+t)}} \gg \delta^{O(1)} N.$$
Using the Cauchy--Schwarz inequality to remove the absolute value gives
$$\E_{t \in [\pm 4HM
]}\sum_{x,s \in \Z} f(x)\overline{f(x+t)}\overline{f(x+s)}f(x+s+t) \gg \delta^{O(1)} N^2.$$
Since $f$ is supported on $[N]$, we can restrict the $s$-variable to $[-N,N]$ (say), and renormalizing gives
$$\norm{f}^4_{[\pm 4HM],[\pm N]}=\E_{t \in [\pm 4HM]}\E_{s \in [\pm N]}\sum_{x \in \Z}f(x)\overline{f(x+t)}\overline{f(x+s)}f(x+s+t) \gg \delta^{O(1)} N.$$
So we have successfully concatenated the differencing sets $[\pm H], [\pm M]$ into a single differencing set $[\pm 4HM]$, at the cost of introducing an additional differencing over $[\pm N]$.  Using monotonicity properties of box norms (see Lemma \ref{L: properties of box norms} below), we can ``trim'' $[\pm 4HM],[\pm N]$ and replace each of them with a copy of $[\pm HM]$.

For our next example, take $d=3$, $r=2$, $D=s=1$, and $\bc_1(h_1,h_2)=h_1h_2$, so that our hypothesis is
$$\E_{h_1,h_2 \in [\pm H]} \norm{f}^2_{h_1h_2 \cdot [\pm M]} \gg \delta N.$$
Expanding the left-hand side and using the Cauchy--Schwarz inequality to duplicate the differencing parameters, as in the first example, gives
\begin{multline*}
\E_{h_{11},\ldots h_{14} \in [\pm H]}\E_{h_{21},\ldots h_{24} \in [\pm H]}\E_{m_1,\ldots,m_4 \in [\pm M]}\sum_{x \in \mathbb{Z}}\\
f(x)\overline{f(x+h_{11}h_{21}m_1+h_{12}h_{22}m_2+h_{13}h_{23}m_3+h_{14}h_{24}m_1)} \gg \delta^4 N.
\end{multline*}
If the quantity $h_{11}h_{21}m_1+h_{12}h_{22}m_2+h_{13}h_{23}m_3+h_{14}h_{24}m_1$ were nearly uniformly distributed on its range $[\pm H^2M]$, then we could replace it by a single new variable ranging over $[\pm H^2M]$ and conclude as in the first example.  But this quantity fails to be sufficiently anti-concentrated, essentially because it is too often divisible by squares of small primes.  The solution to this problem is to apply the Cauchy--Schwarz inequality several more times in order to replace each term $h_{2j}m_j$ with the expression
$$h_{2j1}m_{j1}+h_{2j2}m_{j2}+h_{2j3}m_{j3}+h_{2j4}m_{j4}$$
(at the cost of degrading the power of $\delta$, which we do not care about).  As the variables range, this expression is nearly uniformly distributed on $[\pm 4HM]$ (as in the previous example), and so with an acceptably small loss we can replace it by a new variable $t_j$ that ranges over $[\pm 4HM]$.  In turn, the expression
$$h_{11}t_1+h_{12}t_2+h_{13}t_3+h_{14}t_4$$
is nearly uniformly distributed on $[\pm 16 H^2M]$, and we conclude as in the first example.

This strategy of constructing large ``expanding'' polynomials with nested equidistribution properties suffices to handle the general case where $s=1$ and $\bc_1$ is a monomial.  When $\bc_1$ is not a monomial, we must show that several multilinear expressions are simultaneously nearly uniformly distributed.  For example, with $d=3$, $r=D=2$, $s=1$, and $\bc_1(h_1,h_2)=h_1h_2 \be_1+h_1 \be_2=(h_1h_2,h_1)$, our hypothesis
$$\E_{h_1,h_2 \in [\pm H]} \norm{f}^2_{(h_1h_2 \be_1+h_1 \be_2) \cdot [\pm M]} \gg \delta N$$
leads to the bound
\begin{multline}\label{eq:multiscale-example}
\E_{\substack{h_{i} \in [\pm H]:\\ i \in [4]}}\; \E_{\substack{h_{ij}\in [\pm H]:\\ i,j \in [4]}}\;\E_{\substack{m_{ij} \in [\pm M]:\\ i,j \in [4]}}\;\sum_{\bx \in \mathbb{Z}^2}\\
f(\bx) \overline{f\Bigbrac{\bx+\Bigbrac{\sum_{i \in [4]}h_i \sum_{j \in [4]}h_{ij}m_{ij} }\be_1+\Bigbrac{\sum_{i \in [4]}h_i \sum_{j \in [4]}m_{ij}}\be_2}}\gg \delta^{O(1)} N^2
\end{multline}
after the usual Cauchy--Schwarz maneuvers, so we wish to show that the $\be_1$- and $\be_2$-coefficients are simultaneously equidistributed.  The first step is showing that for each $i \in [4]$, the expressions
$$\sum_{j \in [4]}h_{ij}m_{ij}, \quad \sum_{j \in [4]}m_{ij}$$
are simultaneously equidistributed.  Dropping the $i$'s from the indices, we need to show that for each pair $(n_0,n_1) \in [\pm 4M]\times [\pm 4HM]$, the system of equations
$$ m_1+m_2+m_3+m_4=n_0,\quad h_1m_1+h_2m_2+h_3m_3+h_4m_4=n_1$$
has at most $O(H^3M^2)$ solutions.  (Here $H^3M^2=H^4M^4/(H M \cdot M)$.)  The new idea here is that we can solve for $m_1$ in the second equation and make a substitution into the first equation, which (after regrouping) yields
$$(h_2-h_1)m_2+(h_3-h_1)m_3+(h_4-h_1)m_4=n_1-h_1n_0.$$
Our previous arguments imply that for each of the $O(H)$ choices for $h_1$, this equation has at most $O(H^2M^2)$ solutions $(h_2,h_3,h_4,m_2,m_4,m_4)$.  Such a solution can be extended to at most one choice of $m_1$ (sometimes there are no choices for $m_1$ because we constrained it to lie in $[\pm M]$), so in total we have at most $O(H^3M^2)$ solutions, as desired.

Using the first step, we replace $\sum_{j \in [4]}h_{ij}m_{ij}, \sum_{j \in [4]}m_{ij}$ with new variables $t_i,t'_i$ ranging over $[\pm 4HM], [\pm 4M]$, respectively, and it remains to show that the expressions
$$\sum_{i \in [4]} h_i t_i, \quad \sum_{i \in [4]} h_i t'_i$$
are simultaneously equidistributed.  These expressions are individually identical to the expressions that we discussed above, but they are coupled in the sense that the variables $h_i$ appear in both expressions.  It turns out that the few ``bad'' choices of $h_1,h_2,h_3,h_4$ are now too costly because they lead to concentration in both expressions.  The solution to this problem is to eliminate the bad choices before counting solutions: We can throw out the bad tuples $(h_1,h_2,h_3,h_4)$ in \eqref{eq:multiscale-example} at the cost of degrading the implicit constant in the inequality.

To apply this type of multiscale analysis to multilinear expressions whose degrees differ by more than $1$, we introduce dummy expressions at the intermediate degrees and then run the above argument for pairs of expressions with consecutive degrees.  The genericity assumption on the $h$'s is an essential technical element.  The details, which are quite delicate, are carried out in Section \ref{S: equidistribution}.

Further difficulties arise for $s>1$.  In the $s=1$ case, we iteratively used the Cauchy--Schwarz inequality to double certain variables.  When $s>1$, this maneuver also increases the number of multilinear expressions that must be simultaneously equidistributed---which in turn potentially requires further applications of the Cauchy--Schwarz inequality to replace these new expressions by ones with good equidistribution properties.  Using arguments originally developed by the second author \cite{Kuc23} for concatenation in the finite-field setting, we show in Section \ref{S: general concatenation} that this process can be satisfactorily controlled.  This part of our proof uses no information about the ``polynomial'' nature of our box-norm parameters, and it ends up being convenient to carry out the argument in the more general setting of averages of box norms
\begin{align}\label{E: concatenation input}
    \E_{i\in I}\norm{f}_{H_{1i}, \ldots, H_{si}}^{2^s},
\end{align}
where $H_{1i}, \ldots, H_{si}$ for $i \in I$ are arbitrary finite multisets of a countable abelian group $G$ that are contained in a bounded dilate of the support of $f$.  We also believe that this level of generality might be useful for future extensions of our concatenation argument to other settings. 

The key idea for handling \eqref{E: concatenation input}, elucidated in Lemma \ref{L: concatenation lemma}, is to use the Cauchy--Schwarz and Gowers--Cauchy--Schwarz inequalities in order to control \eqref{E: concatenation input} by
\begin{align*}
    \E_{i_0, i_1\in I}\norm{f}_{H_{1i_0}-H_{1i_1}, H_{2i_0}, H_{2i_1} \ldots, H_{si_0}, H_{si_1}}^{2^{2s-1}};
\end{align*}
we see that the multiset $H_{1i}$ got replaced by a potentially much better distributed multiset $H_{1i_0}-H_{1i_1}$ at the cost of doubling the number of remaining multisets. Iterating this procedure many times, we can control \eqref{E: concatenation input} by an analogous expression in which each multiset $H_{ji}$ is replaced by many iterated multisets of the form $H_{ji_1} + \cdots + H_{ji_\ell}$. In the first example discussed above, this process corresponds to replacing $hm$ by $h_1m_1+h_2m_2+h_3m_3+h_4m_4$, and we have already established that the latter expression enjoys better equidistribution properties.

The last step of our argument, described in Section \ref{S: concatenation along polys}, specializes this general concatenation result to our ``polynomial'' setting and shows how it can be massaged to lead to the types of systems of multilinear equations described above. The results from Section \ref{S: general concatenation} allow us to replace the boxes $\bc_j(\uh)\cdot[\pm M]$ in \eqref{E: PET output} by iterated sums
\begin{align}\label{E: expanded box}
    \sum_{i_1, \ldots, i_r = 1}^\ell \bc_j(h_{1i_1}\cdots h_{ri_1\cdots i_r})\cdot[\pm M].
\end{align}
Each such box generates a system of multilinear equations, corresponding to the monomials in the polynomial $\bc_j$. The arguments from Section \ref{S: equidistribution} then give that if $\bc_j(\uh)= \sum\limits_{\uu\in\{0,1\}^r} \bgamma_{j\uu}\uh^\uu$, then \eqref{E: expanded box} is uniformly distributed in the box $\sum\limits_{\uu\in\{0,1\}^d}\bgamma_{j\uu}\cdot [\pm H^{|\uu|+1}]$. This concludes the proof of Theorem \ref{T: concatenation of polynomials}. 

In Section \ref{S: proof of main thm}, we combine the conclusions of Theorem \ref{T: concatenation of polynomials} with the output of PET to prove Theorem \ref{T: single box norm bound}. In doing so, we use the important observation from PET that the coefficients of the polynomials $\bc_j$ are $O(1)$-multiples of the leading coefficients of the polynomials $\bp_1, \ldots, \bp_\ell$ and their pairwise differences.

Figure \ref{fig:flowchart} provides a visual overview of the main steps of our argument.

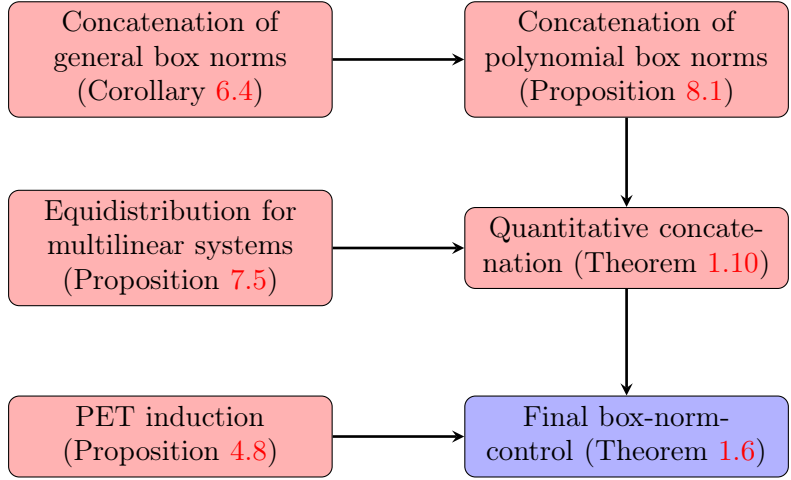
\begin{figure}[ht]
  \centering
\begin{tikzpicture}[mynode/.style={draw, startstop, text width=4cm, align=center}]
    \node[draw, mynode] (A) at (0,0) {Concatenation of general box norms (Corollary \ref{C: iterated concatenation for general groups})};
    \node[draw, mynode] (B) at (6,0) {Concatenation of polynomial box norms (Proposition \ref{P: concatenation of polynomials I})};
    \node[draw, mynode] (C) at (0,-2.5) {Equidistribution for multilinear systems (Proposition \ref{P: systems of multilinear equations})};
    \node[draw, mynode] (D) at (6,-2.5) {Quantitative concatenation (Theorem \ref{T: concatenation of polynomials})};
    \node[draw, mynode] (E) at (0,-5) {PET induction (Proposition \ref{P: PET})};
    \node[draw, mynode, fill=blue!30] (F) at (6,-5) {Final box-norm-control (Theorem \ref{T: single box norm bound})};
    
    \draw[->, >=stealth, line width=1pt] (A) -- (B);
    \draw[->, >=stealth, line width=1pt] (B) -- (D);
    \draw[->, >=stealth, line width=1pt] (C) -- (D);
    \draw[->, >=stealth, line width=1pt] (D) -- (F);
    \draw[->, >=stealth, line width=1pt] (E) -- (F);
\end{tikzpicture}

    \caption{The main steps and logical implications in the proof of Theorem \ref{T: single box norm bound}.}\label{fig:flowchart}
\end{figure}

\subsection{Comparison with previous approaches}

Our argument deviates in several ways from previous concatenation arguments in the literature: the quantitative concatenation arguments of Peluse \cite{Pel20} and Peluse and Prendiville \cite{PP19} for $1$-dimensional polynomial progressions over $\Z$ (see also Prendiville' exposition \cite{Pre20b}); the second author's quantitative concatenation arguments for multidimensional polynomial progressions over finite fields \cite{Kuc23}; and the original qualitative argument of Tao and Ziegler \cite{TZ16}, together with its later development in the context of ergodic theory (e.g., \cite{DFKS22}).

One difference between our argument and the approach of Peluse and Prendiville \cite{Pel20, PP19, Pre20b} lies in the order of carrying out two major maneuvers, namely, duplicating variables using the Cauchy--Schwarz inequality, and replacing a multilinear polynomial like $h_1 m_1 + h_2 m_2 + h_3 m_3 + h_4 m_4$ by a new variable using equidistribution properties. As described in the sketch above, our argument clearly separates these two operations: First, we apply the Cauchy--Schwarz inequality many times to convert the original average \eqref{E: PET output} into one where each arithmetic progression $\bc_j(\uh)\cdot[\pm M]$ can be replaced by a generalized arithmetic progression $\bc_j(\uh_1)\cdot[\pm M] + \cdots + \bc_j(\uh_\ell)\cdot[\pm M]$ (interpreted as a multiset). Only then do we show that the massive system of multilinear forms arising out of this procedure is equidistributed in its range. In \cite{Pel20, PP19, Pre20b}, these two steps are intertwined: Applications of the Cauchy--Schwarz inequality are followed by equidistribution results for intermediate systems of bilinear forms, then more applications of the Cauchy--Schwarz inequality, and so on.  While the approach of Peluse and Prendiville has the advantage of using ``simpler'' equidistribution inputs (only systems of bilinear forms appear), our approach has the advantages of clearly separating the aforementioned two steps. The first step outlined above culminates in a quantitative concatenation result for averages of arbitrary box norms in $\Z^D$ which more closely resembles the ``Bessel inequalities'' of  Tao and Ziegler \cite{TZ16} (see also its iterated version in \cite[Corollary 3.7]{DFKS22}). Say that a multiset is \textit{symmetric} if the multiplicity of $\bx$ equals the multiplicity of $-\bx$ for all $\bx$.
\begin{restatable}[Iterated concatenation of box norms]{corollary}{iteratedconcatenationofboxnorms}\label{C: iterated concatenation for general groups}
    Let $A, D, s,N\in\N$, and let $\ell$ be a power of 2. There exist a positive integer $t=O_{\ell,s}(1)$ and positive reals $C=O_{\ell,s,D}(1)$, $C'=O_{\ell,s}(1)$  such that the following holds. Let $I$ be a finite indexing set, and let $H_{ji}\subseteq [\pm AN]^D$ be a symmetric
   multiset for each $i\in I$ and $j\in[s]$. Then for all $1$-bounded functions $f:\Z^D\to\C$ supported on $[N]^D$, the bound
    \begin{align*}
        \E_{i\in I}\norm{f}_{H_{1i}, \ldots, H_{si}}^{2^{s}}\geq \delta N^D
    \end{align*}
    implies that
    \begin{align}\label{E: iterated concatenation}
        \E_{\substack{k_1, \ldots, k_t\in I}}\norm{f}_{\{H_{j k_{i_1}}+\cdots + H_{j k_{i_{\ell}}}:\ j\in[s],\ 1\leq i_1<\cdots < i_{\ell}\leq t\}}^{2^d}\geq \frac{1}{C} A^{-C'D} \delta^C N^D,
    \end{align}
    where $d$ is the degree of the box norms on the left-hand side. 
\end{restatable}

Another important point of difference is the choice of induction scheme.  In \cite{Pel20}, Peluse begins the induction with degree-$2$ box norms as the base case (see \cite[Lemma 5.5]{Pel20}), and she treats degree-$1$ box norms separately (\cite[Lemma 5.4]{Pel20}).  The strategy in \cite{PP19, Pre20b} is organized the same way. 
Here, in order to concatenate averages of degree-$2$ box norms, Peluse and Prendiville apply an inverse theorem for $2$-dimensional box norms and then remove the resulting structured functions using their periodicity properties. For us, degree-$1$ box norms serve as the base case, which we analyze using a simple application of the Cauchy--Schwarz inequality. Consequently, our approach is more elementary and requires neither inverse theory nor periodicity properties.  This strategy first appeared in the aforementioned work of the second author~\cite{Kuc23} over finite fields.

Furthermore, dealing with multidimensional progressions forces us to keep careful track---more so than in \cite{Pel20}---of how polynomial coefficients change during the PET procedure.  Compared with the approach of \cite{Pel20}, our coefficient-tracking scheme\footnote{Our scheme is strongly inspired by a similar scheme from \cite{DFKS22}, but it improves on the latter in that we are able to identify fewer (and more relevant) properties that need to be tracked throughout PET.} records more information about lower-degree terms of polynomials. It is precisely this extra information that allows Theorems \ref{T: single box norm bound simplified} and \ref{T: single box norm bound} to extend Peluse's result even in the single-dimensional setting (see Example \ref{Ex: Peluse}).

Lastly, whereas Peluse and Prendiville use basic Fourier analysis on several occasions, our argument is entirely carried out on the physical side. In particular, we do not use Fourier analysis in replacing equidistributed expressions like $h_1 m_1 + h_2 m_2 + h_3 m_3 + h_4 m_4$ by a new variable.

The relationship between our argument and the recent quantitative concatenation result over the finite fields by the second author \cite{Kuc23} is more complex. As already mentioned, our concatenation argument consists of two big steps: handling averages of box norms \eqref{E: concatenation input} along arbitrary multisets by means of the Cauchy--Schwarz inequality, and establishing equidistribution results for systems of multilinear equations. The first of these steps is largely based on an analogous argument from \cite{Kuc23}, and its adaptation to the integer setting requires only minor technical modifications. The two approaches deviate in the second step, in which we replace each monomial $h_1 \cdots h_r m$ by a new variable.  This task is much simpler in the finite-field setting because the  monomial expression in question is already very uniformly distributed; the integer setting requires several new ideas.

The differences between quantitative concatenation arguments (whether ours or those from \cite{Kuc23, Pel19, PP20}) and the results of Tao and Ziegler \cite{TZ16} run more deeply. Our main goal is to obtain an effective inequality between two averages of box norms. By contrast, Tao and Ziegler's main result compares different algebras of structured functions naturally related to various box norms. Specifically, they show that the intersection of the algebras of dual functions related to two different box norms can be embedded into a single algebra of dual functions arising from some (higher-degree) box norm.  This leads to a (completely ineffective) ``comparison'' result between averages of box norms.  Tao and Ziegler make ample use of the Gowers--Cauchy--Schwarz inequality but also employ infinitary arguments that are hard to translate into the quantitative, finitary world.
Thus, even though they obtain better control on the degree of the ``concatenated'' box norms, their conclusion is purely qualitative and hence unsuitable for quantitative applications.


\section{Definitions and preliminaries}\label{S: definitions}
\subsection{Basic notation}
We write $\N, \N_0,\Z, \R, \C$ for the sets of positive integers, nonnegative integers, integers, real numbers, and complex numbers (respectively). For reals $a<b$, we set $[a,b] = \{\ceil{a}, a+1, \ldots, \floor{b}\}$, abbreviating $[1,N]$ as $[N]$ and $[-N,N]$ as $[\pm N]$.

To facilitate the reading of the paper, we fix the uses of various letters. The letter $d$ denotes the maximum degree of polynomials under consideration; $D$ is the dimension of the ambient space $\Z^D$ in which we work; $K$ is the length of the interval from which the polynomial variable $z$ is taken; $\ell$ is the length of the polynomial family, and later on it also denotes the number of summands in concatenated box norms; $N$ refers to the {size of the} support of the functions; $r$ is the number of parameters $h_1, \ldots, h_r$, and $H$ concerns their range; $s,t$ denote the degrees of box norms appearing in the arguments; and $V$ bounds the size of polynomial coefficients.

We write $\x = (x_1, \ldots, x_D)$ for an element of $\Z^D$ (for some fixed $D\in\N$) and $x$ for an element of $\Z$. When the context is clear, we often write a sum over $\Z^D$ or $\Z$ as $\sum_{\bx}$ or $\sum_x$ instead of $\sum_{\bx \in \Z^D}$ or $\sum_{x \in \Z}$.

We usually write $\uh=(h_1, \ldots, h_r)$ for a tuple of length $r$, and we denote its $\ell^1$-norm by $|\uh|:=|h_1|+\cdots + |h_r|$. For $\uu\in\N_0^r$, we let $\supp(\uu):=\{i\in[r]:\; u_i>0\}$ and get the monomial $\uh^\uu:= h_1^{u_1}\cdots h_r^{u_r}$. Given $\uh,\uh'\in\Z^r$ and $\ueps\in\{0,1\}^r$, we also set $\uh^\ueps: =(h_1^{\epsilon_1}, \ldots, h_r^{\epsilon_r})$, where $h_i^{\epsilon_i}: = h_i$ if $\epsilon_i = 0$ and $h_i^{\epsilon_i}: = h_i'$ if $\epsilon_i = 1$.
{While there is a potential for confusion here, the meaning of the notation should be clear from the context.}

We frequently work with multisets. 
\begin{definition}[Multisets]
For a finite multiset $X$, we write $|X|$ for the total number of elements in $X$, counted with multiplicity.  We use $\supp(X)$ to denote the set of elements appearing in $X$ with strictly positive multiplicity.  With regard to containment of multisets, we write $X \subseteq Y$ if the set $\supp(X)$ is contained in the set $\supp(Y)$, and we denote $X \leq Y$ if each element of $ \supp(X)$ has multiplicity at least as large in $Y$ as in $X$. Lastly, for multisets $X,Y$, we define the sumset $X+Y$ to be the Minkowski sum of $X,Y$ with multiplicities: The multiplicity of the element $z$ in $X+Y$ is $|\{(x,y)\in X\times Y:\; z = x+y\}|$.  Equivalently, the weight function of $X+Y$ is the convolution of the weight functions of $X,Y$.  Note that $|X+Y|=|X| \cdot |Y|$.  The difference set $X-Y$ is defined analogously.     
\end{definition}

\begin{definition}[Indicator function]
    We use $\mathbf{1}(E)$ to denote the indicator function of an event $E$.  For instance, $\mathbf{1}(a + b = c)$ is the function (of $a,b,c$) that equals $1$ when $a+b=c$ and equals $0$ when $a+b\neq c$. For a multiset $E$, we also abbreviate $\mathbf{1}(x\in \supp(E))$ by $\mathbf{1}_E(x)$.    
\end{definition}

Let $X$ be a nonempty multiset. If $X$ is finite, we let $\E_{x\in X} := \frac{1}{|X|}\sum_{x\in X}$ be the average over $X$ (with elements counted with multiplicity).

We call a function $f:X\to \C$ \textit{$1$-bounded} if $\norm{f}_\infty:=\sup\limits_{x\in X}|f(x)|\leq 1$. 

We use the following standard asymptotic notation. If $f,g: X\to\C$, with $g$ taking positive real values, we denote $f=O(g)$, $f\ll g$, $g\gg f$, or $g = \Omega(f)$ if there exists $C>0$ such that $|f(x)|\leq C g(x)$ for all $x\in X$.  If the constant $C$ depends on a parameter, we record this dependence with a subscript. Lastly, we write $f\asymp g$ if $f\ll g$ and $f\gg g$.

We let $\CC: \C \to \C$ denote the complex conjugation operator $\CC z=\overline{z}$.

\subsection{Van der Corput inequalities}

If $G$ is a countable abelian group and $E\subseteq G$ is a finite {and nonempty} multiset, we define the \textit{Fej\'er kernel}
\begin{align*}
    \mu_E(h) := \E_{h_1, h_2\in E}\mathbf{1}({h = h_1 - h_2}),
\end{align*}
and also $\mu_E(h_1, \ldots, h_r) = \mu_E(h_1)\cdots \mu_E(h_r)$.
Summing over $h$, we easily observe that $\sum_h \mu_E(h) = 1$. Moreover, if there is some $C>0$ such that every element of $E$ has multiplicity at most $C\frac{|E|}{|\supp(E)|}$, 
then we also have the pointwise bound $0\leq \mu_E(h)\leq \mathbf{1}_{E-E}(h)\cdot\frac{C}{|\supp(E)|}$.\footnote{The lower bound is obvious; for the upper bound, write
\begin{align*}
    \mu_E(h) = \mathbf{1}_{E-E}(h)\cdot\frac{1}{|E|^2}\sum_{h_1\in E}\sum_{h_2\in E}\mathbf{1}(h_2 = h_1 + h),
\end{align*}
and notice that the sum over $h_2$ is bounded by the maximum multiplicity of an element of $E$.} 
When $E = [\pm H]$, we use the shorthand notation
\begin{align*}
    \mu_H(h):=\mu_{[\pm H]}(h) = \E_{h_1, h_2\in[\pm H]} \mathbf{1}_{h = h_1 - h_2} = \frac{1}{2H+1}\brac{1-\frac{|h|}{2H+1}}_+.
\end{align*}
This definition differs slightly from those in \cite{Pel20, PP19, PP20, Pre17} in that we take the ranges of $h_1, h_2$ to be $[\pm H]$ rather than $[H]$. Making the range of $h_1, h_2$  symmetric around the origin will be convenient for us later.

We  record several versions of the standard van der Corput inequality.
The proof of the following basic version is standard (see, e.g., \cite[Lemma 3.2]{Pre17}).

\begin{lemma}[van der Corput inequality, version 1]\label{L: vdC 0}
    Let $f:\Z\to\C$ be $1$-bounded, and let $H, K\in\N$. Then: 
    \begin{enumerate}
        \item (Symmetric van der Corput inequality)
    \begin{align*}
        \abs{\E_{z\in[K]}f(z)}^2 \leq \frac{K+2H}{K}\E_{h, h'\in[\pm H]}\frac{1}{K}\sum_{z\in([K]-h)\cap([K]-h')}f(z + h)\overline{f(z+h')}.
    \end{align*}
        \item (Asymmetric van der Corput inequality)
            \begin{align*}
        \abs{\E_{z\in[K]}f(z)}^2 \leq \frac{K+2H}{K}\sum_{h}\mu_H(h)\frac{1}{K}\sum_{z\in[K]\cap([K]-h)}f(z)\overline{f(z+h)}.
    \end{align*}
    \end{enumerate}
\end{lemma}
    Of course, it does not make sense to take $H\geq K$ since $[K]\cap([K]-h)$ is empty whenever $|h|\geq K$.
    
In applications, we will use the following corollary of Lemma \ref{L: vdC 0}.
\begin{lemma}[van der Corput inequality, version 2]\label{L: vdC}
    Let $0<\delta\leq 1$ and $H, K\in\N$ satisfy $H\leq \frac{\delta^2}{4}K$. If $f:\Z\to\C$ is a $1$-bounded function satisfying 
    \begin{align}\label{E: assumption in vdC}
        \abs{\E_{z\in[K]}f(z)}\geq \delta,
    \end{align}
    then
    \begin{align*}
        \E_{h,h'\in[\pm H]}\E_{z\in[K]}f(z+h)\overline{f(z+h')}\geq \frac{\delta^2}{4}
    \end{align*}
    and
    \begin{align*}
        {\sum_{h}\mu_H(h)\E_{z\in[K]}f(z)\overline{f(z+h)}}\geq \frac{\delta^2}{4}.
    \end{align*}
\end{lemma}
\begin{proof}
We prove only the second inequality since the proof of the first is almost identical.
    Since $K - |[K]\cap([K]-h)|\leq |h|\leq H\leq \frac{\delta^2}{4}K$, we have
    \begin{align*}
    \abs{\sum_{h}\mu_H(h)\sum_{z\in[K]}f(z)\overline{f(z+h)} - \sum_{h}\mu_H(h)\sum_{z\in[K]\cap([K]-h)}f(z)\overline{f(z+h)}}\leq \frac{\delta^2}{4}K. 
    \end{align*}
    It then follows Lemma \ref{L: vdC 0} and the assumption \eqref{E: assumption in vdC} that
    \begin{align*}
        {\sum_{h}\mu_H(h)\E_{z\in[K]}f(z)\overline{f(z+h)}}\geq \frac{K}{K+2H}\abs{\E_{z\in[K]}f(z)}^2 - \frac{\delta^2}{4}\geq \frac{\delta^2}{4},
    \end{align*}
    as desired.
\end{proof}

\subsection{Box norms}
In this section, we summarize
the theory of box norms on an arbitrary countable\footnote{{If the group $G$ were uncountable, then one would typically prefer to replace the counting measure with some more natural choice, which could raise issues of measurability. Our interest in concatenation comes from $\Z^D$, so we restrict our attention to countable groups in order not to get sidetracked by such considerations.}} abelian group $G$; throughout, we write $\sum_x:=\sum_{x\in G}$. Given $h, h'\in G$ and a finitely supported function $f: G\to\C$, we define
the (multiplicative) discrete derivatives
$\Delta_h f(x) := f(x)\overline{f(x+h)}$ and $\Delta'_{(h,h')} f(x) := f(x+h)\overline{f(x+h')}$.  Similarly, compositions of such discrete derivatives are given by
\begin{align}
    \label{E: asymmetric derivative} \Delta_{h_1, \ldots, h_s}f(x) &= \Delta_{h_1}\cdots \Delta_{h_s}f(x) = \prod_{\ueps\in\{0,1\}^s}\CC^{|\ueps|}f(x+\ueps\cdot \uh),\\
    \label{E: symmetric derivative}\Delta'_{(h_1, h_1'), \ldots, (h_s,h_s')}f(x) &= \Delta'_{(h_1, h_1')}\cdots\Delta'_{(h_s, h_s')}f(x) = \prod_{\ueps\in\{0,1\}^s}\CC^{|\ueps|}f(x+(\underline{1}-\ueps)\cdot\uh + \ueps\cdot\uh'),
\end{align}
where again $\CC z = \overline{z}$ for $z\in\C$.
We will later also use the notation
\begin{align*}
    \Delta'_{\alpha_1,\ldots, \alpha_s; (\uh, \uh')} := \Delta'_{(\alpha_1 h_1, \alpha_1 h_1'), \ldots, (\alpha_s h_s, \alpha_s h_s')}
\end{align*}
{for $\uh, \uh' \in \Z^s$ and $\alpha_1, \ldots, \alpha_s\in G$.}
For finite multisets $E_1, \ldots, E_s\subseteq G$, we define the \textit{box norm} of $f$ {along}  $E_1, \ldots, E_s$ to be
\begin{align}\label{E: box norm}
    \norm{f}_{E_1, \ldots, E_s}:=\brac{\E_{h_1, h_1'\in E_1}\cdots\E_{h_s, h_s'\in E_s}\sum_x \Delta'_{(h_1, h_1'), \ldots, (h_s,h_s')}f(x)}^{1/2^s}.
\end{align}
We can equivalently express this as
\begin{align}\label{E: box norms with Fejer kernels}
    \norm{f}_{E_1, \ldots, E_s} = \brac{\sum_{x, h_1, \ldots, h_s} \mu_{E_1}(h_1)\cdots \mu_{E_s}(h_s) \Delta_{h_1, \ldots, h_s}f(x)}^{1/2^s}
\end{align}
after shifting $x\mapsto x-h'_1-\cdots - h_s'$ and removing the extraneous averages over $h'_1, \ldots, h'_s$. If a multiset $E$ repeats $s$ times, we often write $E^s$ instead of listing it $s$ times, and we also write $\norm{f}_{U^s(E)} = \norm{f}_{E^s}$, calling it the degree-$s$ \textit{Gowers norm} of $f$ along $E$.

For functions $f_\ueps$ indexed by $\ueps\in\{0,1\}^s$, we also define the \textit{Gowers--Cauchy--Schwarz inner product} (or \textit{box inner product}) along $E_1, \ldots, E_s$ to be
\begin{align*}
    \langle (f_\ueps)_{\ueps\in\{0,1\}^s} \rangle_{E_1, \ldots, E_s} &:=\E_{h_1, h_1'\in E_1}\cdots\E_{h_s, h_s'\in E_s}\sum_x \prod_{\ueps\in\{0,1\}^s}\CC^{|\ueps|}f_\ueps(x+(\underline{1}-\ueps)\cdot\uh + \ueps\cdot\uh')\\
    &= \sum_{x, h_1, \ldots, h_s} \mu_{E_1}(h_1)\cdots \mu_{E_s}(h_s) \prod_{\ueps\in\{0,1\}^s}\CC^{|\ueps|}f_\ueps(x+\ueps\cdot\uh).
\end{align*} 
The classical \textit{Gowers--Cauchy--Schwarz inequality}\footnote{Even though this result is usually stated only for sets, exactly the same proof works for multisets.} then states that
\begin{align}
    \langle (f_\ueps)_{\ueps\in\{0,1\}^s} \rangle_{E_1, \ldots, E_s} \leq \prod_{\ueps\in\{0,1\}^s}\norm{f_\ueps}_{E_1, \ldots, E_s}.
\end{align}

Box norms satisfy a number of well known properties, listed in the following lemma.
\begin{lemma}[Properties of box norms]\label{L: properties of box norms}
Let $\delta>0$, let $d,s\in\N$, let $G$ be a countable abelian group, 
and let $E_1, \ldots, E_s\subseteq G$ be finite {and nonempty} multisets. Let $f:G\to\C$ be 1-bounded and supported on a finite nonempty set $B$. Then the following properties hold:
\begin{enumerate}
    \item\label{i: inductive formula} (Inductive formula) For any $k\in[s]$, we have
    \begin{align*}
        \norm{f}_{E_1, \ldots, E_s}^{2^s} = \E_{h_1, h_1'\in E_1}\cdots \E_{h_k, h_k'\in E_k}\norm{\Delta'_{(h_1, h_1'), \ldots, (h_k, h_k')}f}_{E_{k+1}, \ldots, E_s}^{2^{s-k}}.
    \end{align*}
    \item\label{i: permutation invariance} (Permutation invariance) For any permutation $\sigma:[s]\to[s]$, we have 
    \begin{align*}
        \norm{f}_{E_1, \ldots, E_s} = \norm{f}_{E_{\sigma(1)}, \ldots, E_{\sigma(s)}}.
    \end{align*}
    \item\label{i: monotonicity} (Monotonicity) If 
    $|\supp(B-E_s)|\leq C|B|$ for some $C>0$, then
    \begin{align}\label{E: monotonicity}
        \norm{f}_{E_1, \ldots, E_{s-1}}^{2^{s-1}}\geq \delta|B| \quad \Longrightarrow \quad \norm{f}_{E_1, \ldots, E_{s}}^{2^{s}}\geq \delta^2|B|/C.
    \end{align}
    \item\label{i: enlarging} (Enlarging the sets) If $s\geq 2$ and $E'_i\subseteq G$ is a finite multiset satisfying $E'_i \geq E_i$ for each $i\in\{1,\ldots, s\}$, then
    \begin{align*}
        \norm{f}_{E_1, \ldots, E_s}\leq \brac{\frac{|E_1'|\cdots|E_s'|}{|E_1|\cdots|E_s|}}^{1/2^{s-1}}\norm{f}_{E_1', \ldots, E_s'}.
    \end{align*}
    \item\label{i: trimming} (Trimming the dimensions of boxes)\footnote{Although in this and the next two properties we have taken a single parameter $d$, setting some $H_{ij}$'s to be $0$ lets us apply these results to the case where different boxes $E_i$ have different dimensions.} Let $1\leq d'_i\leq d$ for $i\in[s]$, and define the multisets
    \begin{align*}
        E_i = \sum_{j\in [d_i]}\beta_{ij}\cdot[\pm H_{ij}]\quad \textrm{and}\quad E_i' = \sum_{j\in [d'_i]}\beta_{ij}\cdot[\pm H_{ij}]
    \end{align*}
    for some $\beta_{ij}\in G$ and $H_{ij}\in\N_0$. Then
    \begin{align*}
        \norm{f}_{E_1, \ldots, E_s}^{2^s}\geq \delta |B|\quad\Longrightarrow \quad \norm{f}_{E_1', \ldots, E_s'}^{2^s}\geq \delta^{2^s}|B|.
    \end{align*}
    \item\label{i: trimming 2} (Trimming the lengths of boxes) Let
    \begin{align*}
        E_i = \sum_{j\in [d]}\beta_{ij}\cdot[\pm H_{ij}]\quad \textrm{and}\quad E_i' = \sum_{j\in [d]}\beta_{ij}\cdot[\pm \kappa_{ij} H_{ij}]
    \end{align*}
    for some $\beta_{ij}\in G$, $\kappa_{ij}>0$, and $H_{ij}\in \N_0$. Then
    \begin{align*}
        \norm{f}_{E_1, \ldots, E_s}^{2^s}\geq \delta |B|\quad \Longrightarrow\quad        \norm{f}_{E_1', \ldots, E_s'}^{2^s}\gg_{d,s}\delta^{2^s}\min(1, \delta^{2^s}/\kappa)^{2ds}|B|
    \end{align*}
    for $\kappa = \max_{i,j}\kappa_{ij}.$
    \item\label{i: passing to APs} (Passing to sub-arithmetic progressions)
    Suppose that 
    \begin{align*}
        E_i = \sum_{j\in [d]}\beta_{ij}\cdot[\pm H_{ij}]\quad \textrm{and}\quad E_i' = \sum_{j\in [d]}q_{ij}\beta_{ij}\cdot[\pm H_{ij}/q_{ij}]
    \end{align*}
    for some $\beta_{ij}\in G$ and $q_{ij}, H_{ij}\in\N$. There exists a positive real $C=O_{d,s}(1)$ such that if $q_{ij} \leq C \delta^{2^s}H_{ij}$ for all $i, j$, then
    \begin{align*}
        \norm{f}_{E_1, \ldots, E_s}^{2^s}\geq \delta |B|\quad \Longrightarrow\quad        \norm{f}_{E_1', \ldots, E_s'}^{2^s}\gg_{d,s} \delta|B|.
    \end{align*}
\end{enumerate}
\end{lemma}
In applications, we will take $G = \Z^D$, $B = [N]^D$, and $E_i, E'_i$ will be generalized arithmetic progressions of bounded dimension (interpreted appropriately as multisets).

\begin{proof}

The first two properties are standard, and we omit the proofs. For the third property, we expand the definition of $\norm{f}_{E_1, \ldots, E_{s-1}}$ and introduce an extra averaging over $E_s$ to obtain
\begin{align*}
    \norm{f}_{E_1, \ldots, E_{s-1}}^{2^{s-1}} = \E_{h_1, h_1'\in E_1}\cdots\E_{h_{s-1}, h_{s-1}'\in E_{s-1}}\sum_x \E_{h_s\in E_s} \Delta'_{(h_1, h_1'), \ldots, (h_{s-1},h_{s-1}')}f(x+h_s).
\end{align*}
The term $f(x+h_1 + \cdots + h_s)$ in $\Delta'_{(h_1, h_1'), \ldots, (h_{s-1},h_{s-1}')}f(x+h_s)$ vanishes unless $x+h_1+\cdots + h_{s-1} \in B - E_s$, so we can introduce the indicator function $\mathbf{1}_{B-E_s}(x+h_1+\cdots + h_{s-1})$ into the expression above without altering its value. Applying the Cauchy--Schwarz inequality then gives
\begin{align*}
    \delta^2 |B|^2 \leq \norm{f}_{E_1, \ldots, E_{s-1}}^{2^s} &\leq \E_{h_1, h_1'\in E_1}\cdots\E_{h_{s-1}, h_{s-1}'\in E_{s-1}}\sum_x \mathbf{1}_{B-E_s}(x+h_1+\cdots + h_{s-1})\\
    & \qquad \times \E_{h_1, h_1'\in E_1}\cdots\E_{h_s, h_s'\in E_s}\sum_x \Delta'_{(h_1, h_1'), \ldots, (h_s,h_s')}f(x)\\
    & \leq C|B| \cdot \norm{f}_{E_1, \ldots, E_s}^{2^s},
\end{align*}
and the claim follows.

For the fourth property, we write
    \begin{align*}
        \norm{f}_{E_1, \ldots, E_s}^{2^s} = \E_{h_1, h_1'\in E_1}\cdots \E_{h_{s-1}, h_{s-1}'\in E_{s-1}}\sum_x \abs{\E_{h_s\in E_s}\Delta'_{(h_1, h_1'), \ldots, (h_{s-1}, h_{s-1}')}f(x+h_s)}^2.
    \end{align*}
    Using the assumption $s\geq 2$ and nonnegativity, we can extend the averages over $E_1, \ldots, E_{s-1}$ to $E_1', \ldots, E_{s-1}'$ and upper-bound the previous expression by
    \begin{align*}
        \brac{\frac{|E_1'|\cdots|E_{s-1}'|}{|E_1|\cdots|E_{s-1}|}}^2\cdot \E_{h_1, h_1'\in E_1'}\cdots \E_{h_{s-1}, h_{s-1}'\in E_{s-1}'}\sum_x \abs{\E_{h_s\in E_s}\Delta'_{(h_1, h_1'), \ldots, (h_{s-1}, h_{s-1}')}f(x+h_s)}^2.
    \end{align*}    
    We can replace average over $E_s$ with an average over $E_s'$ using the same maneuver.

    To prove the fifth property, we start with the case $s=1$. Dropping the index $i$,
    we can write
    \begin{align*}
        \norm{f}_E^2 = \sum_x \E_{\substack{h'_j\in[\pm H_j]:\\ {j\in [d]}}} f\brac{x+\sum_{j\in [d]}\beta_j h_j'}\E_{\substack{h_j\in[\pm H_j]:\\ {j\in [d]}}}\overline{f\brac{x+\sum_{j\in [d]}\beta_j h_j}}.
    \end{align*}
    Using the Cauchy--Schwarz inequality to double the variables $(h_j)_{j\in [d']}$, we bound
    \begin{multline*}
        \norm{f}_E^4 \leq |B|\cdot \E_{\substack{h_j\in[\pm H_j]:\\ {j\in [d'+1,d]}}} \sum_x \E_{\substack{h_j, h'_j\in[\pm H_j]:\\ {j\in [d']}}}
        f\brac{x+\sum_{j\in [d'+1,d]}\beta_j h_j + \sum_{j\in [d']}\beta_j h_j}\\ \overline{f\brac{x+\sum_{j\in [d'+1,d]}\beta_j h_j + \sum_{j\in [d']}\beta_j h_j'}}.
    \end{multline*}
    Shifting $x\mapsto x - \sum\limits_{j\in [d'+1,d]}\beta_j h_j$, we can ignore the contribution of $j \in [d'+1,d]$ altogether, and we obtain
    \begin{align*}
        \norm{f}_E^4 &\leq |B|\cdot \sum_x \E_{\substack{h_j, h'_j\in[\pm H_j]:\\ {j\in [d']}}}
        f\brac{x+ \sum_{j\in [d']}\beta_j h_j}\overline{f\brac{x+ \sum_{j\in [d']}\beta_j h_j'}} = |B|\cdot \norm{f}_{E'}^2.
    \end{align*}
    We now consider the case $s>1$.  The induction formula for box norms, the Cauchy--Schwarz inequality, and the $s=1$ case give 
    \begin{align*}
        \norm{f}_{E_1, \ldots, E_s}^{2^{s+1}} &\leq \E_{h_1,h_1'\in E_1}\cdots\E_{h_{s-1}, h_{s-1}'\in E_{s-1}}\norm{\Delta'_{(h_1, h_1'), \ldots, (h_{s-1}, h_{s-1}')}f}_{E_s}^4\\
        &\leq |B|\cdot \E_{h_1,h_1'\in E_1}\cdots\E_{h_{s-1}, h_{s-1}'\in E_{s-1}}\norm{\Delta'_{(h_1, h_1'), \ldots, (h_{s-1}, h_{s-1}')}f}_{E'_s}^2 \\
    &= |B| \cdot \norm{f}_{E_1, \ldots, E_{s-1}, E'_s}^{2^s}.    
    \end{align*}
    Performing the same procedure with $E_{s-1}$ in place of $E_s$ gives
    \begin{align*}
        \norm{f}_{E_1, \ldots, E_s}^{2^{s+2}}\leq |B|^2 \cdot \norm{f}_{E_1, \ldots, E_{s-1}, E'_s}^{2^{s+1}}\leq |B|^3 \cdot \norm{f}_{E_1, \ldots, E_{s-2}, E'_{s-1}, E'_s}^{2^{s}},
    \end{align*}
    and continuing in this manner gives
    \begin{align*}
        \norm{f}_{E_1, \ldots, E_s}^{4^{s}}\leq |B|^{2^s-1} \cdot \norm{f}_{E'_1, \ldots, E'_s}^{2^{s}},
    \end{align*}    
    from which the result promptly follows.

    To prove the sixth property, we first set 
    \begin{align*}
        E''_i = \sum_{j\in [d]}\beta_{ij}\cdot[\pm\veps_{ij} H_{ij}]
    \end{align*}
    for parameters $\veps_{ij}>0$ to be chosen later.
    As in the proof of the fifth property, we begin with the $s=1$ case and drop the index $i$. Note that for every $\veps>0$ and $1$-bounded function $g:\Z\to\C$, we have
    \begin{align}\label{E: extra averaging}
        \abs{\E_{h\in[\pm H]}g(h)-\E_{h\in[\pm H]}\E_{h'\in[\pm\veps H]}g(h+h')}\leq 2\veps.
    \end{align}
    Hence we can approximate
    \begin{align*}
        \norm{f}_E^2 &= \sum_x \E_{\substack{h_j, h'_j\in [\pm H_j]:\\ {j\in[d]}}} f\brac{x+\sum_{j\in [d]}\beta_j h_j'} \overline{f\brac{x+\sum_{j\in [d]}\beta_j h_j}}\\
        &=\sum_x \E_{\substack{h_j, h'_j\in [\pm H_j]:\\ {j\in[d]}}} f\brac{x+\sum_{j\in [d]}\beta_j h_j'}\E_{\substack{h''_j\in [\pm\veps_j H_j]:\\ {j\in[d]}}} \overline{f\brac{x+\sum_{j\in [d]}\beta_j (h_j+h''_j)}} + O_d(\veps|B|),
    \end{align*}
    where $\veps := \max_j \veps_j$.
    By the Cauchy--Schwarz inequality in all variables except $h''_j$, we get
    \begin{multline*}
        \norm{f}_{E}^4\leq |B|\cdot \sum_x \E_{\substack{h_j\in [\pm H_j],\\ h'_j, h''_j\in [\pm\veps_j H_j]:\; {j\in[d]}}}f\brac{x+\sum_{j\in [d]}\beta_j (h_j+h'_j)}\\ \overline{f\brac{x+\sum_{j\in [d]}\beta_j (h_j+h''_j)}} + O_d(\veps|B|^2).
    \end{multline*}
    Shifting the variable $x\mapsto x - \sum_{j\in [d]}\beta_j h_j$ gives
    \begin{align*}
        \norm{f}_{E}^4 &\leq |B|\cdot \sum_x \E_{\substack{h'_j, h''_j\in [\pm\veps_j H_j]:\\ {j\in[d]}}}f\brac{x+\sum_{j\in [d]}\beta_j h'_j} \overline{f\brac{x+\sum_{j\in [d]}\beta_j h''_j}} + O_d(\veps|B|^2)\\
        &= |B| \cdot\norm{f}_{E''}^2+ O_d(\veps|B|^2).
    \end{align*}
    Iterating this argument for higher-degree norms as in the proof of \eqref{i: trimming} gives
    \begin{align*}
        \norm{f}_{E_1, \ldots, E_s}^{4^s}\leq |B|^{2^s-1}\cdot\norm{f}_{E_1'', \ldots, E_s''}^{2^s} + O_{d,s}(\veps|B|^{2^s}),
    \end{align*}
    where $\veps := \max_{i,j} \veps_{ij}$, so
    \begin{align*}
        \norm{f}_{E_1, \ldots, E_s}^{2^s}\geq \delta |B|\quad \Longrightarrow\quad \norm{f}_{E_1'', \ldots, E_s''}^{2^s} \geq (\delta^{2^s} - C_{d,s}\veps)|B|
    \end{align*}
    for some constant $C_{d,s}>0$.
    If the $\kappa_{ij}$'s are sufficiently small, namely, if $\kappa:=\max_{i,j}\kappa_{ij}\leq \delta^{2^s}/(2C_{d,s})$,
    then we apply the above bound directly with $\veps_{ij} = \kappa_{ij}$. Otherwise, we first decrease the lengths of boxes, setting $\veps_{ij}=\min(\delta^{2^s}/(2C_{d,s}),\kappa_{ij})$ for all $i,j$,
    and then use \eqref{i: enlarging} to enlarge the boxes from $E''_i$ to $E'_i$, which gives
    \begin{align*}
        \norm{f}_{E_1', \ldots, E_s'}^{2^s} \geq \brac{\prod_{\substack{i\in[s],\; j\in[d]}}\frac{\veps_{ij}^2}{\kappa_{ij}^2}}\norm{f}_{E_1'', \ldots, E_s''}^{2^s}&\gg_{d,s}\delta^{2^s}\cdot(\delta^{2^s}/\kappa)^{2ds}|B|.
    \end{align*}
    Combining these two cases yields the desired bound.

    For the last property, we first observe that for  positive integers $H, q, u$, the sets $[\pm H]$ and $\bigcup\limits_{u=0}^{q-1}(u+q\cdot[\pm H/q])$ differ in at most $2q$ ``boundary'' elements. Hence, if $(x_i)_{i \in \mathbb{Z}}$ is a $1$-bounded sequence of complex numbers, then
    \begin{multline*}
        \abs{\E_{h\in[\pm H]} x_h - \E_{u\in[0,q-1]}\E_{h\in[\pm H/q]} x_{qh + u}}\\
        \leq \frac{2q}{2H+1}+ \abs{\frac{1}{2H+1}-\frac{1}{(2\floor{H/q}+1)q}}(2\floor{H/q}+1)q \leq \frac{3q+1}{2H+1}.
    \end{multline*}
    For the proof proper of the last property, we start with the case $s=1$.
    Approximating each box $\beta_{ij}\cdot[\pm H_{ij}]$ by the union of boxes $\beta_{ij}\cdot(u_j+q_j\cdot[\pm H_j/q_j])$ indexed by $u_j\in[0,q_j-1]$ and applying the estimate above, we get
    \begin{align*}
        \norm{f}_E^2
        &= \E_{\substack{u_j,u_j'\in[0, q_j-1]:\\ j\in[d]}}\sum_x \E_{\substack{h_j, h'_j\in [\pm H_j/q_j]:\\ {j\in[d]}}} f\brac{x+\sum_{j\in [d]}\beta_j (q_j h_j+u_j)}\\
        &\qquad\qquad\qquad\qquad\qquad\qquad\qquad\qquad\overline{f\brac{x+\sum_{j\in [d]}\beta_j (q_j h_j'+u_j')}}+O_d(\veps|B|)
    \end{align*}
    for $\veps := \max_{j\in[d]}\frac{q_j}{H_j}$.
  We pigeonhole in $u_j, u_j'$ and apply
    the Cauchy--Schwarz inequality to get
    \begin{align*}
        \norm{f}_E^4 \leq \norm{f\Bigbrac{\cdot +\sum_{j\in [d]}\beta_j u_j }}_{E'}^2\cdot \norm{f\Bigbrac{\cdot +\sum_{j\in [d]}\beta_j u_j'}}_{E'}^2 + O_d(\veps|B|^2).
    \end{align*}
    The translation-invariance of box norms then gives $\norm{{f}}_{E'}^2\geq (\delta-O_d(\sqrt{\veps}))|B|$.  The statement for general $s$ follows inductively, using the inductive formula for box norms.
\end{proof}

\section{PET induction scheme}\label{S: PET}
In this section, we begin the study of the counting operators
\begin{align}\label{E: counting operator}
	\sum_{\bx}\E_{z\in[K]} f_0(\bx)\cdot f_1(\bx+\p_1(z))\cdots f_\ell(\bx+\p_\ell(z))
\end{align}
associated with families $\CP = (\p_1, \ldots, \p_\ell)$ of essentially distinct polynomials in $\Z^D[z]$ (where ``essentially distinct'' means that all polynomials and their pairwise differences are nonconstant). For instance, in this notation the configuration
      \begin{align*}
          (x_1, x_2),\; (x_1 + z^2+z, x_2),\; (x_1, x_2+z^2+z)
      \end{align*}
      can be rewritten as $\bx,\; \bx+\be_1 (z^2+z),\; \bx + \be_2 (z^2+z)$ and corresponds to the family $\CP = (\be_1 (z^2+z), \be_2 (z^2+z))$.

We present here the standard PET induction scheme which allows us to bound the counting operator \eqref{E: counting operator} by an average of box norms. These box norms can be seen as ``local'' in the sense that the differencing parameters lie in arithmetic progressions of length $\asymp N^{1/d}$, where $d = \max_j \deg \bp_j$. This is the first step in showing that multidimensional polynomial progressions are controlled by a single ``global'' box norm. Combining arguments from \cite{DFKS22, Ts22}, we obtain a precise description of the directions of the box norms; in particular, we show that they come from the coefficients of the polynomials $\bp_j$ and their differences. We start with a simple example in Section \ref{SS: example of PET} that will motivate the argument in the general case. In Section \ref{SS: setting up PET}, we set up the formalism that we need for this purpose and derive various properties of the polynomial families that appear in the PET induction scheme. We then assemble bits and pieces from Section \ref{SS: setting up PET} in order to prove Proposition \ref{P: PET}, the main result of this section.

The material in this subsection is not novel, and various parts of it can be found in previous works on the subject (in \cite{DFKS22, Pel20, Pre17, Ts22}, \emph{inter alia}).

\subsection{Example of PET}\label{SS: example of PET}
Consider the counting operator
\begin{align*}
    \Lambda(f_0, f_1, f_2) := \sum_{\bx}\E_{z\in[K]}f_0(\bx)f_1(\bx + \bbeta_{12}z^2 + \bbeta_{11}z)f_2(\bx + \bbeta_{22}z^2 + \bbeta_{21}z),
\end{align*}
and suppose that 
\begin{align*}
    \Lambda(f_0, f_1, f_2)\geq \delta N^2
\end{align*}
for some $1$-bounded functions $f_0, f_1, f_2:\Z^2\to\C$ supported on $[N]^2$. To avoid degenerate cases, we assume that $\bbeta_{12}, \bbeta_{12}-\bbeta_{22}\neq \textbf{0}$, and that $(\bbeta_{22}, \bbeta_{21})\neq (\textbf{0}, \textbf{0})$.

Our goal in this section is to work out the PET argument for this specific example before delving into it more abstractly. The goal of the PET argument is to pass from the counting operator for a polynomial progression to an average of box norms of one of the functions in the configuration (say, $f_1$). This is achieved in a finite number of steps, in each of which we apply the Cauchy--Schwarz inequality (to remove terms that do not depend on $z$), the van der Corput inequality (to double the variable $z$) and a change of variables (to bring the progression into a form where some functions are independent of $z$). That this procedure ends in finite time follows from the fact that at each step, we reduce the (appropriately measured) complexity, or \textit{type}, of the polynomial progression.

For this particular example, we will need three applications of the aforementioned procedure to reduce the progression to an average of box norms of $f_1$.

\smallskip
\textbf{Step 1}
\smallskip

We apply the Cauchy--Schwarz inequality in $\bx$; since $f_0$ is independent of $z$ and has $L^2$-norm at most $N$, we can obtain the bound
\begin{align*}
    \sum_{\bx}\abs{\E_{z\in[K]}f_1(\bx+ \bbeta_{12}z^2 + \bbeta_{11}z)f_2(\bx + \bbeta_{22}z^2 + \bbeta_{21}z)}^2\geq \delta^2 N^2.
\end{align*}
For the $\Omega(\delta^2 N^2)$ values of $\bx\in[N]^2$ for which the absolute value above is at least $\delta^2 N^2/2$, we then apply the van der Corput inequality (Lemma \ref{L: vdC}) for some $H\ll \delta^8 K$; this has the effect of doubling $z$ and gives
\begin{multline*}
    \sum_{\bx, h_1}\mu_H(h_1)\E_{z\in[K]}f_1(\bx + \bbeta_{12}z^2 + \bbeta_{11}z)\overline{f_1(\bx + \bbeta_{12}(z+h_1)^2 + \bbeta_{11}(z+h_1))}\\
    f_2(\bx + \bbeta_{22}z^2 + \bbeta_{21}z)\overline{f_2(\bx + \bbeta_{22}(z+h_1)^2 + \bbeta_{21}(z+h_1))}\gg \delta^2 N.
\end{multline*}
We now want to bring the configuration into a form where one copy of $f_2$ has an argument independent of $z$; then this copy of $f_2$ can be removed at the subsequent step of the argument. Shifting $\bx \mapsto \bx - \bbeta_{22}z^2 - \bbeta_{21}z$, we obtain
\begin{multline*}
    \sum_{\bx, h_1}\mu_H(h_1)\E_{z\in[K]}f_1(\bx + (\bbeta_{12}-\bbeta_{22})z^2 + (\bbeta_{11}-\bbeta_{21})z)\\
    \overline{f_1(\bx + (\bbeta_{12}-\bbeta_{22})z^2 + (2\bbeta_{12}h_1+\bbeta_{11}-\bbeta_{21} )z + \bbeta_{12}h_1^2 + \bbeta_{11}h_1)}\\
    f_2(\bx)\overline{f_2(\bx+2\bbeta_{22}h_1z + \bbeta_{22} h_1^2 + \bbeta_{21}h_1)}\gg\delta^2 N^2. 
\end{multline*}
This new progression (or, rather, family of progressions indexed by $h_1$) may not look simpler than the original one, but it is: Whereas the original progression involved two quadratic\footnote{in $z$ (the degree of a polynomial in this section always refers to the degree in $z$)} polynomials with distinct leading coefficients $\bbeta_{12}, \bbeta_{22}$, the new progression  involves two quadratic polynomials with the same leading coefficient $\bbeta_{12}-\bbeta_{22}$ (which is nonzero by assumption), as well as a linear polynomial. The decrease in the number of quadratic polynomials with distinct leading coefficients is what makes this progression ``less complex'' than the original one.

\smallskip
\textbf{Step 2}
\smallskip

Once again, we apply the Cauchy--Schwarz inequality in $\bx, h_1$. This time, this has the effect of removing one of the two instances of $f_2$, giving 
\begin{multline*}
    \sum_{\bx, h_1}\mu_H(h_1)\left|\E_{z\in[K]}f_1(\bx + (\bbeta_{12}-\bbeta_{22})z^2 + (\bbeta_{11}-\bbeta_{21})z)\right.\\
    \overline{f_1(\bx + (\bbeta_{12}-\bbeta_{22})z^2 + (2\bbeta_{12}h_1+\bbeta_{11}-\bbeta_{21} )z + \bbeta_{12}h_1^2+ \bbeta_{11}h_1)}\\
    \left.\overline{f_2(\bx+2\bbeta_{22}h_1z + \bbeta_{22} h_1^2 + \bbeta_{21}h_1)}\vphantom{\E_{z\in[K]}}\right|^2\gg \delta^4 N^2.
\end{multline*}
The van der Corput inequality (Lemma \ref{L: vdC}) then gives
\begin{multline*}
    \sum_{\bx, h_1, h_2}\mu_H(h_1, h_2)\E_{z\in[K]}f_1(\bx + (\bbeta_{12}-\bbeta_{22})z^2 + (\bbeta_{11}-\bbeta_{21})z)\\
    \overline{f_1(\bx + (\bbeta_{12}-\bbeta_{22})(z+h_2)^2 + (\bbeta_{11}-\bbeta_{21})(z+h_2))}\\
    \overline{f_1(\bx + (\bbeta_{12}-\bbeta_{22})z^2 + (2\bbeta_{12}h_1+\bbeta_{11}-\bbeta_{21} )z + \bbeta_{12}h_1^2+ \bbeta_{11}h_1)}\\
    {f_1(\bx + (\bbeta_{12}-\bbeta_{22})(z+h_2)^2 + (2\bbeta_{12}h_1+\bbeta_{11}-\bbeta_{21} )(z+h_2) + \bbeta_{12}h_1^2+ \bbeta_{11}h_1)}\\
    \overline{\Delta_{2\bbeta_{22}h_1 h_2}f_2(\bx+2\bbeta_{22}h_1z + \bbeta_{22} h_1^2 + \bbeta_{21}h_1)}\gg \delta^4 N^2. 
\end{multline*}
We want to make the remaining instances of $f_2$ independent of $z$.  To this end we shift\footnote{If $\bbeta_{22}=\textbf{0}$, then this term can be removed before applying the Cauchy--Schwarz inequality, and we can proceed directly to the change of variables performed in Step 3 below.} $\bx \mapsto \bx - 2\bbeta_{22}h_1 z$, which gives
\begin{multline*}
    \sum_{\bx, h_1, h_2}\mu_H(h_1, h_2)\E_{z\in[K]}f_1(\bx + (\bbeta_{12}-\bbeta_{22})z^2 + (-2\bbeta_{22}h_1+\bbeta_{11}-\bbeta_{21})z)\\
    \overline{f_1(\bx + (\bbeta_{12}-\bbeta_{22})(z+h_2)^2 + (-2\bbeta_{22}h_1+\bbeta_{11}-\bbeta_{21})z + \ba_1(h_1, h_2))}\\
    \overline{f_1(\bx + (\bbeta_{12}-\bbeta_{22})z^2 + (2(\bbeta_{12}-\bbeta_{22})h_1+\bbeta_{11}-\bbeta_{21} )z + \ba_2(h_1,h_2))}\\
    {f_1(\bx + (\bbeta_{12}-\bbeta_{22})(z+h_2)^2 + (2(\bbeta_{12}-\bbeta_{22})h_1+\bbeta_{11}-\bbeta_{21} )z +  \ba_3(h_1,h_2))}\\
    \overline{\Delta_{2\bbeta_{22}h_1 h_2}f_2(\bx+\ba_4(h_1, h_2))}\gg \delta^4 N^2
\end{multline*}
for some polynomials $\ba_j$ in $h_1, h_2$, the exact form of which is irrelevant for us. Even though the configuration has increased in length (this is an unfortunate aspect of the PET argument), it has become less complex in the sense that all of the quadratic polynomials now have the same leading coefficient, and there is no linear polynomial left (as contrasted with the previous step, where we still had one polynomial linear in $z$). We are ready for the last step, in which we will remove $f_2$ altogether.

\smallskip
\textbf{Step 3}
\smallskip

Proceeding as before, we apply the Cauchy--Schwarz inequality in $\bx, h_1, h_2$ to get
\begin{multline*}
    \sum_{\bx, h_1, h_2}\mu_H(h_1, h_2)\left|\E_{z\in[K]}f_1(\bx + (\bbeta_{12}-\bbeta_{22})z^2 + (-2\bbeta_{22}h_1+\bbeta_{11}-\bbeta_{21})z)\right.\\
    \overline{f_1(\bx + (\bbeta_{12}-\bbeta_{22})(z+h_2)^2 + (-2\bbeta_{22}h_1+\bbeta_{11}-\bbeta_{21})z + \ba_1(h_1, h_2))}\\
    \overline{f_1(\bx + (\bbeta_{12}-\bbeta_{22})z^2 + (2(\bbeta_{12}-\bbeta_{22})h_1+\bbeta_{11}-\bbeta_{21} )z + \ba_2(h_1,h_2))}\\
    \left.\vphantom{\E_{z\in[K]}} {f_1(\bx + (\bbeta_{12}-\bbeta_{22})(z+h_2)^2 + (2(\bbeta_{12}-\bbeta_{22})h_1+\bbeta_{11}-\bbeta_{21} )z +  \ba_3(h_1,h_2))}\right|^2\gg \delta^8 N^2. 
\end{multline*}
An application of Lemma \ref{L: vdC} then gives
\begin{multline*}
    \sum_{\bx, h_1, h_2, h_3}\mu_H(h_1, h_2, h_3)\E_{z\in[K]}f_1(\bx + (\bbeta_{12}-\bbeta_{22})z^2 + (-2\bbeta_{22}h_1+\bbeta_{11}-\bbeta_{21})z)\\
    \overline{f_1(\bx + (\bbeta_{12}-\bbeta_{22})(z+h_3)^2 + (-2\bbeta_{22}h_1+\bbeta_{11}-\bbeta_{21})(z+h_3))}\\
    \overline{f_1(\bx + (\bbeta_{12}-\bbeta_{22})(z+h_2)^2 + (-2\bbeta_{22}h_1+\bbeta_{11}-\bbeta_{21})z + \ba_1(h_1, h_2))}\\
    {f_1(\bx + (\bbeta_{12}-\bbeta_{22})(z+h_2+h_3)^2 + (-2\bbeta_{22}h_1+\bbeta_{11}-\bbeta_{21})(z+h_3) + \ba_1(h_1, h_2))}\\
    \overline{f_1(\bx + (\bbeta_{12}-\bbeta_{22})z^2 + (2(\bbeta_{12}-\bbeta_{22})h_1+\bbeta_{11}-\bbeta_{21} )z + \ba_2(h_1,h_2))}\\
    {f_1(\bx + (\bbeta_{12}-\bbeta_{22})(z+h_3)^2 + (2(\bbeta_{12}-\bbeta_{22})h_1+\bbeta_{11}-\bbeta_{21})(z+h_3) + \ba_2(h_1,h_2))}\\
    {f_1(\bx + (\bbeta_{12}-\bbeta_{22})(z+h_2)^2 + (2(\bbeta_{12}-\bbeta_{22})h_1+\bbeta_{11}-\bbeta_{21} )z +  \ba_3(h_1,h_2))}\\
    \overline{f_1(\bx + (\bbeta_{12}-\bbeta_{22})(z+h_2+h_3)^2 + (2(\bbeta_{12}-\bbeta_{22})h_1+\bbeta_{11}-\bbeta_{21} )(z+h_3) +  \ba_3(h_1,h_2))}\\
    \gg \delta^8 N^2. 
\end{multline*}
Shifting $$\bx\mapsto \bx - (\bbeta_{12}-\bbeta_{22})z^2 - (-2\bbeta_{22}h_1+\bbeta_{11}-\bbeta_{21})z,$$ we end up with
\begin{multline*}
    \sum_{\bx, h_1, h_2, h_3}\mu_H(h_1, h_2, h_3)\E_{z\in[K]}f_1(\bx)
    \overline{f_1(\bx + 2(\bbeta_{12}-\bbeta_{22})h_3 z+\ba_1(h_1, h_2, h_3))}\\
    \overline{f_1(\bx + 2(\bbeta_{12}-\bbeta_{22})h_2 z+\ba_2(h_1, h_2, h_3))}
    {f_1(\bx + 2(\bbeta_{12}-\bbeta_{22})(h_2+h_3)z + \ba_3(h_1, h_2, h_3))}\\
    \overline{f_1(\bx +2\bbeta_{12}h_1z + \ba_4(h_1,h_2, h_3))}
    {f_1(\bx + (2(\bbeta_{12}-\bbeta_{22})h_3 + 2\bbeta_{12}h_1)z + \ba_5(h_1,h_2, h_3))}\\
    {f_1(\bx + (2(\bbeta_{12}-\bbeta_{22})h_2 + 2\bbeta_{12}h_1)z +  \ba_6(h_1,h_2, h_3))}\\
    \overline{f_1(\bx + (2(\bbeta_{12}-\bbeta_{22})(h_2+h_3) + 2\bbeta_{12}h_1) z +  \ba_7(h_1,h_2, h_3))}\gg \delta^8 N^2.
\end{multline*}
for some polynomials $\ba_j$ in $h_1, h_2, h_3$. The important point about this new expression is that it is an average, indexed by $(h_1, h_2, h_3)$, of counting operators for the $8$-point linear (in $z$) configurations
\begin{gather*}
    \bx,\; \bx + 2(\bbeta_{12}-\bbeta_{22})h_3 z,\; \bx + 2(\bbeta_{12}-\bbeta_{22})h_2 z,\; \bx + 2(\bbeta_{12}-\bbeta_{22})(h_2+h_3)z,\\ \bx +2\bbeta_{12}h_1z,\;  \bx + (2(\bbeta_{12}-\bbeta_{22})h_3 + 2\bbeta_{12}h_1)z,\; \bx + (2(\bbeta_{12}-\bbeta_{22})h_2 + 2\bbeta_{12}h_1)z,\\ \bx + (2(\bbeta_{12}-\bbeta_{22})(h_2+h_3) + 2\bbeta_{12}h_1) z;
\end{gather*}
as such, each of these configurations can be controlled individually by a certain degree-$7$ box norm as in Lemma \ref{L: linear averages} below, using Gowers' original argument for controlling arithmetic progressions by Gowers norms \cite{Go98a, Go01}. Gowers' argument, which we abstain from showing here, consists of $7$ more applications of the Cauchy--Schwarz and van der Corput inequalities, together with changes of variables, and its output is the norm-control
\begin{align*}
     \sum_{h_1, h_2, h_3}\mu_H(h_1, h_2, h_3)\norm{f_1}^{2^7}_{\substack{2(\bbeta_{12}-\bbeta_{22})h_3\cdot[\pm M],\; 2(\bbeta_{12}-\bbeta_{22})h_2\cdot[\pm M],\\ 2(\bbeta_{12}-\bbeta_{22})(h_2+h_3)\cdot[\pm M],\; 2\bbeta_{12}h_1\cdot[\pm M],\\ (2(\bbeta_{12}-\bbeta_{22})h_3 + 2\bbeta_{12}h_1)\cdot[\pm M],\\ (2(\bbeta_{12}-\bbeta_{22})h_2 + 2\bbeta_{12}h_1)\cdot[\pm M],\\  (2(\bbeta_{12}-\bbeta_{22})(h_2+h_3) + 2\bbeta_{12}h_1)\cdot[\pm M]}}\gg \delta^{2^{10}}N^2
\end{align*}
for any $M\ll \delta^{2^{10}} K$. The important points about this average of box norms are that the directions
\begin{gather*}
     2(\bbeta_{12}-\bbeta_{22})h_3 ,\;  2(\bbeta_{12}-\bbeta_{22})h_2 ,\;  2(\bbeta_{12}-\bbeta_{22})(h_2+h_3),\\ 2\bbeta_{12}h_1,\;   2(\bbeta_{12}-\bbeta_{22})h_3 + 2\bbeta_{12}h_1,\;  2(\bbeta_{12}-\bbeta_{22})h_2 + 2\bbeta_{12}h_1,\\  2(\bbeta_{12}-\bbeta_{22})(h_2+h_3) + 2\bbeta_{12}h_1
\end{gather*}
in the average of box norms above depend only on the leading coefficients of the polynomials $\bp_1, \bp_1-\bp_2$, and that they are linear polynomials in $h_1, h_2, h_3$ (for a general polynomial family $\CP$, the directions will be multilinear polynomials in $h_i$ of degree one smaller than the degree  of $\CP$). Understanding these directions for a general polynomial family is an important (even if tedious) task, on which we embark now.

\subsection{Setting up the induction scheme}\label{SS: setting up PET}

In order to carry out the PET argument in general, we first need to set up some notation that allows us to talk about general multidimensional polynomial progressions. Let $d, D, \ell, r, s\in\N$, and let $\uh:=(h_1, \ldots, h_r)$ be a tuple of length $r$. All of the polynomial families that we shall consider will take the following form.
\begin{definition}[Normal polynomial family]
    Let $\CQ = (\q_1, \ldots, \q_{s})$ be a family of polynomials in $\Z^D[z, \uh]$ of degree at most $d$. We call $\CQ$ \textit{normal} if the polynomials $\q_1, \ldots, \q_s$ are nonzero and distinct, $\q_1$ has highest degree (with respect to the variable $z$), and $\q_1(0, \uh), \ldots, \q_s(0,\uh)$ are the zero polynomial for all $\uh\in\Z^r$ (equivalently, $z$ divides all monomial terms of each $\q_i(0,\uh)$).
\end{definition}

 During the PET argument, we will subject $\CP$ to an iterated application of the following operation.
      \begin{definition}[van der Corput operation]
          Let $\CQ = (\q_1, \ldots, \q_s)$ be a normal family of polynomials in $\Z^D[z,\uh]$. Given $m\in[s]$, we define the new polynomial family
      \begin{align*}
            \partial_m\CQ := &(\sigma_{h_{r+1}}{\q}_1-{\q}_m, \ldots, \sigma_{h_{r+1}}{\q}_s - {\q}_m, {\q}_1-{\q}_m, \ldots,  {\q}_s-{\q}_m)^*,
      \end{align*}
      where $\sigma_{h_{r+1}}\q(z, \uh) := \q(z+h_{r+1}, \uh) - \q(h_{r+1}, \uh)\in \mathbb{Z}^D[z,h,h_{r+1}]$
      and the $^*$ operation removes all zero polynomials and all but the first copy of each polynomial that appears multiple times.
            \end{definition}
            We refer to the operation $\CQ\mapsto \partial_m \CQ$ as a \textit{van der Corput operation} on $\CQ$ since the new family $\partial_m \CQ$ will arise from $\CQ$ via an application of the van der Corput inequality. It is through this operation that we obtain a new polynomial family from $\CQ$ at each step of the PET induction scheme.  

            {We caution the reader that $\partial_m \CQ$ need not be normal even if $\CQ$ is, since $\sigma_{h_{r+1}}{\q}_1-{\q}_m$ need not be the highest-degree polynomial (in $z$) in $\partial_m\CQ$ (the other conditions for normality are easily seen to be met). For instance, if $\CQ = (z^2, 2z^2)$, then $\partial_1 \CQ = (2h z, z^2 + 4zh, z^2)$, and clearly the first polynomial is lower-degree in $z$ than the other two. However, in ensuing applications, we will always choose the index $m$ to maximize the degree of $\sigma_{h_{r+1}}{\q}_1-{\q}_m$ among the elements of $\partial_m \CQ$.\footnote{This condition is explicitly assumed in Proposition \ref{P: vdC preserves descendancy} and ensured in Proposition \ref{P: vdC terminates} below.} For instance, we will never pass from the $\CQ$ as above to $\partial_1\CQ$, but rather to the normal family $\partial_2\CQ = (-z^2 + 2hz, -z^2, 4hz)$.}

    The following lemma shows how we can use the Cauchy--Schwarz and van der Corput inequalities, together with changes of variables, to pass  from the family $\CQ$ to the family $\partial_m\CQ$ at each step of PET.
    \begin{lemma}\label{L: bounding averages in vdC}
        Let $d, D, m, \ell, r, s\in\N$, and let $\CQ = (\q_1, \ldots, \q_s)$ be a normal family of polynomials in $\Z^D[z,\uh]$ of degree at most $d$. Suppose that the degree of $Q_1$ as a polynomial in $z$ exceeds $1$.
        Let $\delta>0$ and $H, K, N\in\N$ be parameters satisfying 
        $H\leq \frac{\delta^2}{16}K$. Then for any $1$-bounded functions $f_0, \ldots, f_\ell:\Z^D\to\C$ supported on $[N]^D$ and any $\uh\in\Z^s$, the lower bound
        \begin{align*}
            \abs{\sum_{\bx}\E_{z\in[K]}
            f_0(\bx)f_1(\bx+\q_1(z,\uh))\cdots f_{s}(\bx+\q_s(z,\uh))} \geq \delta N^D
        \end{align*}
        implies that
        \begin{align*}
            \sum_{h_{r+1}}\mu_H(h_{r+1})\sum_{\bx}\E_{z\in[K]}f_{\mathbf{Q_0}}(\bx)\cdot \prod_{\bq\in\partial_m\CQ}f_{\bq}(\bx+\q(z,\uh, h_{r+1}))\geq \frac{\delta^2}{16}N^D
        \end{align*}
        for some $1$-bounded functions $f_\bq:\Z\to\C$ {(depending on $\uh, h_{r+1}$)} with $f_{\sigma_{h_{r+1}}\q_1 - \q_m} = f_1$. 
    \end{lemma}
    The property $f_{\sigma_{h_{r+1}}\q_1 - \q_m} = f_1$ is crucial for letting us keep track of the function $f_1$ over the course of the PET procedure.

    {In forthcoming applications, we will compare weighted averages over $\uh$ of the inequalities considered in Lemma \ref{L: bounding averages in vdC}, hence the seemingly redundant presence of $\uh$ in the statement of the lemma.}
    \begin{proof}
     From the triangle inequality we see that
        \begin{align*}
            \sum_{\bx}\abs{\E_{z\in[K]}
            f_0(\bx)f_1(\bx+\q_1(z,\uh))\cdots f_{s}(\bx+\q_s(z,\uh))} \geq \delta N^D.
        \end{align*}     
        By the popularity principle, we have
                \begin{align*}
            \abs{\E_{z\in[K]}
            f_1(\bx+\q_1(z,\uh))\cdots f_{s}(\bx+\q_s(z,\uh))} \geq (\delta/2) N^D.
        \end{align*} 
        for at least $(\delta/2)N^D$ values $\bx\in[N]^D$. 
     Applying the asymmetric version of the van der Corput inequality from Lemma \ref{L: vdC} to each such value $\bx$, 
     we deduce, for arbitrary $H\leq \frac{\delta^2}{16}K$, that
     \begin{align*}
         \sum_{\bx}\sum_{h_{r+1}}\mu_H(h_{r+1})\E_{z\in[K]}\prod_{j\in[s]}\brac{f_j(\bx+\q_j(z,\uh))\overline{f_j}(\bx+\q_j(z+h_{r+1},\uh))} \geq \frac{\delta^2}{16}N^D.
     \end{align*}
     Swapping the order of summation, shifting $\bx\mapsto \bx -{\bq}_m(z, \uh)-\bq_1(h_{r+1}, \uh)$, and swapping the order of summation again, we deduce that
     \begin{align}\label{E: PET reduction intermediate}
         \sum_{h_{r+1}}\mu_H(h_{r+1})\sum_{\bx}\E_{z\in[K]}\prod_{j\in[s]}\big(f_j(\bx+\mathbf{R}_j(z,\uh, h_{r+1}))\overline{f_j}(\bx+\mathbf{R}'_j(z,\uh, h_{r+1})\big) \geq \frac{\delta^2}{16}N^D,
     \end{align}
     where
     \begin{align*}
         \mathbf{R}_j(z,\uh, h_{r+1}) &:= \q_j(z,\uh)-{\bq}_m(z, \uh)-\bq_1(h_{r+1}, \uh),\\
         \mathbf{R}'_j(z,\uh, h_{r+1}) &:= \q_j(z+h_{r+1},\uh)-{\bq}_m(z, \uh)-\bq_1(h_{r+1}, \uh).
     \end{align*}
     
     To show that \eqref{E: PET reduction intermediate} indeed takes the form stipulated by the lemma, we observe that up to subtracting terms depending only on $\uh, h_{r+1}$, each polynomial $\mathbf{R}_j, \mathbf{R}_j'$ is one of
        \begin{align}\label{E: polys after vdC}
            \sigma_{h_{r+1}}{\q}_2-{\q}_m, \ldots, \sigma_{h_{r+1}}{\q}_s - {\q}_m, {\q}_1-{\q}_m, \ldots,  {\q}_s-{\q}_m
        \end{align}
        or $\sigma_{h_{r+1}}{\q}_1-{\q}_m$. Hence we can group all the terms in \eqref{E: PET reduction intermediate} whose polynomials equal some particular $\bq\in\partial_m \CQ\cup\{\mathbf{0}\}$ up to terms depending only on $\uh, h_{r+1}$. Specifically, if 
        \begin{align*}
            \mathcal{R}_\bq &:=\{j\in[s]:\; \mathbf{R}_j(z,\uh, h_{r+1}) - \bq(z, \uh, h_{r+1})\in\Z[\uh,h_{r+1}]\},\\
            \mathcal{R}'_\bq &:=\{j\in[s]:\; \mathbf{R}'_j(z,\uh, h_{r+1}) - \bq(z, \uh, h_{r+1})\in\Z[\uh,h_{r+1}]\},
        \end{align*}
        then we set 
        \begin{align*}
            f_\bq(\bx) := \prod_{j\in \mathcal{R}_\bq}f_j(\bx + (\mathbf{R}_j-\bq)(z,\uh, h_{r+1}))\cdot \prod_{j\in \mathcal{R}'_\bq}\overline{f_j}(\bx + (\mathbf{R}'_{j}-\bq)(z, \uh, h_{r+1}));
        \end{align*}
        notice that this depends on $\uh, h_{r+1}$ but not $z$.
        Hence the inequality takes the claimed form, with each $f_\bq$ being a product of some $f_j$'s and their conjugates shifted by polynomials in $\uh, h_{r+1}$. We remark that this grouping of terms corresponds to applying the operation $*$ when obtaining the polynomial family $\partial_m \CQ$ from $\CQ$.

        What remains to be shown is that $f_{\sigma_{h_{r+1}}\q_1 - \q_m} = {f_1}$. To this end, we make an important observation that each polynomial in \eqref{E: polys after vdC} is distinct from $\sigma_{h_{r+1}}{\q}_1-{\q}_m$. This can be seen on a case-by-case basis:
        \begin{enumerate}
            \item $\sigma_{h_{r+1}}{\q}_1-{\q}_m$ is distinct from $\sigma_{h_{r+1}}{\q}_j-{\q}_m$ for $j\neq 1$ because $\q_1$ is distinct from $\q_j$, which is a consequence of the normality of $\CQ$;
            \item $\sigma_{h_{r+1}}{\q}_1-{\q}_m$ is distinct from ${\q}_j-{\q}_m$ for $j\neq 1$ because the $h_{r+1}$-free parts are simply $\bq_1-\bq_m$ and $\bq_j-\bq_m$, which are distinct by the normality of $\CQ$;
            \item $\sigma_{h_{r+1}}{\q}_1-{\q}_m$ is distinct from ${\q}_1-{\q}_m$ because $\q_1$ is nonlinear in $z$.\footnote{Note that in the linear case, ${\q}_1-{\q}_m$ and $\sigma_{h_{r+1}}{\q}_1-{\q}_m$ would be the same because the $\sigma$ operation would remove entirely the shift by $h_{r+1}$.}
        \end{enumerate}
        Consequently, when we group terms above, we do not group any term with $f_1$. Hence $f_{\sigma_{h_{r+1}}\q_1 - \q_m} = \overline{f_1}$, since $\mathbf{R}_1' = \sigma_{h_{r+1}}\q_1 - \q_m$.    
        The complex conjugate can be removed by conjugating the left-hand side of \eqref{E: PET reduction intermediate} and redefining other $f_\bq$'s accordingly.     
    \end{proof}

    We deal with the linear case differently. 
    \begin{lemma}\label{L: linear averages}
    Let $\ell, D\in\N$, and let $\balpha_{1}, \ldots, \balpha_\ell\in\Z^D$ be nonzero vectors and set $\balpha_0 := \mathbf{0}$. Let $\delta>0$ and $K, M, N\in\N$ be parameters satisfying $M\leq \frac{\delta^{2^\ell}}{16^{2^\ell - 1}}K$. Then 
    for any $1$-bounded functions $f_0, \ldots, f_\ell:\Z^D\to\C$ supported on $[N]^D$, the lower bound
    \begin{align*}
        \abs{\sum_{\bx}\E_{z\in[K]}f_0(\bx)\cdot \prod_{j\in[\ell]}f_j(\bx+\balpha_j z)}\geq \delta N^D
    \end{align*}
    implies that
    \begin{align*}
        \E_{\um,\um\in[\pm M]^\ell}\sum_{\bx} \Delta'_{\balpha_j - \balpha_0, \ldots,  \balpha_j - \balpha_{j-1}, \balpha_j - \balpha_{j+1}, \ldots, \balpha_j - \balpha_{\ell}; (\um,\um')} f_j(\bx) \geq \frac{\delta^{2^\ell}}{16^{2^\ell - 1}} N^D
    \end{align*}
    for each $j\in[\ell]$.
\end{lemma}
{We remark that the result needs no bounds on the size of the vectors $\balpha_j$, for if they are comparable with $N$, then the result essentially degenerates. For instance, if we consider the 1-dimensional progression $x,\; x+\frac{N}{10}z$ and set $K = N$, then the constraints on the support of $f_0, f_1$ force the summand to vanish whenever $z\geq 10$. In that case the only permissible $\delta$'s for which the lower bound holds are themselves of size $\asymp 1/N$, and the conclusion becomes trivial.}
\begin{proof}
Since the hypothesis of the lemma is invariant under permutations of the functions $f_1, \ldots, f_\ell$, it suffices to consider the case $j=1$. We proceed by induction on $\ell$. The popularity principle yields at least $(\delta/2)N^D$ values of $\bx\in[N]^D$ for which 
    \begin{align*}
        \abs{\E_{z\in[K]}\prod_{j\in[\ell]}f_j(\bx+\balpha_j z)}\geq (\delta/2) N^D.
    \end{align*}
The symmetric van der Corput inequality from Lemma \ref{L: vdC} gives
    \begin{align*}
        \E_{m_1, m_1'\in[\pm M]}\sum_{\bx}\E_{z\in[K]}\prod_{j\in[\ell]}\Delta'_{(\balpha_j m_1, \balpha_j m_1')}f_j(\bx+\balpha_j z) \geq \frac{\delta^2}{16}N^D
    \end{align*}
    as long as $M\leq \frac{\delta^2}{16}K$.
    We then translate $\bx\mapsto\bx-\balpha_\ell z$ (temporarily swapping the order of $\bx$ and $z$ for that purpose), so that
    \begin{align*}
        \E_{m_1, m_1'\in[\pm M]}\sum_{\bx}\E_{z\in[K]}\prod_{j\in[\ell]}\Delta'_{(\balpha_j m_1, \balpha_j m_1')}f_j(\bx+(\balpha_j-\balpha_1) z) \geq \frac{\delta^2}{16}N^D.
    \end{align*}    
    This immediately gives the claimed result in the case $\ell = 1$. For $\ell>1$, we apply the induction hypothesis (separately for each $(\um, \um')$) to the shorter progression
    \begin{align*}
        \bx,\; \bx + (\balpha_1 - \balpha_\ell)z,\; \ldots,\; \bx + (\balpha_{\ell-1} - \balpha_\ell)z.
    \end{align*}
    and the functions $\Delta'_{(\balpha_j m_1, \balpha_j m_1')}f_j$ in place of $f_j$ upon observing that $$(\balpha_1 - \balpha_\ell) - (\balpha_j - \balpha_\ell) = \balpha_1 - \balpha_j.$$ 
    The final lower bound and a threshold on $M$ follow from a rather straightforward computation.
\end{proof}

       {Let $\CP = (\p_1, \ldots, \p_\ell)$ be a normal family of polynomials in $\Z^D[z]$. Our next goal is to describe the properties of $\CP$ that are preserved by the van der Corput operation. Let $d =\deg\p_1$; by normality, $\deg p_j\leq d$ for all $j\in[\ell]$.
We denote the coefficients of the polynomials $\p_j$ by $\bp_j(z) = \sum_{i=1}^d \bbeta_{ji}z^i$, and we  also set $\bp_0 := \mathbf{0}$ and $\bbeta_{0i} := \mathbf{0}$.} 

Let $\CQ = (\q_1, \ldots, \q_{s})$ be a normal family of polynomials in $\Z^D[z, \uh]$ of degree at most $d$, and write 
      \begin{align}\label{E: q_j}
        \bq_j(z,\uh) = \sum_{i=1}^d \bgamma_{ji}(\uh)z^i      
      \end{align}
     for some $\bgamma_{ji}\in\Z[\uh]$. In other words, the $\bgamma_{ji}$'s are the coefficients of $\bq_j$ as a polynomial in $z$, and they are themselves polynomials in $\uh$. For $j,j'\in[0,s]$, also define
      $$d_{jj'} := \max\{i\in[d]:\; \bgamma_{ji}(\uh)\neq \bgamma_{j'i}(\uh)\}$$
      to be the degree of $\q_j - \q_{j'}$ as a polynomial in $z$,
      where we set $\bq_0 := \mathbf{0}$ and $\bgamma_{0i}:=\mathbf{0}$ for all $i\in\N_0$.
      We want to show that if $\CQ$ is a family obtained from $\CP$ at some stage of the PET induction scheme, then the polynomials $\bgamma_{ji}$ preserve several useful properties.
      
      \begin{definition}[Descendent family of polynomials]\label{D: descendence}
      A {normal} family $\CQ = (\q_1, \ldots, \q_{s})$ of polynomials in $\Z^D[z, \uh]$ \textit{descends from} $\CP$ if for each $j\in[s]$, the polynomials $\bgamma_{ji}$ in \eqref{E: q_j} can be expressed as\footnote{The multinomial coefficient
      \begin{align*}
          {{{u_1+\cdots + u_r+i}\choose{u_1, \ldots, u_r}}} = \frac{(u_1+\cdots+u_r+i)!}{u_1!\cdots u_r!\cdot i!}
      \end{align*}
      is the coefficient of $\uh^{\uu}z^i = h_1^{u_1}\cdots h_r^{u_r} z^i$ in $(z+h_1+\cdots + h_r)^{i+u_1+\cdots+u_r}$.}
          \begin{align}\label{E: gamma_ji}
              \bgamma_{ji}(\uh) =  \sum_{\substack{\uu\in \N_0^{r},\\ |\uu|\leq d-i}} {{{u_1+\cdots + u_r+i}\choose{u_1, \ldots, u_r}}}(\bbeta_{w_{j\uu}(|\uu|+i)}-\bbeta_{w_{\uu}(|\uu|+i)})\uh^\uu
          \end{align}
        for some indices $w_{\uu}, w_{j\uu}\in\{0, \ldots, \ell\}$ satisfying the following properties:
      \begin{enumerate}
          \item\label{i: w_1} (Coefficients of $\q_1$) $w_{1\uu} = 1$ for all $\uu$.
          \item\label{i: multilinear} (Multilinearity) For each $j\BK{\in[0,s]}$, the polynomial $\bgamma_{1d_{1j}}-\bgamma_{jd_{1j}}$ (the leading coefficient of $\bq_1 - \bq_{j}$ as a polynomial in $z$) is multilinear in $\uh$.  
        \item\label{i: leading coeffs} (Leading coefficients) For each $j\BK{\in[0,s]}$ and multidegree $\uu$, the vector $\bbeta_{1(|\uu|+d_{1j})}-\bbeta_{w_{j\uu}(|\uu|+d_{1j})}$ is nonzero only if it is the leading coefficient of  $\bp_1-\bp_{w_{j\uu}}$.
      \end{enumerate}
      \end{definition}
        
      {\begin{example}
      Consider the polynomial family $\CP = (\bp_1, \bp_2)$ given by $$\bp_1(z) =\bbeta_{12}z^2 + \bbeta_{11}z\quad \textrm{and}\quad \bp_2(z) = \bbeta_{22}z^2 + \bbeta_{21}z,$$ as studied in Section \ref{SS: example of PET}. Like we did there, we assume that $\bbeta_{12}, \bbeta_{12}-\bbeta_{22}\neq \mathbf{0}$. Also let $\CQ = (\bq_1, \bq_2, \bq_3)$, where
      \begin{gather*}
             \bq_1(z,h) = (\bbeta_{12}-\bbeta_{22})z^2 + (2\bbeta_{12}h+\bbeta_{11}-\bbeta_{21} )z + \bbeta_{12}h^2 + \bbeta_{11}h,\\ \bq_2(z,h) = (\bbeta_{12}-\bbeta_{22})z^2 + (\bbeta_{11}-\bbeta_{21})z,\quad \bq_3(z,h) = 2\bbeta_{22}hz + \bbeta_{22} h^2 + \bbeta_{21}h.
      \end{gather*}
      The family $\CQ$ consists of polynomials from $\partial_{2}\CP$ after chopping off the terms independent from $z$, and it has been obtained from $\CP$ in the first step of the PET argument outlined in Section \ref{SS: example of PET}. We claim that $\CQ$ descends from $\CP$. Observe first that
      \begin{align*}
          \bgamma_{11}(h) &= 2\bbeta_{12}h+\bbeta_{11}-\bbeta_{21}, & \bgamma_{12}(h) &= \bbeta_{12}-\bbeta_{22},\\
          \bgamma_{21}(h) &= \bbeta_{11}-\bbeta_{21}, & \bgamma_{22}(h) &= \bbeta_{12}-\bbeta_{22},\\
          \bgamma_{31}(h) &= 2\bbeta_{22}h, & \bgamma_{32}(h) &= \textbf{0},
      \end{align*}
      and so $d_{10} = 2$ (since $\bgamma_{12}\neq \mathbf{0}$), $d_{12} = 1$ (since $\bgamma_{12} = \bgamma_{22}$, and $\bgamma_{11}\neq \bgamma_{21}$) and $d_{13} = 2$ (since $\bgamma_{12}\neq \bgamma_{32}$).
    An inspection of the coefficients of $\bgamma_{ji}(h)$ shows that they are indeed of the form \eqref{E: gamma_ji}; for instance, $\bgamma_{11}, \bgamma_{12}$ equal
    \begin{align*}
        \bgamma_{11}(h) = 2(\bbeta_{w_{11}2}-\bbeta_{w_12})h+\bbeta_{w_{10}1}-\bbeta_{w_01}\quad\textrm{and}\quad \bgamma_{12}(h) = \bbeta_{w_{10}2}-\bbeta_{w_02}
    \end{align*}
    with $w_{11} = w_{10} = 1$, $w_1 = 0$, $w_0 = 2$. This in particular implies the property \eqref{i: w_1} from Definition \ref{D: descendence}. The (multi)linearity property \eqref{i: multilinear} of $\bgamma_{ji}$'s and their differences is immediate. Lastly, the special form \eqref{i: leading coeffs} of the coefficients is a more detailed version of the observation that the leading coefficients of the polynomials $\bgamma_{12}$, $\bgamma_{11}-\bgamma_{21}$, and $\bgamma_{12}-\bgamma_{32}$ (which are themselves ``leading coefficients'' of $\bq_1, \bq_1-\bq_2$, $\bq_1 - \bq_3$ understood as polynomials in $z$) are all scalar multiples of $\bbeta_{12}, \bbeta_{12}-\bbeta_{22}$
 
      \end{example}}
            The rationale behind recording the particular properties listed in Definition \ref{D: descendence} is as follows. The normality of the family ensures that the progression does not become degenerate\footnote{A degenerate family would be one in which the difference of some two polynomials is independent of $z$; such a family is ``bad'' because successive applications of the Cauchy--Schwarz and van der Corput inequalities would not allow us to separate the functions corresponding to these polynomials.} and that the polynomial with index $1$ continues to have the highest degree (which is essential for the function originally evaluated at $\bx + \p_1(z)$ not to disappear at an intermediate step of PET). 
            The general form of the polynomials $\bgamma_{ji}$ ensures that the coefficients of the $\q_j$'s and their pairwise differences are bounded multiples of the coefficients of the $\p_j$'s and their differences. The property \eqref{i: w_1} guarantees that we do not lose information about the coefficients of $\p_1$ while carrying out PET.
            The multilinearity of the polynomials $\bgamma_{1d_{1j}}-\bgamma_{jd_{1j}}$ given in \eqref{i: multilinear} will be crucial in our concatenation argument since, as we show in Section \ref{S: equidistribution}, bounding the number of solutions to systems of multilinear equations is \textit{much} simpler than obtaining bounds for general systems of polynomial equations. Lastly, property \eqref{i: leading coeffs} ensures that the coefficients of $\bgamma_{1d_{1j}}-\bgamma_{jd_{1j}}$ come from only the \textit{leading} coefficients of $\p_1, \p_1-\p_2, \ldots, \p_1-\p_\ell$, and so the end product of PET will not depend on the lower-degree coefficients.

      We note several consequences of Definition \ref{D: descendence}: 
      \begin{enumerate}
	      \item If $\CQ$ descends from $\CP$, then the formulas \eqref{E: q_j} and \eqref{E: gamma_ji} jointly ensure that the polynomials in $\CQ$ cannot have their total degree in $z, h_1, \ldots, h_r$ greater than $d$.
	       \item $\CP$ descends from itself: Since $r=0$, the only choice of $\uu$ is $\underline{0}$, and one can set $w_{j\underline{0}}=j$ and $w_{\underline{0}} = 0$ for each $j\in[\ell]$.
	      \item The multilinearity in \eqref{i: multilinear} tells us that only $\uu$'s in $\{0,1\}^r$ can contribute to $\bgamma_{1d_{1j}}-\bgamma_{jd_{1j}}$, so the multinomial coefficients simplify and we have
          \begin{align*}
              \bgamma_{1d_{1j}}(\uh)-\bgamma_{jd_{1j}}(\uh) = \sum_{\substack{\uu\in \{0,1\}^r,\\ |\uu|\leq d-d_{1j}}} \frac{(|\uu|+d_{1j})!}{d_{1j}!}\cdot (\bbeta_{1(|\uu|+d_{1j})}-\bbeta_{w_{j\uu}(|\uu|+d_{1j})})\uh^\uu.
          \end{align*}
          \item\label{i: homogeneity} As a special case of \eqref{i: leading coeffs}, if the leading coefficient of $\bp_1$ differs from the leading coefficients of $\bp_2, \ldots, \bp_\ell$, then for each $j$, the polynomial $\bgamma_{1d_{1j}}-\bgamma_{jd_{1j}}$ is a multilinear, homogeneous polynomial of degree $d-d_{1j}$,
        and the expression for the polynomials $\bgamma_{1d_{1j}}-\bgamma_{jd_{1j}}$ further simplifies to
          \begin{align*}
              \bgamma_{1d_{1j}}(\uh)-\bgamma_{jd_{1j}}(\uh) = \sum_{\substack{\uu\in \{0,1\}^r,\\ |\uu|= d-d_{1j}}} \frac{d!}{d_{1j}!}\cdot (\bbeta_{1d}-\bbeta_{w_{j\uu}d})\uh^\uu.              
          \end{align*}
          \item If the leading coefficient of $\p_1$ is the same as the leading coefficient of some $\p_j$ for $2 \leq j \leq \ell$, then the polynomial $\bgamma_{1d_{1j}}-\bgamma_{jd_{1j}}$ ceases to be homogeneous and of degree $d-d_{1j}$; for instance, if $\CP = (z^2+z, z^2)$, then $\partial_2 \CP = ((2h+1)z, 2h z, z)$ descends from $\CP$, but $2h+1$ fails to be homogeneous and {the constant polynomial $1$ occurring as the coefficient of the polynomial $z$ in $\partial_2 \CP$} has degree $0<1=d-d_{12}$.
      \end{enumerate}

The descendent family obtained during the van der Corput operation admits more properties: For instance, the coefficients $w_{j\uu}, w_\uu$ can be shown to satisfy $w_{j\uu} = w_{j\uu'}$ and $w_\uu = w_{\uu'}$ whenever $\supp(\uu) = \supp(\uu')$ (here, $\supp(\uu):=\{i\in[r]:\; u_i>0\}$; this property has been observed in \cite[Definition 5.5 (P3)]{DKS22} and used in this paper in the proof of \cite[Proposition 5.7]{DKS22}). However, the use of this observation has been superseded in our paper by the use of property \eqref{i: multilinear} from Definition \ref{D: descendence}; this property previously appeared implicitly in \cite{Pel20, Ts22}.

    The following lemma ascertains that the van der Corput operation preserves the abovementioned properties inherited by $\CQ$ from $\CP$ by virtue of descending from $\CP$.
    \begin{proposition}\label{P: vdC preserves descendancy}
        Let $d, D, \ell,r, s\in\N$, and let $\CP = (\p_1, \ldots, \p_\ell)$ be a normal family of polynomials in $\Z^D[z]$ with coefficients given by $\bp_j(z) = \sum_{i=1}^d \bbeta_{ji}z^i$.
        Suppose that the family $\CQ = (\q_1, \ldots, \q_s)$ of polynomials in $\Z^D[z,\uh]$ descends from $\CP$. Then the family $\partial_m \CQ$ also descends from $\CP$ for each $m\in[s]$ for which $\sigma_{h_{r+1}}{\q}_1-{\q}_m$
        has the highest degree (as a polynomial in $z$) among the elements of $\partial_m \CQ$.
    \end{proposition}
 
    \begin{proof}
        We first note that the family $\partial_m\CQ$ is normal: The application of the $*$ operation ensures that the polynomials of the family are nonzero and distinct, and the assumption in the proposition ensures that the first polynomial in the family has the highest degree (in $z$).

        We now check that $\partial_m\CQ$ satisfies the three properties of a descendent family.
        Since $\CQ$ is a descendent family, the polynomials $\bq_j$ satisfy the formulas \eqref{E: q_j} and \eqref{E: gamma_ji} for some indices $w_{j\uu}, w_\uu$ with $w_{1\uu}=1$.
        Using the binomial theorem, we can expand
        \begin{align*}
            \bq_j(z+h_{r+1}) &= \sum_{i=0}^d \sum_{\substack{\uu\in \N_0^r,\\ |\uu|\leq d-i}} {{{u_1+\cdots + u_r+i}\choose{u_1, \ldots, u_r}}}(\bbeta_{w_{j\uu}(|\uu|+i)}-\bbeta_{w_{\uu}(|\uu|+i)})\uh^\uu (z+h_{r+1})^i\\
            &=\sum_{i=0}^d \sum_{\substack{\uu\in \N_0^r,\\ |\uu|\leq d-i}} {{{u_1+\cdots + u_r+i}\choose{u_1, \ldots, u_r}}}(\bbeta_{w_{j\uu}(|\uu|+i)}-\bbeta_{w_{\uu}(|\uu|+i)})\uh^\uu \sum_{l=0}^i{{{i}\choose{l}}}h_{r+1}^{i-l}z^l.
        \end{align*}
        Changing the order of summation of $i$ and $l$,  setting $u_{r+1} := i-l$, and using the identity
        \begin{align*}
            {{{u_1+\cdots +u_{r+1}+l}\choose{u_1, \ldots, u_{r+1}}}} = {{{u_1+\cdots + u_r+u_{r+1}+l}\choose{u_1, \ldots, u_r}}}{{{u_{r+1}+l}\choose{l}}},
        \end{align*}
        we obtain
        \begin{multline*}
            \bq_j(z+h_{r+1}) = \sum_{l=0}^d z^l \sum_{\substack{(\uu, u_{r+1})\in \N_0^{r+1},\\ |\uu|+u_{r+1}\leq d-l}} {{{u_1+\cdots + u_{r+1}+l}\choose{u_1, \ldots, u_{r+1}}}}\\ (\bbeta_{w_{j\uu}(|\uu|+u_{r+1}+l)}-\bbeta_{w_{\uu}(|\uu|+u_{r+1}+l)})\uh^\uu h_{r+1}^{u_{r+1}}.
        \end{multline*}
        Hence 
        \begin{multline}\label{E: q_j - q_m}
             \bq_j(z+h_{r+1})-\bq_m(z) = \sum_{l=0}^d z^l \sum_{\substack{(\uu, u_{r+1})\in \N_0^{r+1},\\ |\uu|+u_{r+1}\leq d-l}} {{{u_1+\cdots + u_{r+1}+l}\choose{u_1, \ldots, u_{r+1}}}}\\ (\bbeta_{w_{j(\uu, u_{r+1})}(|\uu|+u_{r+1}+l)}-\bbeta_{w_{(\uu, u_{r+1})}(|\uu|+u_{r+1}+l)})\uh^\uu h_{r+1}^{u_{r+1}},
        \end{multline}
        where we set $w_{j(\uu, u_{r+1})} := w_{j\uu}$ and 
        \begin{align*}
            w_{(\uu, u_{r+1})} := \begin{cases} w_\uu,\; &u_{r+1}>0,\\
            w_{m\uu},\; &u_{r+1} = 0. 
            \end{cases}
        \end{align*}
      In particular, we have $w_{1(\uu, u_{r+1})} = w_{1\uu}= 1$ by the induction hypothesis, and so $\partial_m \CQ$ satisfies property \eqref{i: w_1}.

        To show properties \eqref{i: multilinear} and \eqref{i: leading coeffs},
        we need to compare $\sigma_{h_{r+1}}{\q}_1(z,\uh)-{\q}_m(z,\uh)$ with other polynomials in $\partial_m \CQ$ and examine the leading coefficients of these differences (as polynomials in $z$).
       There are three cases to investigate. First, we check \eqref{i: multilinear} and \eqref{i: leading coeffs} for differences of the form
        \begin{multline}\label{E: first case PET}
            (\sigma_{h_{r+1}}{\q}_1(z,\uh)-{\q}_m(z,\uh)) - (\sigma_{h_{r+1}}{\q}_j(z,\uh)-{\q}_m(z,\uh))\\
            = \sum_{i=1}^{d_{1j}}(\bgamma_{1d_{1j}}(\uh)-\bgamma_{jd_{1j}}(\uh))(z+h_{r+1})^i - (\q_1(h_{r+1},\uh) - \q_j(h_{r+1},\uh)),
        \end{multline}
        where the terms ${\q}_m(z,\uh)$ cancel out.
        The leading coefficient of this expression (interpreted as a polynomial in $z$) is $\bgamma_{1d_{1j}}(\uh)-\bgamma_{jd_{1j}}(\uh)$, the same as the leading coefficient of the polynomial $\q_1 - \q_m$, and this is multilinear by assumption.  For \eqref{i: leading coeffs}, we need to show that the nonzero vectors 
        $\bbeta_{1(|\uu|+ u_{r+1}+d_{1j})}-\bbeta_{w_{j(\uu, u_{r+1})}(|\uu|+u_{r+1}+d_{1j})}$ 
        are leading coefficients of the polynomials $\bp_1-\bp_{w_{j(\uu, u_{r+1})}}$. If $u_{r+1}=0$, then this follows immediately from the formula 
        \begin{align*}
            \bbeta_{1(|\uu|+ u_{r+1}+d_{1j})}-\bbeta_{w_{j(\uu, u_{r+1})}(|\uu|+u_{r+1}+d_{1j})} = \bbeta_{1(|\uu|+d_{1j})}-\bbeta_{w_{j\uu}(|\uu|+d_{1j})}
        \end{align*}
        and the assumption on $\CQ$, while if $u_{r+1} = 1$, then 
        \begin{align*}
            \bbeta_{1(|\uu|+ u_{r+1}+d_{1j})}-\bbeta_{w_{j(\uu, u_{r+1})}(|\uu|+u_{r+1}+d_{1j})} = \bbeta_{1(|\uu|+d_{1j}+1)}-\bbeta_{w_{j\uu}(|\uu|+d_{1j}+1)}=\mathbf{0}
        \end{align*}
        because the leading coefficient $\bgamma_{1d_{1j}}(\uh)-\bgamma_{jd_{1j}}(\uh)$ is independent of $h_{r+1}$.

        Second, we let $d_1$ denote the degree of $\q_1$ in $z$ and expand
    \begin{align}\label{E: second case PET}
        &(\sigma_{h_{r+1}}{\q}_1(z,\uh)-{\q}_m(z,\uh)) - ({\q}_1(z,\uh)-{\q}_m(z,\uh))\\ 
        \nonumber &\qquad\qquad\qquad\qquad\qquad= \sum_{i=1}^{d_1}\bgamma_{1i}(\uh)\brac{(z+h_{r+1})^i-z^i}-\bq_1(h_{r+1}, \uh)\\
        \nonumber &\qquad\qquad\qquad\qquad\qquad= \sum_{i=1}^{d_1}\sum_{l=1}^{i-1}\bgamma_{1i}(\uh){{{i}\choose{l}}}h_{r+1}^{i-l}z^l\\
        \nonumber &\qquad\qquad\qquad\qquad\qquad= \sum_{l=1}^{d_1-1}z^l \sum_{i=l+1}^{d_1}\bgamma_{1i}(\uh){{{i}\choose{l}}}h_{r+1}^{i-l};
    \end{align}
        the leading coefficient here is $d_1h_{r+1} \bgamma_{1d_1}(\uh)$, corresponding to $l=d_1-1, i=d_1$. The polynomial $\bgamma_{1d_1}(\uh)$ is the leading coefficient of $\bq_1$, so it is multilinear by assumption, and hence $d_1h_{r+1} \bgamma_{1d_1}(\uh)$ is multilinear as well. Thus \eqref{i: multilinear} holds. To show property \eqref{i: leading coeffs}, we first observe that the coefficient of $\uh^\uu$ in $\bgamma_{1d_1}(\uh)$ is $\frac{(|\uu|+d_1)!}{d_1!} (\bbeta_{1(|\uu|+d_1)}-\bbeta_{w_{\uu}(|\uu|+d_1)})$ (where we use implicitly property \eqref{i: w_1}). Therefore, the coefficient of $\uh^\uu h_{r+1}$ in $d_1h_{r+1} \bgamma_{1d_1}(\uh)$ is
        \begin{multline*}
            \frac{(|\uu|+d_1)!}{(d_1-1)!} (\bbeta_{1(|\uu|+d_1)}-\bbeta_{w_{\uu}(|\uu|+d_1)})\\
            = \frac{(|(\uu,1)|+(d_1-1))!}{(d_1-1)!} (\bbeta_{1(|(\uu,1)|+(d_1-1))}-\bbeta_{w_{(\uu,1)}(|(\uu,1)|+(d_1-1))}).
        \end{multline*}
        By the induction hypothesis, if $\bbeta_{1(|\uu|+d_1)}-\bbeta_{w_{\uu}(|\uu|+d_1)}$ is nonzero, then it is the leading coefficient of $\p_1-\p_{w_{\uu}}$, and hence
        \begin{align*}
            \bbeta_{1(|(\uu,1)|+(d_1-1))}-\bbeta_{w_{(\uu,1)}(|(\uu,1)|+(d_1-1))} = \bbeta_{1(|\uu|+d_1)}-\bbeta_{w_{\uu}(|\uu|+d_1)}
        \end{align*}
        is also nonzero and the leading coefficient of $\p_1 - \p_{w_{(\uu,1)}} = \p_1 - \p_{w_{\uu}}$.
        
        It remains to consider
        \begin{align*}
            (\sigma_{h_{r+1}}{\q}_1(z,\uh)-{\q}_m(z,\uh)) - ({\q}_j(z,\uh)-{\q}_m(z,\uh))
        \end{align*}
        {for $j \neq 1$}; we split this case into two subcases. If $\bgamma_{1 d_1}(\uh)\neq \bgamma_{jd_1}(\uh)$, then the leading coefficient equals $\bgamma_{1d_{1j}}(\uh)-\bgamma_{jd_{1j}}(\uh) = \bgamma_{1 d_1}(\uh)-\bgamma_{jd_1}(\uh)$. This is the leading coefficient of $\q_1 - \q_j$, so it is multilinear by assumption, and property \eqref{i: leading coeffs} holds for the same reason as it held for \eqref{E: first case PET}.
        
        Otherwise, $\bgamma_{1 d_1}(\uh)=\bgamma_{jd_1}(\uh)$, in which case the leading coefficient equals 
        \begin{align}\label{E: leading coeff, third case}
            d_1h_{r+1} \bgamma_{1d_1}(\uh) + \bgamma_{1 (d_1-1)}(\uh)-\bgamma_{j(d_1-1)}(\uh),
        \end{align}
        where 
        \begin{align*}
            \bgamma_{1 (d_1-1)}(\uh)-\bgamma_{j(d_1-1)}(\uh) = \begin{cases} \bgamma_{1d_{1j}}(\uh)-\bgamma_{jd_{1j}}(\uh),\; &\textrm{if}\quad d_{1j} = d_1-1,\\
            \mathbf{0},\; &\textrm{if}\quad d_{1j}<d_1-1.
            \end{cases}
        \end{align*}
        The term $d_1h_{r+1} \bgamma_{1d_1}(\uh)$ is multilinear by the same argument as in the previous case.  The term $\bgamma_{1 (d_1-1)}(\uh)-\bgamma_{j(d_1-1)}(\uh)$ is trivially multilinear when $d_{1j}<d_1-1$, and its multilinearity follows from the assumption on $\CQ$ when $d_{1j} = d_1-1$ (since then it is the leading coefficient of $\q_1 - \q_j$). Hence \eqref{E: leading coeff, third case} is multilinear as well. 

        It remains only to prove property \eqref{i: leading coeffs} in this very last case, namely, to show that if the coefficient of the monomial $\uh^\uu h_{r+1}^{u_{r+1}}$ in \eqref{E: leading coeff, third case} is nonzero, then it is an appropriate scalar multiple of the leading coefficient of the polynomial $\p_1 - \p_{w_{j(\uu, u_{r+1})}} = \p_1 - \p_{w_{j\uu}}$. If $u_{r+1}=1$, then the coefficient of $\uh^\uu h_{r+1}^{u_{r+1}}$ in \eqref{E: leading coeff, third case} is the same as in $d_1h_{r+1} \bgamma_{1d_1}(\uh)$, and the claim follows by the same reasoning as in the case of \eqref{E: second case PET}. Otherwise, $u_{r+1}=0$, and the coefficient $\uh^\uu h_{r+1}^{u_{r+1}}$ in \eqref{E: leading coeff, third case} comes fully from $\bgamma_{1d_{1j}}(\uh)-\bgamma_{jd_{1j}}(\uh)$. The result then follows by the same argument as one used to prove property \eqref{i: leading coeffs} for \eqref{E: first case PET}.
    \end{proof}

    The next result ensures that we reach a family of linear polynomials after a bounded number of van der Corput operations; this fact is the foundation of all PET induction schemes. 
    \begin{proposition}\label{P: vdC terminates}
        Let $d, D, r, s\in\N$, let $\CP$ be a normal family of  polynomials in $\Z^D[z]$ of degree at most $d$, and let $\CQ$ be a normal family of polynomials in $\Z^D[z, \uh]$ of length $s$ that descends from $\CP$. There exist a nonnegative integer $r'=O_{d,s}(1)$ and positive integers $m_1, \ldots, m_{r'}$ such that the family $\CQ_i := \partial_{m_i}\cdots\partial_{m_1}\CQ$ descends from $\CP$ for each $i\in[r']$ and, moreover, the elements of $\CQ_{r'}$ are linear in $z$.
    \end{proposition}
    We remark that $r'$ is independent of the dimension $D$ of the underlying space.
    \begin{proof}
        Let $\CQ:=(\q_1, \ldots, \q_s)$. As usual, we expand each $\q_j$ as in \eqref{E: q_j} and say that its leading coefficient is $\bgamma_{ji}\in\Z^D[\uh]$ for the largest $i\in[d]$ with $\bgamma_{ji}\neq\mathbf{0}$. Let $d'$ be the maximum degree of $\bq_j, \ldots, \bq_s$ as a polynomial in $z$; since $\CQ$ descends from $\CP$ and is normal, necessarily $1\leq d'\leq d$.  For $l\in[d']$, let $w_l$ denote the number of distinct leading coefficients of the {degree-$l$} polynomials  in $\CQ$. We then define $w=(w_1, \ldots, w_{d'})$ to be the \textit{type} of $\CQ$, and we order the types colexicographically.\footnote{That is, we have $w>w'$ if $w_l>w'_l$ for the largest $l$ with $w_l \neq w'_l$.} 
        
        Since the family $\CQ$ is normal and has positive length, we have $w > (0, \ldots, 0)$; let $l$ be the smallest index for which $w_l\neq 0$. We define a parameter $m_1$ as follows.
        If $l<d'$, then we choose an arbitrary $m_1\in[s]$ with $\deg \bq_{m_1} = l$. If $l=d'$ (i.e., all of the polynomials in $\CQ$ have degree $d'$ as polynomials in $z$) and $w_{d'} > 1$, then we choose an arbitrary $m_1\in[s]$ such that $\bq_{m_1}$ and $\bq_1$ have different leading coefficients. Finally, if $l=d'$ and $w_{d'} =1$ (i.e., all of the polynomials have the same leading term),  then we choose ${m_1} = 1$. A standard calculation shows that in each case the new family $\CQ_1 = \partial_{m_1}\CQ$ has type $$w' = (w'_1, \ldots, w'_{l-1}, w_l-1, w_{l+1}, \ldots, w_{d'})<w,$$ for some $0\leq w'_1, \ldots, w'_{l-1}\leq 2s$, and the polynomial $\partial_h \bq_1 - \bq_{m_1}$ has the maximum degree among the polynomials in $\CQ_1$. It follows from Proposition \ref{P: vdC preserves descendancy} that $\CQ_1$ descends from $\CP$.
        
We now iterate this algorithm.  Since each decreasing sequence of types $w^0>w^1 > w^2 > \cdots$ with $w^0 = w$ eventually reaches a type $w^{r'} = (w^{r'}_1, 0, \ldots, 0)$ for some $r'=O_{d',s}(1)=O_{d,s}(1)$ 
and $w^{r'}_1=O_{r',s}(1)=O_{d,s}(1)$, the algorithm terminates after $O_{d,s}(1)$ steps, and we obtain the conclusion of the proposition.
\end{proof}

We remark that in the existing literature, the type of a polynomial family (as defined in the proof of Proposition \ref{P: vdC terminates}) is often expressed as a matrix rather than a (row) vector, with different coordinates in $\Z^D$ giving rise to different rows in the matrix (see e.g. \cite{BL96, CFH11, DKS22}). This bookkeeping is excessive, however; as the above proof illustrates, the type-reduction proceeds in the same way regardless of $D$, and the number of steps does not depend on $D$. The argument is even more robust: The number of steps would in fact be the same if we considered polynomial families over any integral domains with characteristic either equal to $0$ or sufficiently large in terms of the degrees of the polynomials and the order of their coefficients.\footnote{In the small-characteristic case, there may be additional degeneracies. For instance, a term $\beta_d z^d$ may reduce eventually to a term $\beta_d d z^{d-1}h$, and if the underlying ring has characteristic dividing $\beta_d d$, then this term vanishes, which requires a more refined treatment. Other complications arise if the underlying ring has zero divisors. See \cite{BB23} for an example of a PET argument dealing with the latter issue.}

        In the example discussed in Section \ref{SS: example of PET}, the original family has type $(0,2)$ (or $(1,1)$ if $\bbeta_{22}=\textbf{0}$; we ignore this case for now). The new families obtained in steps 1, 2, and 3 have types $(1,1)$,
        $(0,1)$ and $(7,0)$ (respectively).

    \subsection{Bounds by averages of box seminorms}\label{SS: PET bound}
    Having presented the PET induction scheme in detail, we are ready now to show that the counting operator \eqref{E: counting operator} is controlled by an average of box norms whose directions can be described in terms of the coefficients of the polynomials appearing in the multidimensional polynomial progression and the coefficients of differences of these polynomials. We no longer assume that $\p_1$ has the highest degree among $\p_1, \ldots, \p_\ell$.
\begin{proposition}[PET bound for polynomials]\label{P: PET}
			Let $d, \ell,D\in\N$ and $T>0$. There exist positive integers $r, s=O_{d, \ell}(1)$ and a positive real $C=O_{d,\ell,D,T}(1)$ such that the following holds.  Let $\delta>0$ and $H, K, M, N\in\N$ satisfy $H,M \leq T\delta^{-T}K$.  Let $\p_1, \ldots, \p_\ell\in\Z^D[z]$ be essentially distinct polynomials  with degrees at most $d$ and coefficients $\p_j(z) = \sum_{i=0}^d \bbeta_{ji} z^i$.
   Then for any $1$-bounded functions $f_0, \ldots,  f_\ell:\Z^D\to\C$, the lower bound  
			\begin{align*}
            \abs{\sum_{\bx}\E_{z\in[K]}
            f_0(\bx)\cdot f_1(\bx+\p_1(z))\cdots f_\ell(\bx+\p_\ell(z))} \geq \delta N^D
			\end{align*}
   implies that
   \begin{align*}
       \E_{\uh\in[\pm H]^{r}}\norm{f_1}_{{\bc_1(\uh)}\cdot[\pm M], \ldots, {\bc_s(\uh)\cdot[\pm M]}}^{2^s} \geq \frac{1}{C} \delta^{C} N^D,
   \end{align*}
			where $\bc_1, \ldots, \bc_s\in\Z^D[\uh]$ are nonzero multilinear polynomials depending only on $\bp_1, \ldots, \bp_\ell$, and they take the form
			\begin{align}\label{E: polynomials c_j}
				\bc_{j}(\uh) = \sum_{\substack{\uu\in \{0,1\}^{r},\\ |\uu|\leq d-1}} (|\uu|+1)! \cdot (\bbeta_{1(|\uu|+1)}-\bbeta_{w_{j\uu}(|\uu|+1)})\uh^\uu
			\end{align}
                for some indices $w_{j\uu}\in[0,\ell]$ (with $\bbeta_{0(|\uu|+1)}:=\mathbf{0}$). Moreover, each nonzero vector $\bbeta_{1(|\uu|+1)}-\bbeta_{w_{j\uu}(|\uu|+1)}$ from \eqref{E: polynomials c_j} is the leading coefficient of $\p_1 - \p_{w_{j\uu}}$.
  \end{proposition}

We typically apply Proposition \ref{P: PET} with $K\asymp_d (N/V)^{1/d}$ for $V := \max_{j,i}|\bbeta_{ji}|$, as this provides the maximum possible range for which $\bp_j(z)\in[\pm N]^D$ for all $j\in[\ell]$ and $z\in[K]$. Since we do not place any special assumptions on $\p_1$, the role of $f_1$ in the conclusion can be replaced by any of the functions $f_0, \ldots, f_\ell$ (with an appropriate modification of the formulas for \eqref{E: polynomials c_j}). For $f_2, \ldots, f_\ell$, this is obvious by symmetry, whereas for $f_0$ this follows after shifting $\bx \mapsto\bx - \bp_m$ for some arbitrary $m\in[\ell]$ and observing (as in the proof of Proposition \ref{P: PET} below) that the resulting collections of polynomials \eqref{E: collections of polys} are identical.

  \begin{proof}
      We will proceed under the assumption that $\p_1$ has the largest degree among the $\p_j$'s; at the end we will describe the necessary modifications in the general case. 
      
      Proposition  \ref{P: vdC terminates} provides $r=O_{d,\ell}(1)$ and $m_1, \ldots, m_r\in\N$ such the families $\CQ_i = \partial_{m_i}\cdots\partial_{m_1}\CP$ all descend from $\CP$ and $\CQ_r$ is linear in $z$.  Let 
      \begin{align*}
            \b_{1}(\uh)z, \ldots, \b_{s}(\uh)z   
      \end{align*}
       be the elements of $\CQ_r$. Since the van der Corput operation at most doubles the length of a polynomial family, we have $s := |\CQ_{r}|\leq 2^{r}|\CP|\ll_{d, \ell} 1$. Applying Lemma \ref{L: bounding averages in vdC} inductively to each of $\CQ_0, \ldots, \CQ_{r-1}$ (with $\CQ_0 := \CP$), we deduce that 
    \begin{align*}
        \sum_{\uh\in\Z^r}\mu_H(\uh) {\sum_{\bx} \E_{z\in[K]}
        f_0(\bx)\cdot \prod_{j=1}^s f_j(\bx + \b_{j}(\uh)z) }\gg_{d, D, \ell} \delta^{O_{d, \ell}(1)} N^D.
    \end{align*}
The power of $\delta$ depends only on the number of iterations of Lemma \ref{L: bounding averages in vdC}; in particular, it is independent of $D$.
    
    Since the family $\CQ_r$ descends from $\CP$, the polynomials $\b_{1}, \ldots, \b_{s}$ are all distinct. Lemma \ref{L: linear averages} then implies that
       \begin{align*}
       \sum_{\uh\in\Z^{r}}\mu_H(\uh) \E_{\um, \um'\in[\pm M]^s}
       \sum_{\bx}\Delta'_{{\bc_1(\uh)}, \ldots, {\bc_s(\uh)}; (\um, \um')} f_1(\bx) \gg_{d, D, \ell} \delta^{O_{d, \ell}(1)} N^D,
   \end{align*}
   where $\bc_1, \ldots, \bc_s$ are the polynomials $\b_{1}, \b_{1}-\b_{2}, \ldots, \b_{1} - \b_{s}$. Since $\mu_H(\uh)\leq \frac{\mathbf{1}_{[-2H,2H]^r}(\uh)}{(4H+1)^r}$ pointwise and the inner average is nonnegative (by virtue of being an average of box norms), the inequality in the conclusion of the proposition follows after we replace $H$ with $H/2$.
    
    That the polynomials $\bc_1, \ldots, \bc_s$ are nonzero is a consequence of the distinctness of $\b_{1}, \ldots, \b_{s}$, which in turn is a consequence of the tuple $\CQ_r$ being normal. The form of $\bc_1, \ldots, \bc_s$ is then a consequence of the properties \eqref{i: w_1}-\eqref{i: leading coeffs} from Definition \ref{D: descendence}.

    We now address the case where $\p_1$ does not have the largest degree.  Then some other $\p_m$ has the largest degree, and we can translate $\bx\mapsto \bx-\p_m(z)$ so that
    \begin{multline*}
        \sum_{\bx}\E_{z\in[K]}
        f_0(\bx)\cdot f_1(\bx+\p_1(z))\cdots f_\ell(\bx+\p_\ell(z))\\
        = \sum_{\bx}\E_{z\in[K]} 
        f_0(\bx-\p_m(z))\cdot \prod_{\substack{j\in[\ell]}} f_j(\bx+\p_j(z)-\p_m(z)).
    \end{multline*}
    We set
    \begin{align*}
        \tilde{\p}_j = \begin{cases} \p_j - \p_m,\; &j\neq m,\\
        -\p_m,\; & j = m,
        \end{cases}
    \end{align*}
   so that $\tilde{\p}_1$ has the largest degree among the $\tilde{\p}_j$'s and we can apply the argument above (after exchanging $f_0$ and $f_m$).  The desired result now follows from the observation that
    \begin{align}\label{E: collections of polys}
        \{\p_1, \p_1-\p_j: j\in[2,\ell]\} \quad\textrm{and}\quad \{\tilde{\p}_1, \tilde{\p}_1-\tilde{\p}_j: j\in[2,\ell]\}
    \end{align}
    are identical as sets and hence these families of polynomials have the same leading coefficients. 
  \end{proof}

\section{Model example: concatenation of degree-$1$ box norms}\label{S: model}

To illustrate some of the ideas that go into the proof of Theorem \ref{T: concatenation of polynomials}, we start with a model concatenation result for averages of degree-$1$ box norms.

\begin{proposition}\label{P: concatenation 1-deg}
    Let $d, D, r\in\N$ with $d\leq r$, and let $T>0$.  There exists 
    a positive real $C=O_{r,D,T}(1)$ such that the following holds.  Let $\delta>0$ and $H,M,N\in\N$ satisfy
    \begin{align}\label{E: conditions in 1-deg concatenation}
        C\delta^{-C} \leq H\leq M \quad \textrm{and}\quad H^d M \leq N,
    \end{align}
    and let $\bc(\uh)= \sum\limits_{\uu\in\{0,1\}^r}\bgamma_\uu \uh^\uu\in\Z^D[h_1, \ldots, h_{r}]$ be a multilinear polynomial of degree $d$ with coefficients of size at most $T$.  Then for all $1$-bounded functions $f:\Z^D\to\C$ supported on $[N]^D$, the bound
    \begin{align}\label{E: lower bound in 1-deg concatenation}
         \E_{\uh\in[\pm H]^r}\norm{f}_{\bc(\uh)\cdot[\pm M]}^2\geq \delta N^D
    \end{align}
    implies that
    \begin{align*}
        \norm{f}_{[\pm N]^D, E}^4\geq \frac{1}{C} \delta^{C}N^D,
    \end{align*}
    where
    \begin{align}\label{E: box E}
        E = \sum_{\uu\in\{0,1\}^r}\bgamma_\uu\cdot[\pm H^{|\uu|}M].
    \end{align}
\end{proposition}
We observe that the multiset $E$ is a generalized arithmetic progression spanned by the coefficients $\bgamma_\uu$, and that the length of the interval corresponding to the coefficient $\bgamma_\uu$ is of the same order of magnitude as the maximum value of the monomial $\uh^\uu m$. Moreover, if $\bgamma_\uu = \textbf{0}$, i.e., the monomial $\uh^\uu$ does not appear in the polynomial $\bc$, then the corresponding term does not contribute to the set $E$.

\begin{example}\label{E: example of 1-deg concatenation}
To illustrate the content of Proposition \ref{P: concatenation 1-deg} with a specific example, we consider the polynomial $$\bc(h_1, h_2) = \be_1h_1 h_2 + \be_2 h_1 = (h_1 h_2, h_1),$$
where $D = 2$, $H=M=N^{1/3}$; for simplicity, we take $N$ to be a perfect cube.
Proposition \ref{P: concatenation 1-deg} tells us that if
\begin{align*}
    \sum_{x_1, x_2, m} \E_{h_1, h_2\in[\pm N^{1/3}]} \mu_{N^{1/3}}(m) f(x_1, x_2) \overline{f(x_1 + h_1h_2m, x_2 +h_1 m)}\geq \delta N^2,
\end{align*}
or equivalently
\begin{align*}
    \sum_{\bx, m}\E_{\uh\in[\pm N^{1/3}]^2} \mu_{N^{1/3}}(m) f(\bx)\overline{f(\bx + (\be_1h_1h_2 + \be_2 h_1)m)}\geq \delta N^2,
\end{align*}
then
\begin{align*}
    \sum_{\bx}\E_{\bn_1, \bn'_1\in[\pm N]^2}\E_{\bn_2, \bn'_2\in[\pm N]\times [\pm N^{2/3}]} \Delta'_{(\bn_1,\bn_1'),\; (\bn_2, \bn'_2)}f(\bx)\gg \delta^{O(1)}N^{2}.
\end{align*}
Writing $\bn_2 = (n_{21}, n_{22})$ and $\bn_2'=(n_{21}', n_{22}')$, we see that the variables $n_{21}, n_{21}'$ correspond to the monomial $h_1 h_2$ (equivalently, to the coefficient $\uu = (1,1)$) and hence run along the coefficient $\be_1$ of $h_1h_2$ in $\bc$; the variables $n_{22}, n_{22}'$ correspond to the monomial $h_1$ (equivalently, to $\uu = (1,0)$) and therefore run along the coefficient $\be_2$ of $h_1$ in $\bc$.
\end{example}

In applications, we usually want to discard the directions that correspond to lower-degree terms in the original polynomial, such as the direction $\be_2$ corresponding to the index $(1,0)$ in the example above. We accomplish this using Lemma \ref{L: properties of box norms}, and obtain the following streamlined version of Proposition \ref{P: concatenation 1-deg}.
\begin{corollary}\label{C: concatenation 1-deg}
   Let $d, D, r\in\N$ with $d\leq r$, and let $T>0$. There exists 
   a positive real $C=O_{r,D,T}(1)$ such that the following holds. Let $N\in\N$ and $\delta>0$ satisfy $N \geq C \delta^{-C}$, and let $\bc(\uh)= \sum\limits_{\uu\in\{0,1\}^r}\bgamma_\uu \uh^\uu\in\Z^D[h_1, \ldots, h_{r}]$ be a multilinear polynomial of degree $d$ with coefficients of size at most $T$.  Then for all $1$-bounded functions $f:\Z^D\to\C$ supported on $[N]^D$, the lower bound 
   \begin{align*}
        \E_{\uh\in[\pm N^{1/(d+1)}]^r}\norm{f}_{\bc(\uh)\cdot[\pm N^{1/(d+1)}]}^2\geq \delta N^D
   \end{align*}
   implies that
    \begin{align*}
        \norm{f}_{(\bgamma_\uu\cdot[\pm N])^2}^4\geq \frac{1}{C}\delta^{C}N^D
    \end{align*}
    for each $\uu\in\{0,1\}^r$ with $|\uu| = d$ {and $\bgamma_\uu \neq \textbf{0}$}.
\end{corollary}
\begin{proof} We let all the constants depend on $r, D, T$ (and $d$, which however is bounded in terms of $r$).
Proposition \ref{P: concatenation 1-deg} applied with $H=M=N^{1/(d+1)}$ gives
    \begin{align*}
        \norm{f}_{[\pm N]^D, E}^4\gg \delta^{O(1)}N^D,
    \end{align*}
    where $E$ is as in \eqref{E: box E}. Since the coefficients $\bgamma_\uu$ have size $O(1)$, we can find vectors $\bv_1, \ldots, \bv_k\in\Z^D$ (with $k<D$) of size $O(1)$ such that the set
    \begin{align*}
        \{\bgamma_\uu:\; \uu\in\{0,1\}^r,\; |\uu|=d, \bgamma_\uu \neq \textbf{0}\}\cup \{\bv_1, \ldots, \bv_k\}
    \end{align*}
    generates a subgroup in $\Z^D$ of index $O(1)$. Hence there exists a positive real $C = O(1)$
    such that
    \begin{align*}
        C\cdot [\pm N/C]^D
        \subseteq \sum_{\substack{\uu\in\{0,1\}^r,\\ |\uu|=d}}\bgamma_\uu\cdot[\pm CN]+\sum_{i=1}^k \bv_i\cdot[\pm CN],
    \end{align*}
    and so by successively applying properties \eqref{i: passing to APs}, \eqref{i: enlarging}, \eqref{i: trimming} and \eqref{i: trimming 2} of Lemma \ref{L: properties of box norms}, we get
    \begin{align*}
        \norm{f}_{E', E}^4\gg \delta^{O(1)}N^D
    \end{align*}
    for 
    \begin{align*}
        E' = \sum_{\substack{\uu\in\{0,1\}^r,\\ |\uu|=d}}\bgamma_\uu\cdot[\pm N].
    \end{align*}
    To conclude from here, we pick a single $\uu$ and apply Lemma \ref{L: properties of box norms}\eqref{i: trimming}.  
\end{proof}

Corollary \ref{C: concatenation 1-deg} allows us to upgrade the conclusion of Example \ref{E: example of 1-deg concatenation} to 
\begin{align*}
    \sum_{\bx}\E_{n_1, n_1', n_2, n_2'\in[\pm N]} \Delta'_{(n_1\be_1,n'_1\be_1),\; (n_2\be_1, n_2'\be_1)} f(\bx)\gg \delta^{O(1)}N^{2}.
\end{align*}

\subsection{Sketch of the proof of Proposition \ref{P: concatenation 1-deg}}
We now outline the proof of Proposition \ref{P: concatenation 1-deg} in order to present the logic behind our concatenation argument in a relatively simple setting.  This is an expanded version of the overview given in Section \ref{S: strategy}. The argument gets considerably more complicated in the setting of box norms of degree greater than $1$, so the reader may find the simplified argument useful for grasping the main ideas. In the sketch, we use as a black box some technical results that will be derived in later sections.
 
Suppose that \eqref{E: lower bound in 1-deg concatenation} holds. Expanding the definition of the box norm gives us
\begin{align*}
    \E_{\uh\in[\pm H]^r}\E_{m, m'\in[\pm M]}\sum_{\bx} \Delta'_{(\bc(\uh) m, \bc(\uh) m')}f(\bx)\geq \delta N^D.
\end{align*}
Shifting $\bx \mapsto \bx-\bc(\uh)m$, replacing $m'-m$ by $m$, and changing the order of summation, we get
\begin{align*}
    \sum_\bx f(\bx) \E_{\uh\in[\pm H]^r}\sum_{m}\mu_M(m) \overline{f(\bx + \bc(\uh) m)}\geq \delta N^D.
\end{align*}

We will use several applications of the Cauchy--Schwarz inequality to replace the polynomial $\bc(\uh)m$ by one with better equidistribution properties: Specifically, we want to replace it by a polynomial in a larger number of variables that does not assume any particular value more than a constant multiple of the ``expected'' number of times. Each application of the Cauchy--Schwarz inequality, which increases the number of variables, can be understood as a ``smoothing'' of the polynomial differencing parameter.

\smallskip
\textbf{Step 1: Replacing $\bc(\uh)m$ by a polynomial with better equidistribution properties.}
\smallskip

Applying the Cauchy--Schwarz inequality in $\bx$, we double the $\uh, m$ variables and obtain
\begin{align*}
    \sum_\bx \E_{\uh_1, \uh_2\in[\pm H]^r}\sum_{m_1, m_2}\mu_M(m_1, m_2) f(\bx + \bc(\uh_1) m_1)\overline{f(\bx + \bc(\uh_2) m_2)}\geq \delta^2 N^D,
\end{align*}
then we shift $\bx\mapsto \bx-\bc(\uh_1)m_1$ to get
\begin{align*}
    \sum_\bx f(\bx) \E_{\uh_1, \uh_2\in[\pm H]^r}\sum_{m_1, m_2}\mu_M(m_1, m_2)\overline{f(\bx + \bc(\uh_2) m_2-\bc(\uh_1) m_1)}\geq \delta^2 N^D.
\end{align*}
Applying the Cauchy--Schwarz inequality in $\bx$ again gives
\begin{multline*}
    \sum_\bx \E_{\substack{\uh_{1}, \uh_{2},\\ \uh_{3}, \uh_{4}\in[\pm H]^r}}\sum_{\substack{m_{1}, m_{2},\\ m_{3}, m_{4}}}\mu_M(m_{1}, m_{2}, m_{3}, m_{4})\\
    f(\bx + \bc(\uh_{2}) m_{2}-\bc(\uh_{1}) m_{1})\overline{f(\bx + \bc(\uh_{4}) m_{4}-\bc(\uh_{3}) m_{3})} \geq \delta^4 N^D,
\end{multline*}
and, after changing variables $\bx\mapsto \bx - (\bc(\uh_{2}) m_{2}-\bc(\uh_{1}) m_{1})$, $m_2\mapsto -m_2, m_3\mapsto -m_3$, we obtain
\begin{multline*}
    \E_{\substack{\uh_{1}, \uh_{2},\\ \uh_{3}, \uh_{4}\in[\pm H]^r}}\sum_{\substack{m_{1}, m_{2},\\ m_{3}, m_{4}}}\mu_M(m_{1}, m_{2}, m_{3}, m_{4})\\
    \sum_x f(\bx)\overline{f\brac{\bx+\bc(\uh_{1}) m_{1}+\bc(\uh_{2}) m_{2} + \bc(\uh_{3}) m_{3}+\bc(\uh_{4}) m_{4}}}\geq \delta^{4} N^D.
\end{multline*}
Thus, at the cost of increasing the power of $\delta$ in our lower bound, we have replaced $\bc(\uh) m$ by the polynomial 
\begin{align*}\label{eq:first-CS-outcome}
    \sum\limits_{i\in[4]}\bc(\uh_i)m_{i} = \bc(\uh_{1}) m_{1}+\bc(\uh_{2}) m_{2} + \bc(\uh_{3}) m_{3}+\bc(\uh_{4}) m_{4},
\end{align*}
 which is more uniformly distributed in its range than the original polynomial $\bc(\uh) m$. 

Our strategy will be to fix values of the $\uh_i$'s and then obtain equidistribution as the $m_i$'s range.  However, the polynomial that we just obtained does not yet have all of the necessary equidistribution properties for carrying out this strategy.
The issue is roughly that when the components of the $\uh_i$'s are small or have large common factors, the polynomial will concentrate on values that are small or have ``highly-divisible'' components.  For instance, if $\bc(\uh) = \be_1  h_1 h_2+\be_2 h_1$ (as in Example \ref{E: example of 1-deg concatenation}), then the new polynomial
\begin{gather*}
   \be_1 (h_{11} h_{21} m_{1} + h_{12}h_{22}m_{2} + h_{13}h_{23}m_{3}+h_{14}h_{24}m_{4})\\
   + \be_2(h_{11} m_{1} + h_{12}m_{2} + h_{13}m_{3}+h_{14}m_{4})
\end{gather*}
is not always equidistributed in its ``natural'' range $ \be_1\cdot[\pm 4 H^2 M]+\be_2\cdot[\pm 4 H M]$.  Indeed, when the $h_{ij}$'s are small, so are the $\be_1$- and $\be_2$-coordinates.  More subtly, the $\be_2$-coordinate must be a multiple of $\gcd(h_{11},h_{12},h_{13},h_{14})$, and this leads to concentration when the $\gcd$ is large.  Analogously, the $\be_1$-coordinate is concentrated on a small subset of its natural range when $\gcd(h_{11}h_{21}, h_{12}h_{22},h_{13}h_{23},h_{14}h_{24})$ is large.

A simple fix for each of these problems separately would be to throw out a small set of ``bad'' tuples of $\uh_i$'s that prevents equidistribution in one coordinate or the other. However, the task of ensuring \textit{simultaneous} equidistribution of the two coordinates is complicated by the presence of $h_{1i}, m_i$'s in both coefficients. We eventually resolve this issue by combining two maneuvers. The first maneuver is duplicating the variables in both coordinates via several more applications of the Cauchy-Schwarz inequality; this will allow us to ``decouple'' the coefficients. The second maneuver is throwing out a few bad $\uh_i$'s to prevent the aforementioned issues of concentration.  (The notion of ``badness'' for $\uh_i$'s will look more complicated after we have carried out the first maneuver.)  Let us now turn to the details.


We use more applications of the Cauchy--Schwarz inequality to increase the number of the variables $h_{2i}, \ldots, h_{ri}$ and uniformize the polynomial further; after enough duplications of variables using the Cauchy--Schwarz inequality, we will be able to show that the resulting polynomial is equidistributed in its natural range for a generic choice of the $h$'s.
Writing $\uh_{i} = (h_{1i}, \Tilde{\uh}_{i})$, we now apply the Cauchy--Schwarz
inequality twice in $\bx$ and $h_{1i}$
to obtain
\begin{multline*}
    \E_{(h_{i_1})_{i_1\in [4]}\in[\pm H]^4}    \E_{(\tilde{\uh}_{i_1i_2})_{i_1,i_2\in [4]}\in[\pm H]^{16(r-1)}}
    \sum_{(m_{i_1i_2})_{i_1,i_2\in[4]}} \mu_M((m_{i_1i_2})_{i_1,i_2}) \\ \sum_\bx f(\bx)\overline{f\Bigbrac{\bx+\sum_{\substack{i_1,i_2\in[4]}}
    \bc(h_{1i_1}, \tilde{\uh}_{i_1i_2} )m_{i_1i_2}}}\geq \delta^{16} N^D.
\end{multline*}
We continue in this manner: Applying the Cauchy--Schwarz inequality twice in the variables $$\bx, h_{1i_1}, h_{2i_1 i_2}, \ldots, h_{(l-1)i_1\cdots i_{l-1}}$$ for each $l=3,4, \ldots, r$ (as in the $l=2$ step just illustrated; the $l=1$ step appeared at the very beginning of the proof), and bounding the Fej\'er kernels pointwise, we arrive at the bound
\begin{multline*}
    \E_{(h_{1i_1})_{i_1\in [4]}\in[\pm H]^4}
    \cdots \E_{(h_{ri_1\cdots i_r})_{i_1, \ldots,  i_r\in [4]}\in[\pm H]^{4^r}}\E_{(m_{i_1\cdots  i_r})_{i_1, \ldots, i_r\in[4]}\in[\pm 2M]^{4^r}}\\
    \Bigabs{\sum_\bx f(\bx)\overline{f\Bigbrac{\bx+\sum_{\substack{i_1,\ldots, i_r\in[4]}}\bc(h_{1i_1}, \ldots, h_{ri_1\cdots  i_r})m_{i_1\cdots  i_r}}}}\gg_r \delta^{4^r} N^D.
\end{multline*}

\smallskip
\textbf{Step 2: Applying the equidistribution properties of the new polynomial.}
\smallskip

We have now replaced $\bc(\uh)m$ by the massive polynomial 
\begin{multline}\label{E: massive polynomial}
    \sum_{\substack{i_1, \ldots, i_r\in [4]}}\bc(h_{1i_1}, \ldots, h_{ri_1\cdots  i_r})m_{i_1\cdots  i_r}\\
    = \sum_{i_1, \ldots, i_r\in[4]}\bgamma_\uu \sum_{\substack{i_1\in[4]}}h_{1i_1}^{u_1}
    \sum_{\substack{i_2\in[4]}}h_{2i_1i_2}^{u_2}\cdots \sum_{\substack{i_r\in[4]}} h_{ri_1\cdots  i_r}^{u_{r}} m_{i_1\cdots  i_r}.
\end{multline}
For instance, if $\bc(\uh) = \be_1  h_1 h_2+\be_2 h_1$ (as in Example \ref{E: example of 1-deg concatenation}), then the polynomial \eqref{E: massive polynomial} equals
\begin{align*}
    \be_1\sum_{i_1, i_2\in[4]}h_{1i_1}h_{2i_1i_2}m_{i_1i_2} +\be_2\sum_{i_1, i_2\in[4]}h_{1i_1}m_{i_1i_2}.
\end{align*}
The crucial point is that for a generic choice of $h$'s, the subpolynomials of \eqref{E: massive polynomial} corresponding to different $\uu$'s are simultaneously well distributed in their ranges.  This result on simultaneous equidistribution is one of the new inputs of this paper, and we outsource its exact statement and proof to Proposition \ref{P: systems of multilinear equations}. For now, however, let us briefly explain how this good distribution property of the polynomial \eqref{E: massive polynomial} arises from its ``multiscale'' structure. For a generic choice of $h_{ri_1\cdots  i_r}$, the polynomial $$\Bigbrac{\sum_{\substack{i_r\in[4]}} m_{i_1\cdots  i_r}, \sum_{\substack{i_r\in[4]}} h_{ri_1\cdots  i_r} m_{i_1\cdots  i_r}}$$ is well distributed in its range; this allows us to essentially replace both sums by new variables $t_{i_1\cdots  i_{r-1}0}$ and $t_{i_1\cdots  i_{r-1}1}$ ranging over $[\pm 4 M], [\pm 4HM]$ respectively. 
Then for a generic choice of $h_{(r-1)i_1\cdots i_{r-1}}$, the polynomial
\begin{align*}
    \Bigbrac{\sum_{i_{r-1}\in[4]}t_{i_1\cdots  i_{r-1}0}, \sum_{i_{r-1}\in[4]}t_{i_1\cdots  i_{r-1}1}, \sum_{i_{r-1}\in[4]}h_{(r-1)i_1\cdots  i_{r-1}} t_{i_1\cdots  i_{r-1}0}, \sum_{i_{r-1}\in[4]}h_{(r-1)i_1\cdots  i_{r-1}} t_{i_1\cdots  i_{r-1}1}}    
\end{align*}
 is well distributed in its range $[\pm 16 M]
\times [\pm 16HM]\times [\pm 16HM]
\times [\pm 16H^2M]$, and so on. (We could obtain the same conclusion by summing over the indices in triples rather than quadruples, but we work with quadruples because each application of the Cauchy--Schwarz inequality doubles the number of variables.)
Because of this iterative scheme, it suffices to to show that \eqref{E: massive polynomial} is well distributed in its range for $r=1$, and this can be accomplished using elementary counting arguments.

For $l\in[r]$ and $\eta>0$, define the set
\begin{multline*}
    \CH_{l,\eta} = \{(h_{li_1\cdots i_l})_{i_1, \ldots, i_l\in[4]}\in[\pm H]^{4^l}:\; |h_{li_1\cdots i_l}-h_{li''_1\cdots i''_l}|\geq \eta H,\\ \gcd(h_{li_1\cdots i_l}-h_{li''_1\cdots i''_l}, h_{li'_1\cdots i'_l}-h_{li''_1\cdots i''_l})\leq \eta\inv\\
    \textrm{for {all} distinct}\; (i_1, \ldots, i_l),\; (i'_1, \ldots, i'_l),\; (i''_1, \ldots, i''_l)\in[4]^l\}.
\end{multline*}
This is a set of tuples $(h_{li_1\cdots i_l})_{i_1, \ldots, i_l}$ whose coordinates are not too close to each other and do not have large common factors. It is not difficult to show that for a small $\eta>0$, a generic tuple in $[\pm H]^{4^l}$ belongs to this set.
Indeed, taking $\CH_l := \CH_{l,c\delta^{1/c}}$ for sufficiently small $c= c(r)>0$, we have by Lemma \ref{L: H_l} that all but a $O_r(\delta^{4^r})$-proportion of the elements of $[\pm H]^{4^l}$ lie in $\CH_l$, and hence
\begin{multline*}
    \E_{(h_{1i_1})_{i_1\in [4]}\in[\pm H]^4}
    \cdots \E_{(h_{ri_1\cdots i_r})_{i_1, \ldots,  i_r\in [4]}\in[\pm H]^{4^r}}\E_{(m_{i_1\cdots  i_r})_{i_1, \ldots, i_r\in[4]}\in[\pm 2M]^{4^r}}\\
    \prod_{l\in[r]}\mathbf{1}_{\CH_l(h_{(li_1\cdots i_l)_{i_1, \ldots, i_l}})}\cdot\Bigabs{\sum_\bx f(\bx)\overline{f\Bigbrac{\bx+\sum_{\substack{i_1,\ldots, i_r\in[4]}}\bc(h_{1i_1}, \ldots, h_{ri_1\cdots  i_r})m_{i_1\cdots  i_r}}}}\gg_r \delta^{4^r} N^D.
\end{multline*}
We have thus excluded the $h's$ whose coefficients are close to each other or have large common factors.

For $(n_\uu)_\uu\in\Z^{2^r}$, let
\begin{multline*}
    \CN((n_\uu)_{\uu}) = \E_{(h_{1i_1})_{i_1\in [4]}\in[\pm H]^4}
    \cdots \E_{(h_{ri_1\cdots i_r})_{i_1, \ldots, i_r\in [4]}\in[\pm H]^{4^r}}
    \E_{(m_{i_1\cdots  i_r})_{i_1, \ldots, i_r\in [4]}\in[\pm 2M]^{4^r}}\\ 
    \prod_{l\in[r]}\mathbf{1}_{\CH_l((h_{li_1\cdots i_l})_{i_1, \ldots, i_l})}\cdot \prod_{\uu\in\{0,1\}^{r}} \mathbf{1}\Bigbrac{\sum_{\substack{i_1\in [4]}}h_{1i_1}^{u_1}\cdots \sum_{\substack{i_r\in [4]}} h_{ri_1\cdots  i_r}^{u_{r}} m_{i_1\cdots  i_r} = n_\uu}, 
\end{multline*}
so that
\begin{align}\label{E: inserting K}
    \sum_{(n_\uu)_\uu}\CN((n_\uu)_\uu)\cdot\Bigabs{\sum_\bx f(\bx) \overline{f\Bigbrac{\bx + \sum_{\uu\in\{0,1\}^{r}}\bgamma_\uu n_\uu}}} \gg_r \delta^{4^r} N^D.
\end{align}
The quantity $\CN((n_\uu))$ is the normalized count of $\brac{\uh_{1i_1}, \ldots, \uh_{ri_1\cdots i_r}, \um_{i_1\cdots i_r}}$ such that for each $\uu\in\{0,1\}^{r}$, the subpolynomial associated with $\bgamma_\uu$ in \eqref{E: massive polynomial} equals $n_\uu$. The key equidistribution result on the polynomial \eqref{E: massive polynomial} is the assertion that the normalised count $\CN$ satisfies the pointwise bound
\begin{align}\label{E: pointwise bound on K}
    \CN((n_\uu)_\uu)\ll_r \delta^{-O_r(1)} M^{-2^r}H^{-r2^{r-1}}.
\end{align}
This bound is a special case of Proposition \ref{P: systems of multilinear equations} to be derived later on.

Let us explain why the bound \eqref{E: pointwise bound on K} is of the right order of magnitude.  Note first that $\CN((n_\uu)_\uu) = 0$ unless each $|n_\uu|\leq C_\uu H^{|\uu|} M$ for some $C_\uu>0$. Defining
    \begin{align*}
        E_0 = \left\{(n_\uu)_{\uu\in\{0,1\}^{r}}:\; n_\uu \in [-C_\uu H^{|\uu|}M, C_\uu H^{|\uu|}M],\; \uu\in\{0,1\}^{r} \right\},
    \end{align*}
we observe that
\begin{align*}
    |E_0|\asymp_{r} \prod_{\uu\in\{0,1\}^r} (H^{|\uu|}M) = M^{2^r} H^{r 2^{r-1}}
\end{align*}
and note that the exponents agree with those in \eqref{E: pointwise bound on K}. Combining these observations with \eqref{E: inserting K} and \eqref{E: pointwise bound on K}, we get
\begin{align*}
    \E_{(n_\uu)_\uu\in E_0}\Bigabs{\sum_\bx f(\bx) \overline{f\Bigbrac{\bx + \sum_{\uu\in\{0,1\}^{r}}\bgamma_\uu n_\uu}}} \gg_r \delta^{O_r(1)} N^D.
\end{align*}
Defining
    \begin{align*}
        E_1 =
        \sum_{\uu\in\{0,1\}^{r}}\bgamma_\uu\cdot[\pm C_\uu H^{|\uu|}M],
    \end{align*}
we have
\begin{align*}
    \E_{\bn\in E_1}\Bigabs{\sum_\bx f(\bx) \overline{f\brac{\bx + \bn}}} \gg_r \delta^{O_r(1)} N^D.
\end{align*}
Using the Cauchy--Schwarz inequality to remove the absolute value and recalling that $f$ is supported on $[N]^D$, we obtain
$$\E_{\bn\in E_1} \E_{\bs \in [\pm N]^D}\sum_\bx f(\bx) \overline{f(\bx+\bs)f(\bx+\bn)}f(\bx+\bs+\bn) \gg_r \delta^{O_r(1)} N^D.$$
Two further applications of the Cauchy--Schwarz inequality to double $\bn,\bs$ let us introduce Fej\'er kernels, and we get
$$\norm{f}^4_{E_1, [\pm N]^D} \gg_r \delta^{O_r(1)} N^D.$$
To conclude the proof, we use Lemma \ref{L: properties of box norms}\eqref{i: trimming 2} to replace $E_1$ by $E$.

\section{Concatenation of averages of higher-degree seminorms}\label{S: general concatenation}

In the previous section, we showed how to concatenate an average of degree-$1$ box norms with polynomial directions. The argument becomes significantly more complicated when we deal with higher-degree box norms. In the degree-$1$ setting, applying the Cauchy--Schwarz inequality (as in the proof of Proposition \ref{P: concatenation 1-deg}) did not cause the degree of the box norm to increase. This pleasant property ceases to hold in the setting of higher-degree box norms; rather, each application of the Cauchy--Schwarz and Gowers-Cauchy--Schwarz inequalities will involve passing to norms of much higher degree. The objective of the present section is to carry out this higher-degree argument, along the lines of the prior work of the second author~\cite{Kuc23}. As noted in the proof outline, this stage of the concatenation arguments works for any box norms, not necessarily those that come out of PET and involve polynomials. To emphasize the robustness of this part of the concatenation argument, we phrase the results in this section in a more abstract setting of an average of box norms on an arbitrary countable abelian group $G$.  We will thus deal with expressions of the form
\begin{align*}
    \E_{i\in I}\norm{f}_{H_{1i}, \ldots, H_{si}}^{2^s},
\end{align*}
where $I$ is a finite indexing set and 
$H_{1i}, \ldots, H_{si}\subseteq G$ are finite nonempty multisets for each $i \in I$. In our applications, the set $I$ will be $\uh\in[\pm H]^{r}$, and the {multisets} $H_{1i}, \ldots, H_{si}$ for $i\in I$ will be the sets $$\bc_1(\uh)\cdot[\pm M], \ldots, \bc_s(\uh)\cdot[\pm M]$$ for $\uh\in[\pm H]^r$.

The following lemma will serve as the basis for the concatenation argument. It is an adaptation of \cite[Lemma 2.2]{Kuc23} from $\F_p^D$ to more general countable abelian groups.
\begin{lemma}\label{L: concatenation lemma}
For every $s\in \mathbb{N}$, there is a positive real $C=O_s(1)$ such that the following holds.  Let $A, \delta>0$, 
    let $G$ be a countable abelian group, let $I$ be a finite indexing set, and let $H_i, K_{1i}, \ldots, K_{si}\subseteq G$ be finite multisets for each $i\in I$.  Suppose that $B \subseteq G$ is a set
    satisfying 
    \begin{align*}
        |\supp(B+B-B)|\leq A|B|.
    \end{align*}
Then for all $1$-bounded functions $f$ supported on $B$, the bound
    \begin{align*}
        {\E_{i\in I}\norm{f}_{H_i, K_{1i}, \ldots, K_{si}}^{2^{s+1}}}\geq \delta |B|
    \end{align*}
    implies that
    \begin{align*}
        \E_{i_0, i_1\in I}\norm{f}_{K_{1 i_0}, \ldots, K_{s i_0}, K_{1 i_1}, \ldots, K_{s i_1}, H_{i_1}-H_{i_0}}^{2^{2s+1}}\geq \frac{1}{C} A^{-C} \delta^{2^{2s+1}} |B|.
    \end{align*}
\end{lemma}
In applications, we will take $B = [N]^D$ and $A=3^D$ since $$\supp(B+B - B) = [2-N, 2N-1]^D.$$

\begin{proof}
Let $K_i:=K_{1i}\times \cdots \times K_{si}$.
    The alternative formula \eqref{E: box norms with Fejer kernels} for box norms yields
    \begin{align*}
        \E_{i\in I} \sum_x \sum_{\substack{\uk \in K_i}}\sum_{h\in H_i}\mu_{K_i}(\uk) \mu_{H_i}(h) \Delta_{\uk, h}f(x) \geq \delta |B|.
    \end{align*}
    We expand
    \begin{align*}
        \Delta_{\uk, h} f(x) = f(x) \overline{f(x+h)} \Delta_{\uk}^*\Delta_{h} f(x), 
    \end{align*}
    where $\Delta_{\uk}^* f(x) = \prod_{\ueps\in\{0,1\}^s\setminus\{0\}}\CC^{|\ueps|}f(x+\ueps\cdot\uk)$,
    and change 
    the order of summation to get
    \begin{align*}
        \sum_{x} f(x) \E_{i\in I} \sum_{\substack{\uk \in K_i}}\sum_{h\in H_i}\mu_{K_i}(\uk) \mu_{H_i}(h) \overline{f(x+h)} \Delta^*_{\uk}\Delta_{h} f(x)\geq \delta |B|.
    \end{align*}
    An application of the Cauchy--Schwarz inequality in $x$ gives
    \begin{align*}
        \sum_x \E_{i_0, i_1\in I} \sum_{\substack{\uk_0\in K_{i_0},\\ \uk_1\in K_{i_1}}}\sum_{\substack{h_0\in H_{i_0},\\ h_1\in H_{i_1}}}  \mu_{K_{i_0}}(\uk_0) \mu_{K_{i_1}}(\uk_1) \mu_{H_{i_0}}(h_0) \mu_{H_{i_1}}(h_1)
        \overline{f(x+h_0)}f(x+h_1)\\
        \Delta^*_{\uk_0}\Delta_{h_0} f(x)\overline{\Delta^*_{\uk_1}\Delta_{h_1} f(x)}\geq \delta^2 |B|.
    \end{align*}
    We crucially observe that for each fixed $i_0,i_1\in I$ and $h_0\in H_{i_0}, h_1\in H_{i_1}$, the average
    \begin{align}\label{E: GCS in concatenation}
        \sum_x \sum_{\substack{\uk_0\in K_{i_0},\\ \uk_1\in K_{i_1}}}\mu_{K_{i_0}}(\uk_0) \mu_{K_{i_1}}(\uk_1)
        f(x+h_0)\overline{f(x+h_1)}
        \Delta^*_{\uk_0}\Delta_{h_0} f(x)\overline{\Delta^*_{\uk_1}\Delta_{h_1} f(x)}
    \end{align}
    is a box inner product of $1$-bounded functions along $K_{1i_0}, \ldots, K_{si_0}, K_{1i_1}, \ldots, K_{si_1}$, with $f(x+h_0)\overline{f(x+h_1)}$ being the term with multi-index $\underline{0}$. Moreover, since $f$ is supported on $B$, the summand in \eqref{E: GCS in concatenation} vanishes unless
    \begin{align*}
        x+h_j,\; x + \ueps_j\cdot \uk_j,\; x + \ueps_j\cdot \uk_j + h_j\in B\quad \textrm{for all}\quad j\in\{0,1\},\; \ueps_j\neq \underline{0}.
    \end{align*}
    In particular, if all of these elements lie in $B$, then we also have
    $$x+\ueps_0\cdot\uk_0 + \ueps_1\cdot \uk_1 + h_0=(x+\ueps_0\cdot\uk_0+h_0)+(x+ \ueps_1\cdot \uk_1 + h_1)-(x+h_1) \in B+B-B$$
    for any $\ueps_0, \ueps_1\in\{0,1\}^s\setminus{\{0\}}$,
    and so we can introduce the term
    \begin{align}\label{E: product of inserted indicator functions}
        \prod_{\substack{\ueps_0, \ueps_1\in\{0,1\}^s\setminus{\{0\}}}} 
        \mathbf{1}_{B+B-B-h_0}(x+\ueps_0\cdot\uk_0 + \ueps_1\cdot\uk_1)
    \end{align}
    into \eqref{E: GCS in concatenation} without changing the value of the expression. The purpose of inserting this term is to ensure that if we view \eqref{E: GCS in concatenation} as a box inner product along $K_{1i_0}, \ldots, K_{si_0}, K_{1i_1}, \ldots, K_{si_1}$, then the functions evaluated at $x+\ueps_0\cdot\uk_0 + \ueps_1\cdot \uk_1$ with $\ueps_0, \ueps_1\neq \underline{0}$ have support of size bounded by $|\supp(B+B-B)|$, and hence by $A|B|$ by our assumption.

    We now apply the Gowers-Cauchy--Schwarz inequality to \eqref{E: GCS in concatenation} weighed by the product \eqref{E: product of inserted indicator functions} to obtain
    \begin{align*}
        \E_{i_0, i_1\in I} \sum_x  \sum_{\substack{\uk_0\in K_{i_0},\\ \uk_1\in K_{i_1}}} \mu_{K_{i_0}}(\uk_0) \mu_{K_{i_1}}(\uk_1) \E_{\substack{h_0\in H_{i_0},\\ h_1\in H_{i_1}}} \Delta_{\uk_0, \uk_1}(f(x+h_0)\overline{f(x+h_1)})\gg_{s} A^{-O_s(1)} \delta^{2^{2s+1}}|B|,
    \end{align*}
    or equivalently
    \begin{align*}
        \E_{i_0, i_1\in I} \sum_x  \E_{\substack{\uk_0, \uk'_0 \in K_{i_0},\\ \uk_1, \uk'_1 \in K_{i_1}}} \E_{\substack{h_0\in H_{i_0},\\ h_1\in H_{i_1}}} \Delta'_{(\uk_0,\uk_0'), (\uk_1, \uk_1')}(f(x)\overline{f(x+h_1-h_0)})\gg_{s} A^{-O_s(1)} \delta^{2^{2s+1}}|B|
    \end{align*}
    after we expand the kernels $\mu_{K_{i_0}}, \mu_{K_{i_1}}$ and change variables $x\mapsto x - h_0$. Using the Cauchy--Schwarz inequality to double the variables $h_0, h_1$ yields
    \begin{align*}
        \E_{i_0, i_1\in I} \sum_x  \E_{\substack{\uk_0, \uk'_0 \in K_{i_0},\\ \uk_1, \uk'_1 \in K_{i_1}}} \E_{\substack{h_0, h_0'\in H_{i_0},\\ h_1, h_1' \in H_{i_1}}} \Delta'_{(\uk_0,\uk_0'), (\uk_1, \uk_1'), (h_1-h_0, h_1'-h_0')}f(x)\gg_{s} A^{-O_s(1)} \delta^{2^{2s+2}}|B|.
    \end{align*}
    This expression is now an average (indexed by $i_0, i_1$) of box norms along $K_{i_0}, K_{i_1}, H_{i_1}-H_{i_0}$.  The inductive formula for box norms gives the claimed bound
    \begin{align*}
        \E_{i_0, i_1\in I}\norm{f}_{K_{1 i_0}, \ldots, K_{s i_0}, K_{1 i_1}, \ldots, K_{s i_1}, H_{i_1}-H_{i_0}}^{2^{2s+1}}\gg_{s} A^{-O_s(1)} \delta^{2^{2s+2}} |B|.
    \end{align*}
    We reiterate our convention that $H_{i_1}-H_{i_0}$ denotes the Minkowski sum counted with multiplicity and in particular is itself a multiset.
\end{proof}

By repeatedly applying Lemma \ref{L: concatenation lemma}, we derive the following concatenation result, whose proof follows closely that of \cite[Proposition 2.3]{Kuc23}.
\begin{proposition}[Concatenation of box norms, version I]\label{P: concatenation for general groups}
For every $s\in \mathbb{N}$, there is a positive real $C=O_s(1)$ such that the following holds.     Let $A, \delta>0$,
    let $G$ be a countable abelian group, let $I$ be a finite indexing set, and let $H_{1i}, \ldots, H_{si}\subseteq G$ be finite multisets for each $i\in I$.  Suppose that $B \subseteq G$ is a set
    satisfying 
    \begin{align*}
        |\supp(B+B-B)|\leq A|B|.
    \end{align*}
    Then for all $1$-bounded functions $f$ supported on $B$, the bound
    \begin{align}\label{E: concatenation V1 lower bound}
        {\E_{i\in I}\norm{f}_{H_{1i}, \ldots, H_{si}}^{2^{s}}}\geq \delta |B|
    \end{align}
    implies that
    \begin{align*}
        \E_{\substack{i_\ueps\in I,\\ \ueps\in\{0,1\}^s}}\norm{f}_{\substack{\{H_{j i_\ueps}-H_{j i_{\ueps'}}:\ j\in[s]\ \ueps,\ueps'\in\{0,1\}^s\\ \mathrm{with}\ (\epsilon_1, \ldots,  \epsilon_{s-j}) = (\epsilon'_1, \ldots,  \epsilon'_{s-j}),\ \epsilon_{s+1-j}=1,\ \epsilon'_{s+1-j} = 0\}}}^{2^d} \geq \frac{1}{C} A^{-C} \delta^{C}|B|,
    \end{align*}
    where $d$ is the degree of the box norms on the left-hand side.
\end{proposition}
Since the indices $\ueps,\ueps'$ are always distinct, the multiset $H_{j i_\ueps} - H_{j i_{\ueps'}}$ is usually more uniformly distributed than both of the multisets $H_{j i_\ueps}, H_{j i_{\ueps'}}$.

\begin{proof}
    The proof of Proposition \ref{P: concatenation for general groups} relies on a gradual concatenation of the ``unconcatenated'' multisets $H_{1i}, \ldots, H_{si}$ using Lemma \ref{L: concatenation lemma}. We repeatedly use the inductive formula for box norms in order to reinterpret the average in such a way  that successive applications of Lemma \ref{L: concatenation lemma} concatenate the multisets $H_{1i}, \ldots, H_{si}$ one by one.

 Starting with \eqref{E: concatenation V1 lower bound}, we apply Lemma \ref{L: concatenation lemma} to bound
\begin{align}\label{E: concatenation s>2 1}
    \E_{i_0, i_1\in I}\norm{f}_{H_{1 i_0}, \ldots, H_{(s-1) i_0}, H_{1 i_1}, \ldots, H_{(s-1) i_1}, H_{s i_1}-H_{s i_0}}^{2^{2s-1}}\gg_{s} A^{-O_s(1)} \delta^{O_s(1)}|B|.
\end{align}
Instead of having $s$ unconcatenated multisets indexed by $i$, we now have $s-1$ unconcatenated multisets indexed by $i_0$ and another $s-1$ unconcatenated multisets indexed by $i_1$, together with the concatenated multiset $H_{s i_1}-H_{s i_0}$. Even though the total number of multisets has nearly doubled, we have made progress because the number of unconcatenated multisets for each index $i_0, i_1$ decreased by $1$. 

At the next stage of the argument, we will apply Lemma \ref{L: concatenation lemma} twice to concatenate $H_{(s-1) i_0}$ and then $H_{(s-1) i_1}$. As a consequence, the two indices $i_0, i_1$ will be replaced by four indices $i_{00}, i_{01}, i_{10}, i_{11}$, and for each of them we will have exactly $s-2$ unconcatenated multisets. 
We will continue in this manner: At each stage, the number of indices $i_\ueps$ will double, but the number of unconcatenated multisets with each index $i_\ueps$ will decrease by $1$.
Eventually, after $s$ steps, we will be left with $2^{s-1}$ unconcatenated multisets $H_{1 i_\ueps}$, and $2^{s-1}$ applications of Lemma \ref{L: concatenation lemma} will allow us to concatenate them all without producing any new unconcatenated multisets. This will finish the argument.

To illustrate this general strategy, let us see in detail what happens at the second stage, i.e., after we have obtained the bound \eqref{E: concatenation s>2 1}. Using the induction formula for box norms, we can express \eqref{E: concatenation s>2 1} as
\begin{multline*}
    \E_{i_1\in I} \E_{\substack{h_{1i_1}, h_{1i_1}'\in H_{1i_1}}} \cdots \E_{\substack{h_{(s-1) i_1}, h_{(s-1) i_1}' \in H_{(s-1) i_1}}}\\ \E_{i_0\in I} \norm{\Delta'_{(h_{1i_1}, h_{1i_1}'), \ldots, (h_{(s-1)i_1}, h_{(s-1)i_1}')} f}_{H_{1 i_0}, \ldots, H_{(s-1) i_0}, H_{s i_1}-H_{s i_0}}^{2^{s}}\gg_{s} A^{-O_s(1)} \delta^{O_s(1)} |B|.
\end{multline*}
 For each fixed $i_1, h_{1i_1}, h_{1i_1}', \ldots, h_{(s-1) i_1}, h_{(s-1) i_1}'$, we apply Lemma \ref{L: concatenation lemma} separately to each average over $i_0$, obtaining
\begin{multline*}
    \E_{i_1\in I} \E_{\substack{h_{1i_1}, h_{1i_1}'\in H_{1i_1}}} \cdots \E_{\substack{h_{(s-1) i_1}, h_{(s-1) i_1}' \in H_{(s-1) i_1}}} \E_{i_{00}, i_{01}\in I}\\
    \norm{\Delta'_{(h_{1i_1}, h_{1i_1}'), \ldots, (h_{(s-1)i_1}, h_{(s-1)i_1}')} f}_{\substack{H_{1 i_{00}}, \ldots, H_{(s-2) i_{00}}, H_{1 i_{01}}, \ldots, H_{(s-2) i_{01}},\\ H_{(s-1) i_{01}}-H_{(s-1) i_{00}}, H_{s i_1}-H_{s i_{00}}, H_{s i_1}-H_{s i_{01}}}}^{2^{2s-1}} \hspace{-6mm} \gg_{s} A^{-O_s(1)} \delta^{O_s(1)} |B|.
\end{multline*}
Using the inductive formula for box norms and swapping the indices of $\Delta'$ and the box norm, we can rearrange this to get 
\begin{multline*}
    \E_{i_{00}, i_{01}\in I} \E_{\substack{h_{1i_{00}}, h_{1i_{00}}'\in H_{1i_{00}},\\ h_{1i_{01}}, h_{1i_{01}}'\in H_{1i_{01}}}}\cdots 
    \E_{\substack{h_{(s-1) i_{00}}, h_{(s-1) i_{00}}'\in H_{(s-1) i_{00}},\\ h_{(s-1) i_{01}}, h_{(s-1) i_{01}}'\in H_{(s-1) i_{01}}}}\\
    \E_{i_1\in I} \norm{\Delta'_{\substack{(h_{1 i_{00}}, h_{1 i_{00}}'), \ldots, (h_{(s-2) i_{00}}, h_{(s-2) i_{00}}')\\ (h_{1 i_{01}}, h_{1 i_{01}}'), \ldots, (h_{(s-2) i_{01}}, h_{(s-2) i_{01}}'),\\ (h_{(s-1) i_{01}} - h_{(s-1) i_{00}}, h'_{(s-1) i_{01}} - h'_{(s-1) i_{00}}) }} f}_{\substack{H_{1 i_1}, \ldots, H_{(s-1) i_1},\\ H_{s i_1}-H_{s i_{00}}, H_{s i_1}-H_{s i_{01}}}}^{2^{s+1}} \gg_{s} A^{-O_s(1)} \delta^{O_s(1)} |B|.     
\end{multline*}
By Lemma \ref{L: concatenation lemma} applied separately to each average over $i_1\in I$, we have
\begin{multline*}
    \E_{\substack{i_{00}, i_{01},\\ i_{10}, i_{11}\in I}} 
    \E_{\substack{h_{1i_{00}}, h_{1i_{00}}'\in H_{1i_{00}},\\ h_{1i_{01}}, h_{1i_{01}}'\in H_{1i_{01}}}}\cdots 
    \E_{\substack{h_{(s-1) i_{00}}, h_{(s-1) i_{00}}'\in H_{(s-1) i_{00}},\\ h_{(s-1) i_{01}}, h_{(s-1) i_{01}}'\in H_{(s-1) i_{01}}}}\\
    \norm{\Delta'_{\substack{(h_{1 i_{00}}, h_{1 i_{00}}'), \ldots, (h_{(s-2) i_{00}}, h_{(s-2) i_{00}}')\\ (h_{1 i_{01}}, h_{1 i_{01}}'), \ldots, (h_{(s-2), i_{01}}, h_{(s-2), i_{01}}'),\\ (h_{(s-1) i_{01}} - h_{(s-1) i_{00}}, h'_{(s-1) i_{01}} - h'_{(s-1) i_{00}}) }} f}_{\substack{H_{1 i_{10}}, \ldots, H_{(s-2) i_{10}}, H_{1 i_{11}}, \ldots, H_{(s-2) i_{11}},\\ H_{(s-1) i_{11}}-H_{(s-1) i_{10}}, H_{s i_{10}}-H_{s i_{00}}, H_{s i_{11}}-H_{s i_{00}},\\ H_{s i_{10}}-H_{s i_{01}}, H_{s i_{11}}-H_{s i_{01}}}}^{2^{2s+1}} \hspace{-20mm}\\ \gg_{s} A^{-O_s(1)} \delta^{O_s(1)}|B|.
\end{multline*}
Another application of the inductive formula for box norms lets us write this as
\begin{align*}
    \E_{\substack{i_{00}, i_{01},\\ i_{10}, i_{11}\in I}} \norm{f}_{\substack{H_{1 i_{00}}, \ldots, H_{(s-2) i_{00}}, H_{1 i_{01}}, \ldots, H_{(s-2) i_{01}},\\ H_{1 i_{10}}, \ldots, H_{(s-2) i_{10}}, H_{1 i_{11}}, \ldots, H_{(s-2) i_{11}},\\ H_{(s-1) i_{01}}-H_{(s-1) i_{00}}, H_{(s-1) i_{11}}-H_{(s-1) i_{10}},\\ H_{s i_{10}}-H_{s i_{00}}, H_{s i_{11}}-H_{s i_{00}},\\ H_{s i_{10}}-H_{s i_{01}}, H_{s i_{11}}-H_{s i_{01}}}}^{2^{4s-2}}\gg_{s} A^{-O_s(1)} \delta^{O_s(1)}|B|,
\end{align*}
which can be expressed more compactly as
\begin{align*}
    \E_{\substack{i_{00}, i_{01},\\ i_{10}, i_{11}\in I}} \norm{f}_{\substack{\{H_{j i_\ueps}:\ \ueps\in\{0,1\}^2,\ j=1, \ldots, s-2\},\\ \{H_{(s-1)i_\ueps}-H_{(s-1)i_{\ueps'}}:\ \ueps,\ueps'\in\{0,1\}^2\ \mathrm{with}\ \epsilon_1 = \epsilon'_1,\ \epsilon_2 = 1,\ \epsilon'_2=0\},\\  
    \{H_{s i_\ueps}-H_{s i_{\ueps'}}:\ \ueps,\ueps'\in\{0,1\}^2\ \mathrm{with}\ \epsilon_1 =1,\ \epsilon'_1 = 0\}}}^{2^{4s-2}}\gg_{s} A^{-O_s(1)} \delta^{O_s(1)}|B|.
\end{align*}
We have thus successfully concatenated all of the multisets $H_{(s-1) i_\ueps}$ and $H_{s i_\ueps}$. 

At the next stage, we concatenate the multisets $H_{(s-2)i_\ueps}$. Applying Lemma \ref{L: concatenation lemma} and the induction formula for box norms four times (as in the previous stage), each time to an average over $i_{00}, i_{01}, i_{10}, i_{11}$ (respectively), we arrive at the inequality
\begin{align*}
        \E_{\substack{i_{\ueps}\in I,\\ \ueps\in\{0,1\}^3}} \norm{f}_{\substack{\{H_{j i_\ueps}:\ \ueps\in\{0,1\}^3,\ j=1, \ldots, s-3\},\\ \{H_{j i_\ueps}-H_{j i_{\ueps'}}:\ j= s-2, s-1, s,\ \ueps,\ueps'\in\{0,1\}^3\\ \mathrm{with}\ (\epsilon_1, \ldots, \epsilon_{s-j}) = (\epsilon'_1, \ldots, \epsilon'_{s-j}),\ \epsilon_{s+1-j} = 1,\ \epsilon'_{s+1-j} = 0\}}}^{2^{d'}} 
        \gg_{s} A^{-O_s(1)} \delta^{O_s(1)}|B|,
\end{align*}
where $d'$ is the degree of the box norms on the left-hand side.
This time, we have successfully concatenated the {multisets} $H_{(s-2)i_\ueps}$. At the next step, eight applications of Lemma \ref{L: concatenation lemma} and the induction formula for box norms allow us to concatenate {multisets} $H_{(s-3)i_\ueps}$. Continuing the argument in this manner, we arrive at the conclusion of the proposition after a total of 
\begin{align*}
    1 + 2 + 2^2 + \cdots + 2^{s-1} = 2^s - 1
\end{align*}
applications of Lemma \ref{L: concatenation lemma} and the induction formula for box norms.
\end{proof}

We now present two corollaries of Proposition \ref{P: concatenation for general groups} which are better suited for applications in $\Z^D$. 
In what follows, we order the set $\{0,1\}^s$ lexicographically.
Identifying $\{0,1\}^s$ with $[2^s]$, we obtain the following corollary.

\begin{corollary}[Concatenation of box norms, version II]\label{C: concatenation for general groups III}
Let $s,D \in \mathbb{N}$.  There exist positive reals $C=O_{s,D}(1)$ and $C'=O_{s}(1)$ such that the following holds.
    Let $A,N\in\N$, let $I$ be a finite indexing set, and let $H_{ji}\subseteq [\pm AN]^D$ be finite nonempty multisets for each $i\in I$ and $j\in[s]$. Then for all $1$-bounded functions $f:\Z^D\to\C$ supported on $[N]^D$, the bound
    \begin{align*}
        \E_{i\in I}\norm{f}_{H_{1i}, \ldots, H_{si}}^{2^{s}}\geq \delta N^D
    \end{align*}
    implies that
    \begin{align}\label{E: concatenation simplified}
        \E_{\substack{k_1, \ldots, k_{t}\in I}}\norm{f}_{\{H_{j k_{i_1}}-H_{j k_{i_2}}:\ j\in[s],\ 0<i_1<i_2\leq t\}}^{2^d}\geq \frac{1}{C} A^{-C'D} \delta^{C} N^D,
    \end{align}
    where $t=2^s$ and $d$ is the degree of the box norms on the left-hand side. 
\end{corollary}
\begin{proof}
    Proposition \ref{P: concatenation for general groups} immediately implies that
    \begin{align*}
        \E_{\substack{i_\ueps\in I,\\ \ueps\in\{0,1\}^s}}\norm{f}_{\substack{\{H_{j i_\ueps}-H_{j i_{\ueps'}}:\ j\in[s]\ \ueps,\ueps'\in\{0,1\}^s\\ \mathrm{with}\ (\epsilon_1, \ldots,  \epsilon_{s-j}) = (\epsilon'_1, \ldots,  \epsilon'_{s-j}),\ \epsilon_{s+1-j}=1,\ \epsilon'_{s+1-j} = 0\}}}^{2^{d'}} \gg_{D, s} A^{-O_s(D)}\delta^{O_s(1)}N^D,
    \end{align*}
    where $d'$ is the degree of the box norms on the left-hand side. Ordering $\{0,1\}^s$ lexicographically and identifying it with $[2^s]=[t]$ in an order-preserving way, we obtain a subset $\CK_0\subseteq[t]\times[t]$ such that
    \begin{align*}
        \E_{\substack{k_1, \ldots, k_t\in I}}\norm{f}_{\substack{\{H_{j k_{i_1}}-H_{j k_{i_2}}:\; j\in[s],\ ({i_1}, {i_2})\in\CK_0\}}}^{2^{d'}} \gg_{D, s} A^{-O_s(D)}\delta^{O_s(1)}N^D.
    \end{align*}    
    Note that $\CK_0 \subseteq \{(i_1,i_2) \in [t]\times [t]: i_1<i_2\}$.  Our next goal is to replace $\CK_0$ by the full set $\CK = \{(i_1, i_2)\in[t]\times[t]:\; i_1<i_2\}$.
    By the assumption on the multisets $H_{ji}$, we have
    $$[N]^D - (H_{j k_1}-H_{j k_2})\subseteq (-2A N, (2A+1)N]^D,$$ and so $$\abs{\supp([N]^D - (H_{j k_1}-H_{j k_2}))}\leq (4A+1)^D N^D.$$  The corollary follows from this observation and repeated applications of Lemma \ref{L: properties of box norms}\eqref{i: monotonicity}, the monotonicity property for box norms.
\end{proof}

The next corollary (which was stated in Section~\ref{S: strategy} and which we will now restate for convenience) will let us take the directions in the concatenated box norms to be arbitrarily long sums of the original directions rather than just double sums (as in Corollary \ref{C: concatenation for general groups III}).
\iteratedconcatenationofboxnorms*
It will be useful later that the indices $i_1, \ldots, i_{\ell}$ in the sum
\begin{align*}
    H_{j k_{i_1}}+\cdots + H_{j k_{i_{\ell}}}
\end{align*}
in \eqref{E: iterated concatenation} are all distinct.  The assumption of the symmetry of the multisets $H_{ji}$ (meaning that $H_{ji}=-H_{ji}$)) in the corollary is not necessary, and we include it only so that we do not have to keep track of the minus signs in the expressions $H_{j k_{i_1}}+\cdots + H_{j k_{i_{\ell}}}$.

\begin{proof}
    The proof proceeds by induction on $\ell$, which we recall ranges over powers of 2. For $\ell=1$, the claim is trivial, and for $\ell=2$, the lower bound
    \begin{align}\label{E: iterated concatenation induction}
        \E_{\substack{k_1, \ldots, k_{t_1}\in I}}\norm{f}_{\{H_{j k_{i_1}}+H_{j k_{i_2}}:\ j\in[s],\ 0<i_1< i_2\leq t_1\}}^{2^{d_1}}\gg_{D,  s} A^{-O_{s}(D)} \delta^{O_{s}(1)} N^D,
    \end{align}
    where $t_1 = 2^s$ and $d_1 = {{t_1}\choose{2}}$ is the degree of the box norms appearing, follows from Corollary \ref{C: concatenation for general groups III} and the fact that the sets $H_{ji}$ are symmetric. We will now prove the case $\ell=4$, and which will make the induction scheme for larger $\ell$ clear.

    To prove the result for $\ell=2$, we apply the case $\ell=1$ to \eqref{E: iterated concatenation induction}, taking $(k_1, \ldots, k_{t_1})\in I^{t_1}$ in place of $i\in I$. This gives
    \begin{align}\label{E: iterated concatenation induction 2}
        \E_{\substack{k_{i l}\in I:\\ i\in[t_1],\; l\in[t_1']}}\norm{f}_{\substack{\{H_{j k_{i_1l_1}}+H_{j k_{i_2l_1}} + H_{j k_{i_1l_2}}+H_{j k_{i_2l_2}}:\\ j\in[s],\; 0<i_1< i_2\leq t_1,\; 0<l_1<l_2\leq t_1'\}}}^{2^{d_1'}}\gg_{D, s} A^{-O_{s}(D)} \delta^{O_{s}(1)} N^D,
    \end{align}
    where $t_1' = 2^{d_1}$ and $d_1'$ is the degree of the box norms appearing. Reindexing $(i, l)\in[t_1]\times[t_1']$ as $i\in[t_2]$ and using the monotonicity property and the assumption on $H_{ji}$, we argue as in the proof of Corollary \ref{C: concatenation for general groups III} that
    \begin{align*}
                \E_{\substack{k_1, \ldots, k_{t_2}\in I}}\norm{f}_{\{H_{j k_{i_1}}+\cdots + H_{j k_{i_4}}:\ j\in[s],\ 0<i_1< \cdots < i_4\leq t_2\}}^{2^{d_2}}\gg_{D, s} A^{-O_s(D)} \delta^{O_{s}(1)} N^D,
    \end{align*}
    where $d_2$ is the degree of the box norms appearing.  We continue in the same manner for larger $\ell$.
\end{proof}

\begin{remark}
One can phrase our results using the language of weights, as in \cite[Lemma 5.5]{GS25}. In fact, the appropriate group analogue of \cite[Lemma 5.5]{GS25} is an immediate consequence of our Proposition~\ref{P: concatenation for general groups}, since one can approximate a weight to arbitrary accuracy as the uniform distribution on a suitable finite multiset and then take a limit.
\end{remark}

\section{Equidistribution estimates for concatenation}\label{S: equidistribution}

The purpose of this section is to obtain upper bounds on the number of solutions to systems of multilinear equations that appear in connection with the manipulations in Section \ref{S: concatenation along polys}.
These results will culminate in Proposition \ref{P: systems of multilinear equations}, which is precisely the ingredient needed to complete the proof of Theorem \ref{T: concatenation of polynomials}.

The starting point for our argument is the following simple lemma. 
\begin{lemma}[Bound on solutions to linear congruences]\label{L: linear congruences}
    Let $\ell\geq 2$, let $h_1, \ldots, h_\ell$ be positive integers, and let $M\in\N$. Then
    \begin{align*}
        \max_{t\in\Z}\sum_{m_1, \ldots, m_{\ell-1}\in[\pm M]}\mathbf{1}_{h_\ell \mathbb{Z} + t}(h_1 m_1 + \cdots + h_{\ell-1} m_{\ell-1}) \ll_\ell M^{\ell-1}\frac{\gcd(h_1, \ldots, h_\ell)}{h_\ell} + M^{\ell-2}.
    \end{align*}
\end{lemma}

\begin{proof}
First, we observe that if $(m'_1, \ldots, m'_{\ell-1}) \in [\pm M]^{\ell-1}$ is a solution to the congruence
\begin{align*}
    h_1 m_1 + \cdots + h_{\ell-1} m_{\ell-1} \equiv t \pmod {h_\ell},
\end{align*}
then every other solution in $[\pm M]^{\ell-1}$ takes the form $(m_1,\ldots, m_{\ell-1}) = (m'_1 + m''_1, \ldots, m'_\ell + m''_\ell)$, where $(m''_1, \ldots, m''_\ell)$ is a solution to the homogeneous congruence
\begin{align*}
    h_1 m_1 + \cdots + h_{\ell-1} m_{\ell-1} \equiv 0 \pmod {h_\ell}
\end{align*}
and $m''_i \in [\pm M]-m'_i\subseteq [\pm 2M]$.
Therefore, it suffices to prove the $t=0$ case of the lemma (at the cost of changing the implicit constant in the conclusion of the lemma).

We proceed by induction on $\ell$.  For the base case $\ell=2$, we have
$$h_2\, | \, h_1 m_1 \iff \frac{h_2}{\gcd(h_1, h_2)}\,\Big| \, m_1,$$
and hence
$$\sum_{m_1 \in [\pm M]} \mathbf{1}_{h_2\mathbb{Z}} (h_1m_1)\ll M \frac{\gcd(h_1, h_2)}{h_2} + 1$$
(where $+1$ is the boundary term).

For the induction step, write
\begin{multline*}
\sum_{m_1, \ldots, m_{\ell-1}\in[\pm M]}\mathbf{1}_{h_\ell\mathbb{Z} }(h_1 m_1 + \cdots + h_{\ell-1} m_{\ell-1})=\\
\sum_{m_1, \ldots, m_{\ell-2}\in[\pm M]} \sum_{m_{\ell-1} \in [\pm M]}\mathbf{1}_{h_\ell\mathbb{Z}-(h_1 m_1 + \cdots + h_{\ell-2} m_{\ell-2})}(h_{\ell-1} m_{\ell-1}).
\end{multline*}
The crucial observation is that the inner summand vanishes if $h_1m_1+\cdots+h_{\ell-2}m_{\ell-2}$ is not a multiple of $g:=\gcd(h_{\ell-1},h_\ell)$.  Now suppose that $g$ does divide $h_1m_1+\cdots+h_{\ell-2}m_{\ell-2}$; the induction hypothesis tells us that this occurs for at most $$O_\ell\brac{M^{\ell-2} \frac{\gcd(h_1, \ldots, h_\ell)}{g}+M^{\ell-3}}$$ choices of $m_1, \ldots, m_{\ell-2}$, where we have used the identity 
$$\gcd(h_1, \ldots, h_{\ell-2},\gcd(h_{\ell-1},h_\ell))=\gcd(h_{1},\ldots,h_\ell).$$  
For each such choice of $m_1, \ldots, m_{\ell-2}$, we have
$$\mathbf{1}_{h_\ell\mathbb{Z}-(h_1 m_1 + \cdots + h_{\ell-2} m_{\ell-2})}(h_{\ell-1} m_{\ell-1})=\mathbf{1}_{(h_\ell/g)\mathbb{Z}-(h_1 m_1 + \cdots + h_{\ell-2} m_{\ell-2})/g}((h_{\ell-1}/g)m_{\ell-1}).$$
Now the $\ell=2$ case of the proposition tells us that the sum over $m_{\ell-1}$ has size
$$O\brac{M\frac{\gcd(h_{\ell-1}/g,h_\ell/g)}{h_{\ell}/g}+1}=O\brac{M\frac{g}{h_{\ell}}+1}$$
since $\gcd(h_{\ell-1}/g,h_\ell/g)=1$ by the definition of $g$.  Combining these two observations and using $\gcd(h_1, \ldots, h_\ell) \leq g \leq h_\ell$, we bound
$$\sum_{m_1, \ldots, m_{\ell-1}\in[\pm M]}\mathbf{1}_{h_\ell\mathbb{Z} }(h_1 m_1 + \cdots + h_{\ell-1} m_{\ell-1})\ll_\ell M^{\ell-1}\frac{\gcd(h_1, \ldots, h_\ell)}{h_\ell}+M^{\ell-2},$$
as desired.
\end{proof}

\begin{corollary}[Bound on solutions to linear equations]\label{C: linear congruences}
    Let $\ell\geq 2$, let $h_1, \ldots, h_\ell$ be positive integers, and let $M\in\N$. Then
    \begin{align*}
        \max_{t\in\Z}\sum_{m_1, \ldots, m_\ell\in[\pm M]}\mathbf{1}(h_1 m_1 + \cdots + h_\ell m_\ell = t) \ll_\ell M^{\ell-1}\frac{\gcd(h_1, \ldots, h_\ell)}{h_\ell} + M^{\ell-2}.
    \end{align*}
\end{corollary}
We will later apply this corollary with
\begin{align*}
    \max_{i\in[\ell]}|h_i|\leq H \leq M,\quad \gcd(h_1, \ldots, h_\ell) \leq \eta^{-1}, \quad\textrm{and}\quad h_\ell \geq \eta H
\end{align*}
(for small positive $\eta$), in which case the right-hand side is $\ll_\ell \eta^{-2}M^{\ell-1}H^{-1}$.
\begin{proof}
    For each choice of $m_1, \ldots, m_{\ell-1}$, there is at most one value of $m_\ell$ satisfying $h_1 m_1 + \cdots + h_\ell m_\ell = t$, so we can bound
    \begin{align*}
        \sum_{m_1, \ldots, m_\ell\in[\pm M]}\mathbf{1}_{h_1 m_1 + \cdots + h_\ell m_\ell = t} \leq \sum_{m_1, \ldots, m_{\ell-1}\in[\pm M]}\mathbf{1}_{h_\ell \Z +t}(h_1 m_1 + \cdots + h_{\ell-1} m_{\ell-1});
    \end{align*}
    here we have an inequality instead of an equality because on the right-hand side we have thrown out the constraint $m_\ell \in [\pm M]$. The result now follows from Lemma \ref{L: linear congruences}.  
\end{proof}

In order to state our more general results, we require a bit of notation.  Define
\begin{align}\label{E:set K}
        \CK_{t, \ell, r} := \{(k_{li})_{\substack{(l,i)\in[r]\times[\ell]}}\in[t]^{r\ell}:\; 1\leq k_{l1}<\cdots < k_{l\ell}\leq t\;\; \textrm{for\; all}\;\; l\in[r]\}
    \end{align}
    to be the collection of admissible tuples of indices. For $l\in[r]$ and $\eta>0$, also define
    \begin{multline}\label{E: H_l}
        \CH_{l,\eta} := \{(h_{lk_1\cdots k_l})_{k_1, \ldots, k_l\in[t]}\in[\pm H]^{t^l}:\; |h_{lk_1\cdots k_l}-h_{lk''_1\cdots k''_l}|\geq \eta H,\\ \gcd(h_{lk_1\cdots k_l}-h_{lk''_1\cdots k''_l}, h_{lk'_1\cdots k'_l}-h_{lk''_1\cdots k''_l})\leq \eta\inv\\
    \textrm{for\; distinct}\; (k_1, \ldots, k_l),\; (k'_1, \ldots, k'_l),\; (k''_1, \ldots, k''_l)\in[t]^l\}
    \end{multline}
    to be the set of tuples of $h$'s whose pairwise differences are not too small and whose greatest common divisors are not too big.  The following lemma, which lets us establish anti-concentration results on many pairs of ``scales'' simultaneously, will turn out to be the key component of the proof of Proposition \ref{P: systems of multilinear equations} below.
    
\begin{lemma}\label{L: bound on pinned solutions}
Let $\eta>0$ and $H, M, l, \ell, r, s, t\in\N$ satisfy $H\leq M$,  $3\leq \ell\leq t$, and $l\leq r$. Let $\CK = \CK_{t, \ell, r}$ be as in \eqref{E:set K} and $\CH = \CH_{l, \eta}$ be as in \eqref{E: H_l}.
Then
    \begin{multline*}
        \E_{\substack{h_{lk_1\cdots k_l}\in[\pm H]:\\ k_1, \ldots, k_l\in[t]}}\; \mathbf{1}_{\CH}((h_{lk_1\cdots k_l})_{k_1, \ldots, k_l})\cdot \prod_{i_1, \ldots, i_{l-1}\in[\ell]}\;\prod_{j\in[s]}\;\prod_{\uk\in\CK}\; \prod_{u_{l+1}, \ldots, u_r\in\{0,1\}}\\ \max_{\substack{n_{0}, n_{1}\in\Z}}\;
        \E_{\substack{|m_{i_l}|\leq H^{u_{l+1}+\cdots + u_r} M:\\ i_l\in[\ell]}}\; \mathbf{1}\Bigbrac{\sum\limits_{i_l\in[\ell]}m_{i_l} = n_0}\cdot
        \mathbf{1}\Bigbrac{\sum\limits_{i_l\in[\ell]}h_{lk_{1i_1}\cdots k_{li_l}}m_{i_l} = n_1}\\
        \ll_{l, \ell, r, s, t}\eta^{-2^{r-l+1}s\ell^{l-1}|\CK|}  M^{-2^{r-l+1}s\ell^{l-1}|\CK|} H^{-(r-l+1)2^{r-l}s\ell^{l-1}|\CK|}.
    \end{multline*}
\end{lemma}
\begin{example}
We will illustrate the content of this lemma and the method of proof for $r=2$, $l=1$, $t=4$, $\ell =3$, $s=1$ before supplying the proof in the general case. Letting
\begin{align*}
    \CK' = \{(1,2,3),\; (1,2,4),\; (1,3,4),\; (2,3,4)\}\quad \textrm{so\; that}\quad \CK = \CK' \times\CK' 
\end{align*}
and $|\CK|=16$, and renaming $h_{1k_{1i_1}}$ as simply $h_{k_i}$, we can express the conclusion of the lemma as
    \begin{multline*}
        \E_{\substack{h_1, h_2, h_3, h_4\in[\pm H]}}\; \mathbf{1}_{\CH}(h_1, h_2, h_3, h_4)\cdot \big(\prod_{\uk\in\CK'}\; \prod_{u_{2}\in\{0,1\}} \max_{\substack{n_{0}, n_{1}\in\Z}}\\
        \E_{\substack{|m_1|, |m_2|, |m_3|\leq H^{u_{2}} M}}\; \mathbf{1}(m_1 + m_2 + m_3 = n_0)\cdot \mathbf{1}(h_{k_1}m_1 + h_{k_2}m_2 + h_{k_3}m_3 = n_1)\big)^4\\
        \ll\eta^{-64}  M^{-64} H^{-64},
    \end{multline*}
    where the fourth power comes from the fact that we are splitting $$\prod\limits_{(k_{1i_1}, k_{2i_2})_{i_1,i_2\in[3]}\in\CK} = \prod\limits_{(k_{1i_1})_{i_1\in[3]}\in\CK'}\;\prod\limits_{(k_{2i_2})_{i_2\in[3]}\in\CK'}$$ and the terms in the product are independent of $k_{2i_2}$ (so the product over $\uk\in\CK'$ above corresponds to the product over $(k_{1i_1})_{i_1\in[3]}\in\CK'$ while the fourth power corresponds to the product over $(k_{2i_2})_{i_2\in[3]}\in\CK'$).
    Our normalized count can then be expanded as

\begin{multline*}
        \E_{\substack{h_1, h_2, h_3, h_4\in[\pm H]}}\; \mathbf{1}_{\CH}(h_1, h_2, h_3, h_4)\cdot \big(\prod_{\uk\in\CK'}\; \max_{\substack{n_{00}, n_{01}, n_{10}, n_{11}\in\Z}}\\
        \E_{\substack{|m_{01}|, |m_{02}|, |m_{03}|\leq M,\\ |m_{11}|, |m_{12}|, |m_{13}|\leq H M}}\; \mathbf{1} \begin{pmatrix}
        m_{01} + m_{02} + m_{03} = n_{00},\\
        m_{11} + m_{12} + m_{13} = n_{01}
        \end{pmatrix}\cdot \mathbf{1} \begin{pmatrix}
        h_{k_1}m_{01} + h_{k_2}m_{02} + h_{k_3}m_{03} = n_{10},\\ h_{k_1}m_{11} + h_{k_2}m_{12} + h_{k_3}m_{13} = n_{11}
        \end{pmatrix}\big)^4.
    \end{multline*}
Thus, for each $(h_1, h_2, h_3, h_4)\in\CH$ and $\uk\in\CK'$, we are left with bounding the number of solutions to the system
    \begin{align*}
        m_{01} + m_{02} + m_{03} = n_{00} \quad \textrm{and}\quad h_{k_1}m_{01} + h_{k_2}m_{02} + h_{k_3}m_{03} = n_{10},\\
        m_{11} + m_{12} + m_{13} = n_{01} \quad \textrm{and}\quad h_{k_1}m_{11} + h_{k_2}m_{12} + h_{k_3}m_{13} = n_{11}.
    \end{align*}
    Using the equations on the left to solve for $m_{03}$ and $m_{13}$ and substituting the resulting expressions into the equations on the right reduces the system to
    \begin{align*}
        (h_{k_1}-h_{k_3})m_{01} + (h_{k_2}-h_{k_3})m_{02} &= n_0\\
        (h_{k_1}-h_{k_3})m_{11} + (h_{k_2}-h_{k_3})m_{22} &= n_1
    \end{align*}
    for $n_0 = n_{10} - h_{k_3} n_{00}$ and  $n_1 = n_{11} - h_{k_3} n_{01}$, and each of these two equations can then be handled separately using Corollary \ref{C: linear congruences} and the assumptions on $(h_1, h_2, h_3, h_4)\in\CH_1$. 

    The exponent $64$ comes from the fact that the numbers $n_{00}, n_{01}, n_{10}, n_{11}$ lie in the intervals
    \begin{align*}
        n_{00} \in [\pm 3M],\quad n_{01}, n_{10}\in [\pm 3HM],\quad n_{11}\in[\pm 3H^2M]
    \end{align*}
    and hence there are $\asymp H^4 M^4$ possibilities for the quadruple $(n_{00},n_{01},n_{10},n_{11})$.  Because of the assumption on $(h_1, h_2, h_3, h_4)$, we can use Corollary \ref{C: linear congruences} to ensure that the quadruple $(n_{00},n_{01},n_{10},n_{11})$ is roughly equidistributed in the appropriate box as the $m$'s range; in particular, the probability of hitting each individual quadruple $(n_{00},n_{01},n_{10},n_{11})$ is 
   $$O((M^{-2}\eta^{-1}/(\eta H)) \cdot ((HM)^{-2}\eta^{-1}/(\eta H)))=O(\eta^{-4}H^{-4}M^{-4}).$$ 
 Raising this quantity to the power $4$ (because $|\CK'|=4$), we see that the expression in parentheses above is $O(\eta^{-16}H^{-16}M^{-16})$, and raising this to the power $4$ gives the desired exponent $64$.
\end{example}
\begin{proof}[Proof of Lemma \ref{L: bound on pinned solutions}]
    We let all implicit constants depend on $l, \ell, r, s, t$.
    Using the equation $\sum\limits_{i_l\in[\ell]}{m}_{i_l} = n_{0}$ we can substitute $m_{\ell} = n_{0} - \sum\limits_{i_l\in[\ell-1]}{m}_{i_l}$ into the equation $$\sum\limits_{\substack{i_l\in[\ell]}} h_{lk_{1i_1}\cdots k_{li_l}} {m}_{i_l} = n_{1},$$ so that our normalized count is bounded by
        \begin{multline*}
        \E_{\substack{h_{lk_1\cdots k_l}\in[\pm H]:\\ k_1, \ldots, k_l\in[t]}}\; \mathbf{1}_{\CH}((h_{lk_1\cdots k_l})_{k_1, \ldots, k_l})\cdot \prod_{i_1, \ldots, i_{l-1}\in[\ell]}\;\prod_{j\in[s]}\;\prod_{\uk\in\CK}\; \prod_{u_{l+1}, \ldots, u_r\in\{0,1\}}  H^{-(u_{l+1}+\cdots + u_r)} M\inv\\  \max_{\substack{n\in\Z}}\;
        \E_{\substack{|m_{i_l}|\leq H^{u_{l+1}+\cdots + u_r} M:\\ i_l\in[\ell-1]}} \mathbf{1}\Bigbrac{\sum\limits_{i_l\in[\ell-1]}(h_{lk_{1i_1}\cdots k_{li_l}}- h_{lk_{1i_1}\cdots k_{(l-1)i_{l-1}}k_{l\ell}})m_{i_l} = n},
    \end{multline*}
    where we have set $n=n_1 - n_0 h_{lk_{1i_1}\cdots k_{(l-1)i_{l-1}}k_{l\ell}}$. By Corollary \ref{C: linear congruences}, the inner max is in turn bounded by
    \begin{multline*}
        H^{-(u_{l+1}+\cdots + u_r)} M^{-1}\frac{\gcd(h_{lk_{1i_1}\cdots k_{(l-1)i_{l-1}}k_{l i_l}}- h_{lk_{1i_1}\cdots k_{(l-1)i_{l-1}}k_{l\ell}}:\; i_l\in[\ell-1])}{\max(h_{lk_{1i_1}\cdots k_{(l-1)i_{l-1}}k_{l i_l}}- h_{lk_{1i_1}\cdots k_{(l-1)i_{l-1}}k_{l\ell}}:\; i_l\in[\ell-1])}
    \end{multline*}    
    as long as $\ell\geq 3$ and $H \le M$. Due to our assumption on the set $\CH$, the $\gcd$ in the numerator is at most $\eta\inv$ and the maximum in the denominator is at least $\eta H$, so we can bound our normalized count by
    \begin{align*}
        &\ll\E_{\substack{h_{lk_1\cdots k_l}\in[\pm H]:\\ k_1, \ldots, k_l\in[t]}}\; \prod_{i_1, \ldots, i_{l-1}\in[\ell]}\;\prod_{j\in[s]}\;\prod_{\uk\in\CK}\; \prod_{u_{l+1}, \ldots, u_r\in\{0,1\}}\eta^{-2}M^{-2} H^{-1-2(u_{l+1}+\cdots + u_r)} \\
        &\ll \eta^{-2^{r-l+1}s\ell^{l-1}|\CK|}  M^{-2^{r-l+1}s\ell^{l-1}|\CK|} H^{-(r-l+1)2^{r-l}s\ell^{l-1}|\CK|}.
    \end{align*}
\end{proof}
    The next result, in which we we establish simultaneous anti-concentration for expressions on many scales, is the main estimate of this section. It gives an upper bound on the number of solutions to the systems of multilinear equations that naturally appear while concatenating box norms in the proof of Theorem \ref{T: concatenation of polynomials}.

\begin{proposition}[Solutions to systems of multilinear equations]\label{P: systems of multilinear equations}
    Let $\eta>0$, let $\ell, r, s, t\in\N$ with $3\leq\ell\leq t$, and let $H, M\in\N$ with $H\leq M$. Let $\CK = \CK_{t, \ell, r}$ be as in \eqref{E:set K}, and for each $l\in[r]$, let $\CH_l:=\CH_{l,\eta}$ be as in \eqref{E: H_l}. Then 
        \begin{multline*}
        \max_{\substack{n_{j\uk\uu}\in\Z:\; j\in[s],\\ \uk\in\CK,\; \uu\in\{0,1\}^r}}\;\E_{\substack{m_{j\uk i_1\cdots i_r}\in[\pm M]:\\ j\in[s],\; \uk\in\CK,\; i_1, \ldots,  i_r\in[\ell]}}
                \E_{\substack{h_{lk_{1}\cdots k_{l}}\in[\pm H]:\\  k_1, \ldots, k_r\in[t],\; l\in[r]}}\; \Bigbrac{\prod_{l\in[r]}\mathbf{1}_{\CH_l}((h_{lk_{1}\cdots k_{l}})_{k_1, \ldots, k_l \in [t]})}\\
        \prod_{j\in[s]}\; \prod_{\substack{\uk\in\CK}}\prod_{\substack{\uu\in\{0,1\}^r
        }} \mathbf{1}\Bigbrac{\sum\limits_{i_1, \ldots, i_r\in[\ell]} h_{1k_{1i_1}}^{u_1} \cdots h_{rk_{1i_1}\cdots k_{ri_r}}^{u_r}m_{j\uk i_1\cdots i_r} = n_{j\uk\uu}}\\
        \ll_{\ell, r, s, t} \eta^{-O_{\ell, r, s,t}(1)}  M^{-2^{r}s|\CK|}H^{-r2^{r-1}s|\CK|}.
    \end{multline*}
(Here, the superscript $u$'s are exponents, not indices.)

\end{proposition}

On a first read, the reader may wish to consider the case $t=\ell$, so that $\CK$ consists of only one element, $k_{li_l}$ can be replaced by $k_l$, and the product $\prod_{\substack{\uk\in\CK}}$ disappears.

In the applications of Proposition \ref{P: systems of multilinear equations} to Theorem \ref{T: concatenation of polynomials}, $s$ is fixed and $\ell$ (the number of terms in each equation) can be taken as large as we want, whereas we have no control over $t$, which grows with $\ell$ and will typically be much larger than $\ell$. 

Before proving Proposition \ref{P: systems of multilinear equations}, we illustrate the statement with several examples.
\begin{example}
    If $s=1$, i.e., we are using Proposition \ref{P: systems of multilinear equations} to concatenate an average of degree-$1$ box norms, then we will always have $t = \ell$, whence $|\CK|=1$. Suppose that $t=\ell = 3$ and $r=2$, i.e., we are dealing with a system of equations coming from concatenating an average
    \begin{align*}
        \E_{h_1, h_2\in[\pm H]}\norm{f}_{(\bbeta_{00} + \bbeta_{10}h_1 + \bbeta_{01}h_2 + \bbeta_{11}h_1h_2)\cdot[\pm M]}^2.
    \end{align*}
    The resulting system consists of equations of the form
    \begin{multline*}
        h_{11}^{u_1}(h_{211}^{u_2}m_{11} + h_{212}^{u_2}m_{12} + h_{213}^{u_2}m_{13}) + h_{12}^{u_1}(h_{221}^{u_2}m_{21} + h_{222}^{u_2}m_{22} + h_{223}^{u_2}m_{23})\\
        + h_{13}^{u_1}(h_{231}^{u_2}m_{31} + h_{232}^{u_2}m_{32} + h_{233}^{u_2}m_{33}) = {n_{u_1 u_2}}
    \end{multline*}
    for $(u_1, u_2)\in\{0,1\}^2$. Since $n_{u_1 u_2}$ lies in $[\pm 9H^{u_1+u_2}M]$, we expect to get a cancellation on the order of
    \begin{align*}
        \ll \prod_{u_1, u_2\in\{0,1\}} H^{-(u_1+u_2)}M\inv = H^{-4}M^{-4},
    \end{align*}
    and modulo the powers of $\eta$, this is indeed what Proposition \ref{P: systems of multilinear equations} gives.
    
\end{example}

\begin{example}
    Suppose now that $r=1$, $s=4$, $t=6$, and $\ell = 4$, so that $|\CK| = 15$. Then we have a system of $30$ equations, which we can group into $15$ pairs of equations
    \begin{align*}
    m_{\uk 1} + m_{\uk 2} +m_{\uk 3}+m_{\uk 4} = n_{\uk 0}\quad \textrm{and}\quad m_{\uk 1} h_{k_1} + m_{\uk 2} h_{k_2}+m_{\uk 3} h_{k_3}+m_{\uk 4} h_{k_4} = n_{\uk 1}
\end{align*}
corresponding to the $15$ tuples $$\uk\in\CK = \{(k_1, k_2, k_3, k_4):\; 1\leq k_1 < k_2 < k_3 < k_4\leq 6\}.$$ Each $n_{\uk 0}$, $n_{\uk 1}$ gives cancellation of order $\asymp M\inv, \asymp H\inv M\inv$ (respectively), so the total cancellation will be of order  $\asymp H^{-15}M^{-30}$ (modulo powers of $\eta$, which we ignore here for the simplicity of exposition). We remark that this example does not figure in the proof of Theorem \ref{T: concatenation of polynomials} (but nonetheless is covered by Proposition \ref{P: systems of multilinear equations}).
\end{example}

Let us see how the heuristics from the preceding two examples extend to the general case. 
For each $(j, \uk, \uu)$, we expect the value $n_{j\uk\uu}$ to be distributed fairly uniformly in an interval $[\pm C_{j\uk\uu} H^{u_1+\cdots + u_r} M]$
for some $C_{j\uk\uu}>0$; the key uniform distribution property follows from Corollary \ref{C: linear congruences} and the fact that the tuples $(h_{lk_1\cdots k_l})_{k_1, \ldots, k_l}$ lie in the sets $\CH_l$. 
So for each $(j, \uk, \uu)$ we get a cancellation of the order of
\begin{align}\label{E: cancellation}
\asymp_{r}(H^{-(u_1+\cdots + u_r)} M\inv),    
\end{align}
which is (up to constants) the probability that a uniformly random element of the interval $[\pm C_{j\uk\uu} H^{u_1+\cdots + u_r} M]$ assumes any given value. The powers of $\eta$ in the bound come from restricting the variables $h_l$ to the admissible ranges $\CH_l$, and using Corollary \ref{C: linear congruences} to bound solutions for $h$'s in this range. The final bound then follows upon taking the product of \eqref{E: cancellation} over all $j\in[s], \uu\in\{0,1\}^r$, and $\uk\in\CK$. 

\begin{proof}[Proof of Proposition~\ref{P: systems of multilinear equations}] We let all implicit constants depend on $\ell, r, s, t$. The proof goes by induction on $r$.  The $r=1$ case is the content of Lemma \ref{L: bound on pinned solutions}.  The induction step is in general quite notationally complicated, so, to make it more palatable, we first present the case $r=2$.

\smallskip
\textbf{The case $r=2$.}
\smallskip

In this case, our claimed bound takes the form

        \begin{multline*}
        \max_{\substack{n_{j\uk00},\; n_{j\uk01},\\ n_{j\uk10},\; n_{j\uk11}\in\Z:\\ j\in[s],\; \uk\in\CK}}\;\E_{\substack{m_{j\uk i_1 i_2}\in[\pm M]:\\ j\in[s],\; \uk\in\CK,\; i_1, i_2\in[\ell]}}
                \E_{\substack{h_{1k_1}, h_{2k_1k_2}\in[\pm H]:\\ k_1, k_2\in[t]}}\; \mathbf{1}_{\CH_1}((h_{1k_1})_{k_1}) \cdot \mathbf{1}_{\CH_2}((h_{2k_1k_2})_{k_1, k_2})\\
        \prod_{j\in[s]}\; \prod_{\substack{\uk\in\CK}}\prod_{\substack{\uu\in\{0,1\}^2
        }} \mathbf{1}\Bigbrac{\sum\limits_{i_1, i_2\in[\ell]} h_{1k_{1i_1}}^{u_1} h_{2k_{1i_1}k_{2i_2}}^{u_2}m_{j\uk i_1 i_2} = n_{j\uk\uu}}\\
        \ll \eta^{-O(1)}  M^{-4s|\CK|} H^{-4s|\CK|}.
    \end{multline*}
    For each $(j,\uk)\in[s]\times \CK$, we have the system of four equations
\begin{align*}
    \sum\limits_{i_1, i_2\in[\ell]} m_{j\uk i_1 i_2} &= n_{j\uk00}\\
    \sum\limits_{i_1, i_2\in[\ell]} h_{1k_{1i_1}} m_{j\uk i_1 i_2} &= n_{j\uk10}\\
    \sum\limits_{i_1, i_2\in[\ell]} h_{2k_{1i_1}k_{2i_2}} m_{j\uk i_1 i_2} &= n_{j\uk01}\\
    \sum\limits_{i_1, i_2\in[\ell]} h_{1k_{1i_1}} h_{2k_{1i_1}k_{2i_2}} m_{j\uk i_1 i_2} &= n_{j\uk11}.
\end{align*}
This system is quadratic in the $h$ variables, and our goal is to decouple it into two separate systems that are linear in $h$. We accomplish this by introducing two dummy variables
\begin{align*}
    t_{i_1 j \uk 0} = \sum_{i_2\in[\ell]}m_{j\uk i_1 i_2}\quad \textrm{and}\quad t_{i_1 j \uk 1} = \sum_{i_2\in[\ell]}h_{2k_{1i_1}k_{2i_2}} m_{j\uk i_1 i_2}
\end{align*}
and renaming $t_{i_1 i_2 j\uk} = m_{j\uk i_1 i_2}$.  Let us mention the motivation for choosing these variables.  The above system of four equations has three different scales: The first is at scale $\asymp M$, the second and third are at scale $\asymp HM$, and the fourth is at scale $\asymp H^2 M$.  In order to apply the two-scale result from Lemma~\ref{L: bound on pinned solutions}, we want to think of the fourth equation as having scale $\asymp H(HM)$, where the $HM$ here ``matches'' the scale of the third equation.  This is precisely the role played by the second set of dummy variables, which is at scale $\asymp HM$.  The first set of dummy variables, at scale $\asymp M$, pertains to the first and second equations, and the two sets of dummy variables are chosen in such a way that they are related at consecutive scales as in Lemma~\ref{L: bound on pinned solutions}. 

Thus, solving the original system is equivalent to solving
\begin{align*}
    \sum_{i_1\in[\ell]} t_{i_1 j \uk 0} = n_{j\uk00}\quad \textrm{and}\quad
    \sum_{i_1\in[\ell]} h_{1k_{1i_1}} t_{i_1 j \uk 0} = n_{j\uk10}\\
    \sum_{i_1\in[\ell]} t_{i_1 j \uk 1} = n_{j\uk01} \quad \textrm{and}\quad
    \sum_{i_1\in[\ell]} h_{1k_{1i_1}} t_{i_1 j \uk 1} = n_{j\uk11},
\end{align*}
together with
\begin{align*}
    \sum_{i_2\in[\ell]}t_{i_1 i_2 j\uk} = t_{i_1 j \uk 0} \quad \textrm{and}\quad
    \sum_{i_2\in[\ell]}h_{2k_{i_1}k_{i_2}}t_{i_1 i_2 j\uk} = t_{i_1 j \uk 1}\quad \textrm{for\; each}\quad i_1\in[\ell],
\end{align*}
where the variables $t_{i_1 j \uk 0}$ and $t_{i_1 j \uk 1}$ have ranges $[\pm \ell M]$ and $[\pm\ell H M]$, respectively.  Converting the expectation over the $m$'s to a sum, we see that our normalized count is bounded by $O(M^{-s \ell^2|\CK|})$ times 
        \begin{align*}
        &\max_{\substack{n_{j\uk00},\; n_{j\uk01},\\ n_{j\uk10},\; n_{j\uk11}\in\Z:\\ j\in[s],\; \uk\in\CK}}\;
                \E_{\substack{h_{1k_1}\in[\pm H]:\\  k_1\in[t]}}\; \mathbf{1}_{\CH_1}((h_{1k_1})_{k_1}) \sum_{\substack{t_{i_1 j\uk 0}\in[\pm\ell M],\\ t_{i_1 j\uk 1}\in[\pm \ell H M]:\\ j\in[s],\; \uk\in\CK,\; i_1\in[\ell]}}\\
        &\qquad\qquad\qquad\qquad\qquad\qquad\Big(\prod_{j\in[s]}\; \prod_{\substack{\uk\in\CK}}\prod_{\substack{\uu\in\{0,1\}^2}} \mathbf{1}\Bigbrac{\sum\limits_{i_1\in[\ell]} h_{1k_{1i_1}}^{u_1} t_{i_1 j\uk u_2} = n_{j\uk\uu}}\Big)\\
        &\qquad\qquad\qquad\times \Big(\E_{\substack{h_{2k_1k_2}\in[\pm H]:\\  k_1, k_2\in[t]}}\; \mathbf{1}_{\CH_2}((h_{2k_1k_2})_{k_1, k_2})
        \sum_{\substack{t_{i_1 i_2 j\uk}\in[\pm M]:\\ j\in[s],\; \uk\in\CK,\; i_1, i_2\in[\ell]}}\\
        &\qquad\qquad\qquad\qquad\qquad\qquad\prod_{i_1\in[\ell]}\prod_{j\in[s]}\; \prod_{\substack{\uk\in\CK}}\prod_{\substack{u_2\in\{0,1\}}} \mathbf{1}\Bigbrac{\sum\limits_{i_2\in[\ell]} h_{2k_{1i_1}k_{2i_2}}^{u_2} t_{i_1 i_2 j\uk} = t_{i_1 j\uk u_2}}\Big);
    \end{align*}
    we will deal first with the system involving the variables $h_{2k_1k_2}$ and then with the system involving the variables $h_{1k_1}$. Note that
    \begin{align*}
        &\sum_{\substack{t_{i_1 i_2 j\uk}\in[\pm M]:\\ j\in[s],\; \uk\in\CK,\; i_1, i_2\in[\ell]}}\prod_{i_1\in[\ell]}\prod_{j\in[s]}\; \prod_{\substack{\uk\in\CK}}\prod_{\substack{u_2\in\{0,1\}}} \mathbf{1}\Bigbrac{\sum\limits_{i_2\in[\ell]} h_{2k_{1i_1}k_{2i_2}}^{u_2} t_{i_1 i_2 j\uk} = t_{i_1 j\uk u_2}}\\
        &\qquad=\prod_{i_1\in[\ell]}\prod_{j\in[s]}\; \prod_{\substack{\uk\in\CK}}\sum_{\substack{t_{i_1 i_2 j\uk}\in[\pm M]:\\ i_2\in[\ell]}}\; \mathbf{1}\Bigbrac{\sum\limits_{i_2\in[\ell]} t_{i_1 i_2 j\uk} = t_{i_1 j\uk 0}}\cdot \mathbf{1}\Bigbrac{\sum\limits_{i_2\in[\ell]} h_{2k_{1i_1}k_{2i_2}} t_{i_1 i_2 j\uk} = t_{i_1 j\uk 1}}.
    \end{align*}
    Thus, we have
    \begin{align*}
        &\E_{\substack{h_{2k_1k_2}\in[\pm H]:\\  k_1, k_2\in[t]}}\; \mathbf{1}_{\CH_2}((h_{2k_1k_2})_{k_1, k_2})
        \sum_{\substack{t_{i_1 i_2 j\uk}\in[\pm M],\\ j\in[s],\; \uk\in\CK,\; i_1, i_2\in[\ell]}}\\
        &\qquad\qquad\qquad\qquad\prod_{i_1\in[\ell]}\prod_{j\in[s]}\; \prod_{\substack{\uk\in\CK}}\prod_{\substack{u_2\in\{0,1\}}} \mathbf{1}\Bigbrac{\sum\limits_{i_2\in[\ell]} h_{2k_{1i_1}k_{2i_2}}^{u_2} t_{i_1 i_2 j\uk} = t_{i_1 j\uk u_2}}\\
        &\qquad\qquad= \E_{\substack{h_{2k_1k_2}\in[\pm H]:\\  k_1, k_2\in[t]}}\; \mathbf{1}_{\CH_2}((h_{2k_1k_2})_{k_1, k_2})\prod_{i_1\in[\ell]}\prod_{j\in[s]}\; \prod_{\substack{\uk\in\CK}}\\ 
        &\qquad\qquad\qquad\qquad\sum_{\substack{t_{i_1 i_2 j\uk}\in[\pm M]:\\ i_2\in[\ell]}}\; \mathbf{1}\Bigbrac{\sum\limits_{i_2\in[\ell]} t_{i_1 i_2 j\uk} = t_{i_1 j\uk 0}}\cdot \mathbf{1}\Bigbrac{\sum\limits_{i_2\in[\ell]} h_{2k_{1i_1}k_{2i_2}} t_{i_1 i_2 j\uk} = t_{i_1 j\uk 1}}.
    \end{align*}
    Lemma \ref{L: bound on pinned solutions} applied with $l=r=2$ allows us to bound this expression by 
    $$O(\eta^{-O(1)} M^{s\ell^2|\CK|-2s\ell|\CK|} H^{-s\ell|\CK|}).$$
    Similarly, the count of solutions to the system involving the variables $h_{1k_1}$ is
\begin{align*}
        &\max_{\substack{n_{j\uk00},\; n_{j\uk01},\\ n_{j\uk10},\; n_{j\uk11}\in\Z:\\ j\in[s],\; \uk\in\CK}}\;
        \E_{\substack{h_{1k_1}\in[\pm H]:\\  k_1\in[t]}}\; \mathbf{1}_{\CH_1}((h_{1k_1})_{k_1}) \sum_{\substack{t_{i_1 j\uk 0}\in[\pm\ell M],\\ t_{i_1 j\uk 1}\in[\pm\ell H M]:\\ j\in[s],\; \uk\in\CK,\; i_1\in[\ell]}}\\
        &\qquad\qquad\qquad\prod_{j\in[s]}\; \prod_{\substack{\uk\in\CK}}\prod_{\substack{\uu\in\{0,1\}^2}} \mathbf{1}\Bigbrac{\sum\limits_{i_1\in[\ell]} h_{1k_{1i_1}}^{u_1} t_{i_1 j\uk u_2} = n_{j\uk\uu}}\\
        &\qquad\leq  \E_{\substack{h_{1k_1}\in[\pm H]:\\  k_1\in[t]}}\; \mathbf{1}_{\CH_1}((h_{1k_1})_{k_1}) \prod_{j\in[s]}\; \prod_{\substack{\uk\in\CK}}\prod_{\substack{u_2\in\{0,1\}}}\\
        &\qquad\qquad\qquad\max_{n_0, n_1\in\Z}\; \sum_{\substack{t_{i_1 j\uk u_2}\in[\pm\ell H^{u_2} M]:\\  i_1\in[\ell]}}\mathbf{1}\Bigbrac{\sum\limits_{i_1\in[\ell]} t_{i_1 j\uk u_2} = n_0}\cdot \mathbf{1}\Bigbrac{\sum\limits_{i_1\in[\ell]} h_{1k_{1i_1}} t_{i_1 j\uk u_2} = n_1},
\end{align*}
and this can be bounded by
$$O(\eta^{-O(1)} M^{2 s\ell |\CK|-4s|\CK|} H^{s\ell|\CK|-4s|\CK|})$$
using the case $l=1, r=2$ of Lemma \ref{L: bound on pinned solutions}. Note that the big centered equations had an  \emph{inequality} rather than an \emph{equality} because we interchanged the $\max$ with the sums and products.

    Combining these two bounds and the normalising factor $O(M^{-s \ell^2|\CK|})$, we bound the overall normalised count by
    \begin{align*}
        &\ll M^{-s \ell^2|\CK|}\cdot \brac{\eta^{-O(1)} M^{s\ell^2|\CK|-2s\ell|\CK|} H^{-s\ell|\CK|}}\cdot \brac{\eta^{-O(1)} M^{2 s\ell|\CK|-4s|\CK|} H^{s\ell|\CK|-4s|\CK|}}\\
        &= \eta^{-O(1)} M^{-4s|\CK|}H^{-4s|\CK|},
    \end{align*}
    as claimed.

\smallskip
\textbf{The case $r>2$.}
\smallskip

We handle the case $r>2$ similarly. For $l\in[r]$, $i_1, \ldots, i_l\in[\ell]$, $j\in[s]$, $\uk\in\CK$, $\uu\in\{0,1\}^r$, we define the auxiliary variable
\begin{align*}
    t_{i_1\cdots i_l j \uk\uu} =\sum\limits_{i_{l+1}, \ldots, i_r\in[\ell]}h_{(l+1)k_{1i_1} \cdots k_{(l+1)i_{l+1}}}^{u_{l+1}}\cdots h_{r k_{1i_1}\cdots k_{ri_r}}^{u_r}m_{j\uk i_1\cdots i_r}
\end{align*}
at scale $H^{u_{l+1}+\cdots+u_r}M$.
Notice that $t_{i_1\cdots i_l j \uk\uu}$ depends on only the coordinates $u_{l+1}, \ldots, u_r$ of $\uu$ and that the $t$'s satisfy the relation
\begin{align}\label{E: recurrence relation on t}
    t_{i_1\cdots i_l j \uk\uu} = \sum_{i_{l+1}\in[\ell]}h_{(l+1)k_{1i_1} \cdots k_{(l+1)i_{l+1}}}^{u_{l+1}} t_{i_1\cdots i_{l+1} j \uk\uu}.
\end{align}
Iterating \eqref{E: recurrence relation on t} through intermediate scales, we see that {for each $j \in [s]$ and $\uk \in \CK$,} we can rewrite
\begin{align*}
        &\sum_{\substack{m_{j\uk i_1\cdots i_r}\in[\pm M]:\\ 
        i_1, \ldots,  i_r\in[\ell]}} \mathbf{1}\Bigbrac{\sum\limits_{i_1, \ldots, i_r\in[\ell]} h_{1k_{1i_1}}^{u_1} \cdots h_{rk_{1i_1}\cdots k_{ri_r}}^{u_r}m_{j\uk i_1\cdots i_r} = n_{j\uk\uu}} \\
        &\qquad\qquad\qquad=\sum_{\substack{|t_{i_1 j\uk\uu}|\ll H^{u_2+\cdots + u_r}M:\\ i_1\in[\ell]}} \mathbf{1}\Bigbrac{\sum\limits_{i_1\in[\ell]} h_{1k_{1i_1}}^{u_1} t_{i_1 j\uk\uu}= n_{j\uk\uu}}\\ 
    &\qquad\qquad\qquad\times \sum_{\substack{|t_{i_1i_2 j\uk\uu}|\ll H^{u_3+\cdots +u_r}M:\\ i_1, i_2\in[\ell]}}\;\prod_{i_1\in[\ell]} \mathbf{1}\Bigbrac{\sum\limits_{i_2\in[\ell]} h_{2k_{1i_1}k_{2i_2}}^{u_2}t_{i_1 i_2 j\uk\uu} = t_{i_1 j\uk\uu}}\\
    &\qquad\qquad\qquad\times \cdots \times  \sum_{\substack{|t_{i_1\cdots i_r j\uk\uu}|\leq M:\\ i_1, \cdots, i_r\in[\ell]}}\;\prod_{i_1,\ldots, i_{r-1}\in[\ell]} \mathbf{1}\Bigbrac{\sum\limits_{i_r\in[\ell]} h_{rk_{1i_1}\cdots k_{r i_r}}^{u_r}t_{i_1 \cdots i_r j\uk\uu} = t_{i_1 \ldots i_{r-1} j\uk\uu}},
\end{align*}
where we note that $t_{i_1 \cdots i_r j\uk\uu} = m_{j\uk i_1\cdots i_r}$. Hence our normalized count is bounded by $O(M^{-s\ell^r|\CK|})$ times
\begin{align*}
            &\max_{\substack{n_{j\uk\uu}\in\Z:\; j\in[s],\\ \uk\in\CK,\; \uu\in\{0,1\}^r}}\; \E_{\substack{h_{1k_1}\in[\pm H]:\\ k_1\in[t]}}\; \mathbf{1}_{\CH_1}((h_{1 k_1})_{k_1})\sum_{\substack{|t_{i_1 j\uk\uu}|\ll H^{u_2+\cdots +u_r}M:\; j\in[s],\\ \uk\in\CK,\; u_2, \ldots, u_r\in\{0,1\},\; i_1\in[\ell]}}\\
        &\qquad\qquad\qquad\qquad\Big(\prod_{j\in[s]}\; \prod_{\substack{\uk\in\CK}}\; \prod_{\substack{u_1, \ldots, u_r\in\{0,1\}}}\mathbf{1}\Bigbrac{\sum\limits_{i_1\in[\ell]} h_{1k_{1i_1}}^{u_1} t_{i_1 j\uk\uu}= n_{j\uk\uu}}\Big)
        \\
        &\qquad\qquad\times \E_{\substack{h_{2k_1k_2}\in[\pm H]:\\ k_1,k_2\in[t]}}\; \mathbf{1}_{\CH_2}((h_{2 k_1 k_2})_{k_1, k_2})\; \sum_{\substack{|t_{i_1i_2 j\uk\uu}|\ll H^{u_3+\cdots+u_r}M:\; j\in[s],\\ \uk\in\CK,\; u_3, \ldots, u_r\in\{0,1\},\; i_1, i_2\in[\ell]}}\\
        &\qquad\qquad\qquad\qquad
        \Big(\prod_{i_1\in[\ell]}\; \prod_{j\in[s]}\; \prod_{\substack{\uk\in\CK}}\; \prod_{\substack{u_2, \ldots, u_r\in\{0,1\}}} \mathbf{1}\Bigbrac{\sum\limits_{i_2\in[\ell]} h_{2k_{1i_1}k_{2i_2}}^{u_2}t_{i_1 i_2 j\uk\uu} = t_{i_1 j\uk\uu}}\Big)\\
        &\qquad\qquad\times \cdots \times \E_{\substack{h_{rk_1\cdots k_r}\in[\pm H]:\\ k_1,\ldots, k_r\in[t]}}\; \mathbf{1}_{\CH_r}((h_{r k_1 \cdots k_r})_{k_1, \ldots, k_r})\; \sum_{\substack{|t_{i_1\cdots i_r j\uk\uu}|\leq M:\; j\in[s],\\ \uk\in\CK,\; i_1,\ldots,i_r\in[\ell]}}   \\
        &\qquad\qquad\qquad\qquad
        \prod_{i_1,\ldots, i_{r-1}\in[\ell]}\; \prod_{j\in[s]}\; \prod_{\substack{\uk\in\CK}}\; \prod_{\substack{u_r\in\{0,1\}}}
         \mathbf{1}\Bigbrac{\sum\limits_{i_r\in[\ell]} h_{rk_{1i_1}\cdots k_{r i_r}}^{u_r}t_{i_1 \cdots i_r j\uk\uu} = t_{i_1 \cdots i_{r-1} j\uk\uu}}.
\end{align*}
    As in the case $r=2$, we have decoupled our original system of degree-$r$ multilinear equations into $r$ smaller systems, each consisting of a pair of multilinear equations at consecutive scales.  We will next use Lemma \ref{L: bound on pinned solutions} to bound the counts of solutions for these subsystems individually, starting from the system involving the variables $h_{rk_1\cdots k_r}$. The count of solutions to this system can be rephrased as
    \begin{multline*}
\E_{\substack{h_{rk_1\cdots k_r}\in[\pm H]:\\ k_1,\ldots, k_r\in[t]}}\; \mathbf{1}_{\CH_r}((h_{r k_1 \cdots k_r})_{k_1, \ldots, k_r})  \prod_{i_1,\ldots, i_{r-1}\in[\ell]}\; \prod_{j\in[s]}\; \prod_{\substack{\uk\in\CK}}\\
        \sum_{\substack{|t_{i_1\cdots i_r j\uk\uu}|\leq M:\\ i_r\in[\ell]}}\ 
        \prod_{\substack{u_r\in\{0,1\}}}
         \mathbf{1}\Bigbrac{\sum\limits_{i_r\in[\ell]} h_{rk_{1i_1}\cdots k_{r i_r}}^{u_r}t_{i_1 \cdots i_r j\uk\uu} = t_{i_1 \cdots i_{r-1} j\uk\uu}},
    \end{multline*}
    and bounded by $\ll \eta^{-O(1)}M^{s\ell^{r}|\CK|-2s\ell^{r-1}|\CK|}H^{-s\ell^{r-1}|\CK|}$ using the case $l=r$ of Lemma \ref{L: bound on pinned solutions}. Similarly, the count corresponding to the system involving the variables $h_{(r-1)k_1\cdots k_{r-1}}$ can be rephrased as
    \begin{multline*}
\E_{\substack{h_{(r-1)k_1\cdots k_{r-1}}\in[\pm H]:\\ k_1,\ldots, k_{r-1}\in[t]}}\; \mathbf{1}_{\CH_{r-1}}((h_{(r-1) k_1 \cdots k_{r-1}})_{k_1, \ldots, k_{r-1}})  \prod_{i_1,\ldots, i_{r-2}\in[\ell]}\; \prod_{j\in[s]}\; \prod_{\substack{\uk\in\CK}}\; \prod_{\substack{u_{r}\in\{0,1\}}}\\
        \sum_{\substack{|t_{i_1\cdots i_{r-1} j\uk\uu}|\ll H^{u_r} M:\\ i_{r-1}\in[\ell]}}\ 
        \prod_{\substack{u_{r-1}\in\{0,1\}}}
         \mathbf{1}\Bigbrac{\sum\limits_{i_{r-1}\in[\ell]} h_{(r-1)k_{1i_1}\cdots k_{(r-1) i_{r-1}}}^{u_{r-1}}t_{i_1 \cdots i_{r-1} j\uk\uu} = t_{i_1 \cdots i_{r-2} j\uk\uu}}
    \end{multline*}  
    and bounded by
    \begin{align*}
        \ll \eta^{-O(1)}M^{2s\ell^{r-1}|\CK|-s\ell^{r}|\CK|-4s\ell^{r-1}|\CK|}H^{s\ell^{r-1}|\CK|-4s\ell^{r-1}|\CK|}
    \end{align*}
    with the help of the case $l=r-1$ of Lemma \ref{L: bound on pinned solutions}. Continuing in this manner and combining the results, we can bound our initial normalized count by
    \begin{align*}
        &\ll M^{-s\ell^r|\CK|}\cdot \prod_{l\in[r]} \eta^{-O(1)}  M^{2^{r-l}s\ell^l|\CK|-2^{r-l+1}s\ell^{l-1}|\CK|}\\ 
        &\qquad\qquad\qquad\qquad\qquad\qquad H^{(r-l)2^{r-l-1}s\ell^l|\CK| -(r-l+1)2^{r-l} s\ell^{l-1}|\CK|},
    \end{align*}
    where in each $l$-multiplicand, the negative exponents of $\eta, M, H$ come from an application of Lemma \ref{L: bound on pinned solutions} to bound the count of solutions to the system involving the variables $h_{lk_1\cdots k_l}$ and the positive exponents come from counting the admissible values of the variables $t_{i_1\cdots i_l j\uk\uu}$. The telescoping nature of the exponents leads to cancellations: The exponents of $M$ and $H$ cancel out as
    \begin{gather*}
        -s\ell^r|\CK| + (s\ell^r|\CK| - 2s\ell^{r-1}|\CK|) + \cdots + (2^{r-1}s\ell|\CK| - 2^r s|\CK|) = -2^r s|\CK|,\\
        -s\ell^{r-1}|\CK| + (s\ell^{r-1}|\CK| - 4s\ell^{r-2}|\CK|) + \cdots + ((r-1)2^{r-2}s\ell|\CK| - r2^{r-1}s|\CK|) = -r2^{r-1}s|\CK|,
    \end{gather*}
    respectively, yielding the desired final bound
    \begin{align*}
        \ll\eta^{-O(1)} M^{-2^r s|\CK|} H^{-r2^{r-1}s|\CK|}.
    \end{align*} 
    \end{proof}

The last result in this section ensures that the set $\CH_{l, \eta}$ fills almost all of $[\pm H]^{t^l}$ when $\eta>0$ is small.
\begin{lemma}\label{L: H_l}
    Let $\eta>0$, $H, l, t\in\N$, and let $\CH = \CH_{l,\eta}$ be as in \eqref{E: H_l}. Then all but an $O_{l,t}(\eta)$-proportion of the elements of $[\pm H]^{t^l}$ lie in $\CH$.
\end{lemma}
\begin{proof}
    Setting $k = (k_1, \ldots, k_l)$ and $h_k = h_{lk_1, \ldots, k_l}$, we can rewrite
    \begin{align*}
        \CH = \{(h_k)_{k\in[t^l]}\in[\pm H]^{t^l}:\; &\gcd(h_k - h_{k''}, h_{k'}-h_{k''})\leq \eta\inv,\\ &|h_k|\geq\eta H\; \textrm{for all distinct}\; k,k',k''\in[t^l]\}.
    \end{align*}
    We will show that there are very few elements of $[\pm H]^{t^l}$ failing each condition for membership in $\CH$, and then we will use the union bound to conclude that $[\pm H]^{t^l} \setminus \CH$ is small.  For the gcd condition, we estimate
    \begin{align*}
        \E_{h, h',h''\in[\pm H]}\mathbf{1}\brac{\gcd(h-h'', h'-h'')\geq \eta\inv}&\ll \E_{h, h'\in[\pm 2H]}\mathbf{1}\brac{\gcd(h, h')\geq \eta\inv}\\
        &\ll \frac{1}{H^2}\sum_{g=\eta\inv}^{2H}\sum_{h,h'\in[\pm 2H]}\mathbf{1}\brac{g|h,h'}\\
        &\ll \frac{1}{H^2}\sum_{g=\eta\inv}^{2H}(H/g)^2\\
        &= \sum_{g=\eta\inv}^{2H}\frac{1}{g^2}\ll\eta.
    \end{align*}
For the size condition, we note that $|h|<\eta H$ for an $O(\eta)$-fraction of $h \in [\pm H]$.  The conclusion of the lemma now follows from the union bound.
\end{proof}

\section{Concatenation along polynomials}\label{S: concatenation along polys}
In this section, we combine the conclusions of the previous two sections in order to complete the proof of Theorem \ref{T: concatenation of polynomials}. First, we specialize the results of Section \ref{S: general concatenation} to averages of polynomial box norms.
\begin{proposition}\label{P: concatenation of polynomials I}
    Let $D,d, r, s\in\N$ with $d\leq r$, and let $\ell$ be a power of $2$. There exist a positive integer $t=O_{\ell, r, s}(1)$ and a positive real $C=O_{D,\ell,r,s}(1)$ such that the following holds.  Let $\delta>0$ and $H, M, N, V\in\N$ satisfy $H^d M V \leq N$, and let $\bc_1, \ldots, \bc_s\in\Z^D[\uh]$ be polynomials of degree at most $d$ with coefficients of size at most $V$.  Then for all $1$-bounded functions $f:\Z^D\to\C$ supported on $[N]^D$, the bound
    \begin{align}
        \E_{\uh\in[\pm H]^r}\norm{f}_{\bc_1(\uh)\cdot[\pm M], \ldots, \bc_s(\uh)\cdot[\pm M]}^{2^s}\geq \delta N^D
    \end{align}
    implies that
        \begin{align*}
        \E_{\substack{h_{lk_1\cdots k_l}\in[\pm H]:\\ l\in[r],\; k_1, \ldots, k_r\in[t]}}
        \norm{f}_{\substack{\{\sum_{i_1, \ldots, i_r=1}^{\ell}\bc_j(h_{1k_{1i_1}}, \ldots, h_{rk_{1i_1}\cdots k_{ri_r}})\cdot[\pm M]:\\ j\in[s],\; \uk \in \CK
 \}}}^{2^{w}} \geq \frac{1}{C}\delta^{C} N^D,
    \end{align*}
    where $\CK=\CK_{t,\ell,r}$ and $w$ is the degree of the box norms appearing.
\end{proposition}
\begin{proof}
    The proof consists of several applications of Corollary \ref{C: iterated concatenation for general groups}. First, we apply it with the indexing set $\uh\in[\pm H]^r$ to obtain a natural number $t_1 = O_{\ell, s}(1)$ such that
    \begin{align*}
        \E_{\uh_1, \ldots, \uh_{t_1}\in[\pm H]^r}\norm{f}_{\substack{\{\bc_j(\uh_{k_{11}})\cdot[\pm M]+\cdots + \bc_j(\uh_{k_{1 \ell}}) \cdot[\pm M]:\\ j\in[s],\; 1\leq k_{11} < \cdots < k_{1 \ell}\leq t_1\}}}^{2^{w_1}}\gg_{D, \ell, s} \delta^{O_{\ell, s}(1)} N^D,
    \end{align*}
    where $w_1$ is the degree of the box norms appearing. 

    For the second application of Corollary \ref{C: iterated concatenation for general groups}, we write $\uh_k = (h_{1k}, \Tilde{\uh}_{1k})$ and take the indexing set to be the set of tuples $(\tilde{\uh}_{11}, \ldots, \tilde{\uh}_{1 t_1})\in[\pm H]^{(r-1)t_1}$. Then Corollary \ref{C: iterated concatenation for general groups}, applied separately for each fixed choice of $h_{11}, \ldots, h_{1t_1}$, gives $t_2 = O_{\ell, s}(1)$ for which
    \begin{multline*}
        \E_{\substack{h_{1k_1}\in[\pm H]:\\ k_1\in[t_1]}}\;
        \E_{\substack{\Tilde{\uh}_{1k_1 k_2}\in[\pm H]^{r-1}:\\ k_1\in [t_1],\; k_2\in[t_2]}}\;
        \norm{f}_{\substack{\{\bc_j(h_{1k_{11}}, \Tilde{\uh}_{1k_{11}k_{21}})\cdot[\pm M]+\cdots + \bc_j(h_{1k_{11}}, \Tilde{\uh}_{1k_{11}k_{2\ell}})\cdot[\pm M]\\
        +\cdots + \bc_j(h_{1k_{1\ell}}, \Tilde{\uh}_{1k_{1\ell}k_{21}})\cdot[\pm M]+\cdots + \bc_j(h_{1k_{1\ell}}, \Tilde{\uh}_{1k_{1\ell}k_{2\ell}})\cdot[\pm M]:\\ j\in[s],\; 1\leq k_{11} < \cdots < k_{1 \ell}\leq t_1,\; 1\leq k_{21} < \cdots < k_{2 \ell}\leq t_2 \}}}^{2^{w_2}}\\ \gg_{D, \ell, s} \delta^{O_{\ell, s}(1)} N^D,
    \end{multline*}
    where $w_2$ is the degree of the box norms appearing.

    For the next step, we similarly write $\Tilde{\uh}_{1k_1k_2}=(h_{2k_1k_2}, \Tilde{h}_{2k_1k_2})$ and take our indices to be the tuples $(\tilde{h}_{2k_1k_2})_{k_1\in[t_1], k_2\in[t_2]}\in[\pm H]^{(r-2)t_1 t_2}$.  Applying Corollary \ref{C: iterated concatenation for general groups} separately for each $h_{1k_1}, h_{2k_1k_2}\in[\pm H]$, with $k_1\in[t_1], k_2\in[t_2]$, we obtain $t_3 = O_{\ell, s}(1)$ for which
    \begin{multline*}
        \E_{\substack{h_{1k_1}\in[\pm H]:\\ k_1\in[t_1]}}\; \E_{\substack{h_{1k_1 k_2}\in[\pm H]:\\ k_1\in [t_1],\; k_2\in[t_2]}}\;
        \E_{\substack{\Tilde{\uh}_{2k_1 k_2 k_3}\in[\pm H]^{r-2}:\\ k_1\in [t_1],\; k_2\in[t_2],\; k_3\in[t_3]}}\\
        \norm{f}_{\substack{\{\sum_{i_1, i_2, i_3=1}^{\ell}\bc_j(h_{1k_{1i_1}}, h_{2k_{1i_1}k_{2i_2}}, \Tilde{\uh}_{2k_{1i_1}k_{2i_2}k_{3i_3}})\cdot[\pm M]:\\ j\in[s],\; 1\leq k_{l1} < \cdots < k_{l \ell}\leq t_l,\; l\in[3] \}}}^{2^{w_3}}\\ \gg_{D, \ell, s} \delta^{O_{\ell, s}(1)} N^D,
    \end{multline*}
    where $w_3$ is the degree of the box norms appearing. Continuing in this fashion $r-2$ more times, we obtain $t_3, \ldots, t_r=O_{\ell, r, s}(1)$ such that 
        \begin{align*}
        \E_{\substack{h_{lk_1\cdots k_l}\in[\pm H]:\\ l\in[r],\; k_l\in[t_l]}}\;
        \norm{f}_{\substack{\{\sum_{i_1, \ldots, i_r=1}^{\ell}\bc_j(h_{1k_{1i_1}}, \ldots, h_{rk_{1i_1}\cdots k_{ri_r}})\cdot[\pm M]:\\ j\in[s],\; 1\leq k_{l1} < \cdots < k_{l \ell}\leq t_l,\; l\in[r]\}}}^{2^{w_r}} \gg_{D, \ell, r, s} \delta^{O_{\ell, r, s}(1)} N^D,
    \end{align*}
    where $w_r$ is the degree of the box norms appearing.

Although this inequality is just as useful as the conclusion of the proposition, it will later be notationally convenient to replace the $t_l$'s with a single parameter. Setting $t=\max(t_1, \ldots, t_r)$ and noting that the coordinates of the polynomial $$\sum_{i_1, \ldots, i_r=1}^{\ell}\bc_j(h_{1k_{1i_1}}, \ldots, h_{rk_{1i_1}\cdots k_{ri_r}})$$ are at most $O_{d,r,\ell}(V H^d)$ {for all valid choices of the $h$'s}, we deduce that
    \begin{align*}
        \sum_{i_1, \ldots, i_r=1}^{\ell}\bc_j(h_{1k_{1i_1}}, \ldots, h_{rk_{1i_1}\cdots k_{ri_r}})\cdot[\pm M]\subset[-A N, AN]^D
    \end{align*}
    for some $A=O_{d, \ell, r, s}(1)$; it is here that we are using the assumption on the coefficients of $\bc_1, \ldots, \bc_s$ and the bound $H^d M V \leq N$. We then deduce the result from the monotonicity property of box norms (Lemma \ref{L: properties of box norms}\eqref{i: monotonicity}). 
\end{proof} 

We will finally use Proposition \ref{P: concatenation of polynomials I}, together with the equidistribution results from Section \ref{S: equidistribution}, to prove Theorem \ref{T: concatenation of polynomials}, which we restate for convenience.

\begin{theorem*}[Theorem \ref{T: concatenation of polynomials}]
Let $d, r,s,D \in \N$ with $d\leq r$.  There exist a positive integer $t=O_{r, s}(1)$ and a positive real $C=O_{r,s,D}(1)$ such that the following holds.  Let $H, M, N, V\in\N$ and $\delta>0$ satisfy
    \begin{align*}
        C^{-1}\delta^{-C}\leq H\leq M\quad \textrm{and}\quad H^d M V\leq N,
    \end{align*}
    and let $\bc_1, \ldots, \bc_s\in\Z^D[h_1, \ldots, h_r]$ be multilinear polynomials $\bc_j(\uh) = \sum\limits_{\uu\in\{0,1\}^r} \bgamma_{j\uu}\uh^\uu$ of degree at most $d$ and with coefficients at most $V$.  Then for all $1$-bounded functions $f:\Z^D\to\C$ supported on $[N]^D$, the bound
    \begin{align*}
        \E_{\uh\in[\pm H]^r}\norm{f}_{\bc_1(\uh)\cdot[\pm M], \ldots, \bc_s(\uh)\cdot[\pm M]}^{2^s}\geq \delta N^D
    \end{align*}
        implies that
    \begin{align*}
        \norm{f}_{E_1^{ t}, \ldots, E_s^{ t}}^{2^{st}}\geq \frac{1}{C}\delta^{C}N^{D},
    \end{align*}
    where
    \begin{align*}
        E_j = \sum_{\uu\in\{0,1\}^d}\bgamma_{j\uu}\cdot [\pm H^{|\uu|+1}].
    \end{align*} 
\end{theorem*}

\begin{proof}
The proof amounts to combining the conclusions of Propositions \ref{P: concatenation of polynomials I} and \ref{P: systems of multilinear equations}. We will apply these results with $\ell = 4$ since we need $\ell$ to be a power of 2 (as required by Proposition \ref{P: concatenation of polynomials I}), and we want it to be at least 3 (in order to apply Proposition  \ref{P: systems of multilinear equations}). We let all implicit constants depend on $D, r, s$\footnote{and also on $d$, which, however, is bounded by $r$} and note that the powers of $\delta$ appearing in the proof will be independent of $D$. We also let $\CK = \CK_{t, 4, r}$, which we recall from Section \ref{S: equidistribution} is the set of admissible tuples of indices
    \begin{align*}
        \CK = \{(k_{li})_{\substack{(l,i)\in[r]\times[4]}}\in[t]^{4r}:\; 1\leq k_{l1}<\cdots < k_{l4}\leq t\;\; \textrm{for\; all}\;\; l\in[r]\}.
    \end{align*}

By Proposition \ref{P: concatenation of polynomials I}, there is some positive integer $t=O(1)$ (independent of $D$) for which
    \begin{align}\label{E: unsmoothedconcatenated}
        \E_{\substack{h_{lk_1\cdots k_l}\in[\pm H]:\\ l\in[r],\; k_1, \ldots, k_r\in[t]}}
        \norm{f}_{\substack{\{\sum\limits_{i_1, \ldots, i_r\in[\ell]}\bc_j(h_{1k_{1i_1}}, \ldots, h_{rk_{1i_1}\cdots k_{ri_r}})\cdot[\pm M]:\\ j\in[s],\;
        \uk\in\CK\}}}^{2^{w}} \gg \delta^{O(1)} N^D,
    \end{align}
    where $w$ is the degree of the box norms appearing.  We will now make use of our equidistribution estimates.  It is necessary first to restrict to ``generic'' tuples of the $h$'s.

Recall from \eqref{E: H_l} the definition of the set of tuples $\CH_{l,\eta}\subset[\pm H]^{t^l}$, which by Lemma \ref{L: H_l} contains all but an $O(\eta)$-proportion of the elements of $[\pm H]^{t^l}$. Write $\CH_l = \CH_{l, \eta}$ and take $\eta = c\delta^{1/c}$ for some $c>0$ sufficiently small that the contribution of tuples $(h_{l k_1\cdots k_l})_{k_1, \ldots, k_l} \in [\pm H]^{l^t} \setminus \CH_l$ is at most half (say) of the lower bound in \eqref{E: unsmoothedconcatenated}.  Removing such ``bad'' tuples of $h$'s and expanding \eqref{E: unsmoothedconcatenated}, we obtain
    \begin{multline}\label{E: expanded}
        \E_{\substack{h_{lk_1\cdots k_l}\in[\pm H]:\\ l\in[r],\; k_1, \ldots, k_r\in[t]}}\; \brac{ \prod_{l\in[r]}\mathbf{1}_{\CH_l}((h_{l k_1\cdots k_l})_{k_1, \ldots, k_l})}\sum_{\substack{m_{j\uk i_1\cdots i_r}\in[\pm M]:\\ j\in[s],\; \uk\in\CK,\; i_1, \ldots,  i_r\in[4]}} \mu_M((m_{j\uk i_1\cdots i_r})_{j, \uk, i_1, \ldots, i_r})\\
        \sum_x \Delta_{\substack{\{\sum\limits_{i_1, \ldots, i_r\in[4]}\bc_j(h_{1k_{1i_1}}, \ldots, h_{rk_{1i_1}\cdots k_{ri_r}})m_{j \uk i_1\cdots i_r}:\\ j\in[s],\; \uk\in\CK\}}}\; f(x) \gg \delta^{O(1)} N^D.
    \end{multline}
For each $j\in[s]$, $\uk\in\CK$, $\uu\in\{0,1\}^r$, we introduce dummy variables 
    \begin{align}\label{E: n_jku}
        n_{j\uk\uu} = \sum\limits_{i_1, \ldots, i_r\in[4]} h_{1k_{1i_1}}^{u_1} \cdots h_{rk_{1i_1}\cdots k_{ri_r}}^{u_r}m_{j\uk i_1\cdots i_r}
    \end{align}
    at scale $H^{u_1+\cdots + u_r}M$, and we define \begin{multline*}
        \CN(\un) = \E_{\substack{m_{j\uk i_1\cdots i_r}\in[\pm M]:\\ j\in[s],\; \uk\in\CK,\; i_1, \ldots,  i_r\in[4]}} \; 
        \E_{\substack{h_{lk_{1}\cdots k_{l}}\in[\pm H]:\\ l\in[r],\; k_1, \ldots, k_r\in[t] }}\; \brac{\prod_{l\in[r]}\mathbf{1}_{\CH_l}((h_{l k_1\cdots k_l})_{k_1, \ldots, k_l})} \\
        \prod_{j\in[s]}\; \prod_{\substack{\uk\in\CK}}\prod_{\substack{\uu\in\{0,1\}^r
        }} \mathbf{1}\brac{\sum\limits_{i_1, \ldots, i_r\in[4]} h_{1k_{1i_1}}^{u_1} \cdots h_{rk_{1i_1}\cdots k_{ri_r}}^{u_r}m_{j\uk i_1\cdots i_r} = n_{j\uk\uu}}.
    \end{multline*}
    In words, the kernel $\CN(\un)$ is the probability that \eqref{E: n_jku} is satisfied when the $m$'s are chosen (independently, uniformly) at random from $[\pm M]$ and the tuples of $h$'s are chosen (independently, uniformly) at random from the sets $\CH_l$.
    Inserting the definition of $\CN(\un)$ into \eqref{E: expanded} and applying the triangle inequality gives us the neat bound
    \begin{align*}
        \sum_{\substack{n_{j\uk\uu}\ll H^{|\uu|}M:\\ j\in[s],\; \uk\in\CK,\; \uu\in\{0,1\}^r 
        }}\CN(\un)\cdot\abs{\sum_\bx \Delta_{\{\sum_{\substack{\uu\in\{0,1\}^r
        }}\bgamma_{j\uu} n_{j\uk\uu}:\; j\in[s],\; \uk\in\CK\}}f(\bx)}\gg \delta^{O(1)} N^D.
    \end{align*}
    The kernel $\CN$ is supported on the tuples 
    $(n_{j\uk\uu})_{j,\uk,\uu}$ satisfying $n_{j\uk\uu}\ll H^{|\uu|} M$, of which there are
    \begin{align*}
        \ll \prod_{j\in[s]}\prod_{\uk\in\CK} \prod_{\uu\in\{0,1\}^r} H^{|\uu|} M \ll  M^{2^r s |\CK|} H^{r 2^{r-1} s |\CK|}. 
    \end{align*}
    It is at this point that we crucially incorporate the bound
    \begin{align}\label{E: naive estimate}
        \CN(\un)\ll \delta^{-O(1)}  M^{- 2^r s |\CK|} H^{-r 2^{r-1} s |\CK|}
    \end{align}
    provided by  Proposition \ref{P: systems of multilinear equations} to get the estimate
    \begin{align*}
        \E_{\substack{n_{j\uk\uu}\ll H^{|\uu|}M:\\ j\in[s],\; \uk\in\CK,\; \uu\in\{0,1\}^r
        }}\abs{\sum_\bx \Delta_{\{\sum_{\substack{\uu\in\{0,1\}^r
        }}\bgamma_{j\uu} n_{j\uk\uu}:\; j\in[s],\; \uk\in\CK\}}f(\bx)}\gg \delta^{O(1)} N^D.
    \end{align*}
From here, we use the Cauchy--Schwarz inequality to remove the absolute value and introduce Fej\'er kernels, and then use Lemma \ref{L: properties of box norms}\eqref{i: trimming 2} to trim the boxes (as in the end of the proof sketch of Proposition \ref{P: concatenation 1-deg}).  To conclude the proof, we use Lemma \ref{L: properties of box norms} to replace the differencing with respect to $[\pm N]^D$ by a differencing with respect to one of the boxes $E_j$; by the monotonicity of box norms, we can also increase the multiplicity of every other $E_i$ by $1$, and we conclude the theorem with $t+1$ in place of $t$.
\end{proof}

\section{Proof of Theorem \ref{T: single box norm bound}}\label{S: proof of main thm}

We conclude by deriving Theorem \ref{T: single box norm bound} from Proposition \ref{P: PET} (the output of PET) and Theorem \ref{T: concatenation of polynomials} (the concatenation result).
    Assume that 
    \begin{align*}
            \abs{\sum_{\bx}\E_{z\in[K]} f_0(\bx)\cdot f_1(\bx+\p_1(z))\cdots f_\ell(\bx+\p_\ell(z))} \geq \delta N^D.
    \end{align*}
    Proposition \ref{P: PET} gives us $r, s=O_{d,\ell}(1)$ and nonzero $r$-variable multilinear polynomials $\bc_1, \ldots, \bc_s\in\Z^D[\uh]$ (depending only on $\bp_1, \ldots, \bp_\ell$) of degree at most $d-1$, such that 
    \begin{align*}
        \bc_{j}(\uh) = \sum_{\substack{\uu\in \{0,1\}^{r},\\ |\uu|\leq d-1}}     (|\uu|+1)! \cdot (\bbeta_{1(|\uu|+1)}-\bbeta_{w_{j\uu}(|\uu|+1)})\uh^\uu
    \end{align*}
                for some indices $w_{j\uu}\in[0,\ell]$ (with $\bbeta_{0(|\uu|+1)}:=\mathbf{0}$), for which
      \begin{align*}
       \E_{\uh\in[\pm H]^{r}}\norm{f_1}_{{\bc_1(\uh)}\cdot[\pm M], \ldots, {\bc_s(\uh)\cdot[\pm M]}}^{2^s} \gg_{d, D, \ell} \delta^{O_{d,\ell}(1)}N^D
   \end{align*}
   holds for all $\delta^{-O_{d,\ell}(1)}\ll_{d, D, \ell} H\leq  M\ll_{d, D, \ell}\delta^{O_{d, \ell}(1)}K$.
   Furthermore, by property \eqref{i: leading coeffs} of Definition \ref{D: descendence}, $\bbeta_{1(|\uu|+1)}-\bbeta_{w_{j\uu}(|\uu|+1)}$ is the leading coefficient of $\p_1 - \p_{w_{j\uu}}$, which implies that
  \begin{align*}
      \deg (\bp_1 - \bp_{w_{j\uu}}) = |\uu|+1.
  \end{align*}
  By Theorem \ref{T: concatenation of polynomials} (with $d-1$ in place of $d$), there exists $t_0=O_{d,\ell}(1)$ such that 
    \begin{align*}
        \norm{f}_{E_{10}^{ t_0}, \ldots, E_{s0}^{ t_0}}^{2^{st_0}}\gg_{d, D, \ell}\delta^{O_{d, \ell}(1)}N^D
    \end{align*}
    for the boxes 
    \begin{align*}
        E_{j0} = \sum_{\substack{\uu\in \{0,1\}^{r},\\ |\uu|\leq d-1}}     d_{1w_{j\uu}}! \cdot (\bbeta_{1d_{1w_{j\uu}}}-\bbeta_{w_{j\uu}d_{1w_{j\uu}}})\cdot [\pm H^{d_{1w_{j\uu}}}]
    \end{align*}    
    as long as $H^{d-1} M V \leq N$.  We now set $H = M \asymp_{d, D, \ell} \delta^{c}K$ for a suitably large constant $c=c_{d,\ell}>0$,
which, together with the assumption on the size of $K$, ensures that the aforementioned conditions on $H,M$ are satisfied.  Using parts \eqref{i: enlarging}-\eqref{i: trimming 2} of Lemma \ref{L: properties of box norms}, we perform the following operations on each box $E_{j0}$: Remove all of the directions making up $E_{j0}$ except for one direction corresponding to a (nonzero) highest-degree monomial; enlarge $H$ to $K$ (at the loss of the factor $O_{d, D, \ell}(\delta^{O_{d, \ell}(1)})$ in the lower bound); and remove the factor $d_{1w_{j\uu}}!$.  In total, we replace each box $E_{j0}$ by
    \begin{align*}
        (\bbeta_{1d_{1w_{j\uu}}}-\bbeta_{w_{j\uu}d_{1w_{j\uu}}})\cdot [\pm K^{d_{1w_{j\uu}}}]
    \end{align*}
    for some $w_{j\uu}\neq 1$. The theorem now follows, with $t=t_0\ell$, from the monotonicity property of box norms, where we have implicitly used the bound
    \begin{align*}
        |        (\bbeta_{1d_{1w_{j\uu}}}-\bbeta_{w_{j\uu}d_{1w_{j\uu}}})\cdot K^{d_{1w_{j\uu}}}|\ll_D V K^{d_{1w_{j\uu}}} \ll_{d, D, \ell} V K^d \leq N.
    \end{align*}

\section{One more PET result}\label{S: one more PET}

In the last section, we prove one more result required in our companion paper \cite{KKL24b}.
\begin{proposition}
    Let $d,r\in\N$ with $d\geq 2$, and let $T>0$.  There exist a positive integer $t = O_{d,r}(1)$ and a positive real $C=O_{d,r,T}(1)$ such that the following holds. Let $P\in\Z[z]$ be a polynomial of degree $d$ with coefficients of size at most $V$, and let $\delta>0$ and $K,N\in\N$ satisfy $C\delta^{-C}\leq K\leq T (N/V)^{1/d}$.  Then for every $1$-bounded function $f:\Z\to\C$ supported on $[N]$, the bound
    \begin{align}\label{E: average in one more PET}
        \sum_{h_1, \ldots, h_r}\mu_K(\uh)\sum_x \E_{z\in[K]} \prod_{\ueps\in\{0,1\}^{r}}\CC^{|\ueps|}f(x + P(z+\ueps\cdot\uh)) \geq \delta N
    \end{align}
    implies that
    \begin{align*}
    \norm{f}_{U^t(\beta_d\cdot[\pm N/\beta_d])}^{2^{t}}\geq \frac{1}{C} \delta^{C} N,
\end{align*}
where $\beta_d$ is the leading coefficient of $P$.
\end{proposition}

We remark that the bounds in the preceding proposition are uniform in the choice of $\beta_d$: in particular, $\beta_d$ can depend on $N$. This uniformity is crucially explored in \cite{KKL24b}.
\begin{proof}

    Write $P(z) =  \sum_{i=0}^d \beta_i z^i$. Setting 
    \begin{align*}
        P_\ueps(z,\uh) := P(z+\ueps\cdot\uh) - P(\ueps\cdot\uh),
    \end{align*}
    we first show that the family 
    $\CQ := (P_\ueps:\; \ueps\in\{0,1\}^r)$ 
    (ordered so that $P_{\underline{1}}$ comes first) descends from $\CP=(P - P(0))$. Indeed, $\CQ$ is normal, for which it is crucial that $d\geq 2$, since otherwise $P_\ueps(z, \uh) = \beta_1 z$ for every $\ueps$, as the contributions of $\uh$ are absorbed into the term not depending on $z$. 
    Moreover, a simple computation gives that
    \begin{align*}
        P_\ueps(z,\uh) &= \sum_{i=0}^d z^i \sum_{\substack{\uu\in\N_0^r:\; |\uu|\leq d-i,\\ \supp(\uu)\subseteq\supp(\ueps)}}{{|\uu|+i}\choose{\uu}}\beta_{|\uu|+i} \uh^\uu, 
    \end{align*}
    and hence the formula \eqref{E: gamma_ji} holds with $w_{\ueps\uu} = \mathbf{1}(\supp(\uu)\subseteq\supp(\ueps))$ and $w_\uu = 0$. This immediately implies property \eqref{i: w_1} from Definition \ref{D: descendence}. One also sees that for distinct $\ueps,\ueps'\in\{0,1\}^r$, the leading coefficient of 
    $P_\ueps-P_{\ueps'}$ as a polynomial in $z$ is $d \beta_d (\ueps-\ueps')\cdot\uh$. In particular, the leading coefficient of $P_\ueps-P_{\ueps'}$ is a multilinear homogeneous polynomial of degree $1$. This gives property \eqref{i: multilinear} of Definition \ref{D: descendence}, and property \eqref{i: leading coeffs} follows since $\beta_d$ is the leading coefficient of $P$. 
    Hence $\CQ$ indeed descends from $\CP$. This allows us to use Proposition \ref{P: vdC terminates} to find positive integers $r'=O_{d,r}(1)$ and $m_1, \ldots, m_{r'}$ such that $\CQ'=\partial_{m_{r'}, \ldots, m_1}\CQ$ is linear in $z$ and has size $s:=|\CQ'|\ll_{d,r} 1$.
    
    Since Proposition \ref{P: vdC terminates} also gives that $\CQ'$ descends from $\CP$, the members of $\CQ'$ are multilinear homogeneous polynomials of degree $d-1$ (homogeneity follows from remark \eqref{i: homogeneity} below Definition \ref{D: descendence}). Letting $\uh = (h_1, \ldots, h_{r''})$ from now on with  $r'':=r+r'$, we can enumerate the elements of $\CQ'$ as
    \begin{align*}
    b_1(\uh)z, \ldots, b_s(\uh)z,
    \end{align*}
     where
    \begin{align*}
    b_j(\uh) = \sum_{\uu\in I_j}d! \beta_d \uh^\uu
    \end{align*}
    and the sets $I_j\subseteq \{\uu\in\{0,1\}^{r''}:\; |\uu| = d-1\}$ for $j\in[s]$ are distinct and nonempty.

    Applying the Cauchy-Schwarz and van der Corput inequalities $r'$ times to \eqref{E: average in one more PET} much like\footnote{Technically, the expression \eqref{E: average in one more PET} differs from the one in Proposition \ref{P: PET} in that in \eqref{E: average in one more PET}, we have an extra averaging over $h_1, \ldots, h_r$ upfront. However, the argument for obtaining the bound in terms of an average of box norms for the averaged expression in \eqref{E: average in one more PET} goes exactly the same way as in Proposition \ref{P: PET}. Check e.g. the proof of \cite[Corollary 4.3]{Kuc22b} starting from (14) to see why such averaged counting operators are no harder to analyze than ones without an additional average over $h_1, \ldots, h_r$ upfront.
    The applications of the Cauchy-Schwarz and van der Corput inequalities give us new shifting parameters $h_{r+1}, \ldots, h_{r''}$; hence the polynomial parameters $\uh=(h_1, \ldots, h_{r''})$ come from combining the ``old'' $h_1, \ldots, h_r$ with the ``new'' $h_{r+1}, \ldots, h_{r''}$.}  in the proof of Proposition \ref{P: PET}, we deduce the bound
   \begin{align*}
       \E_{\uh\in[\pm H]^{r''}}\norm{f}_{b_1(\uh)\cdot[\pm H], (b_1(\uh)-b_2(\uh))\cdot[\pm H], \ldots, (b_1(\uh)-b_s(\uh))\cdot[\pm H]}^{2^s} \gg_{d, D, \ell} \delta^{O_{d,\ell}(1)}N^D,
   \end{align*}
   valid for all $H\ll_{d,\ell} \delta^{O_{d,\ell}(1)}K$.  Theorem \ref{T: concatenation of polynomials} combined with Lemma \ref{L: properties of box norms}\eqref{i: trimming} gives a positive integer $t=O_{d,r}(1)$ for which
    \begin{align*}
    \norm{f}_{U^t(d!\beta_d\cdot[\pm H^d])}^{2^{t}}\gg_{d,r} \delta^{O_{d,r}(1)} N.
\end{align*}
Properties \eqref{i: monotonicity} and \eqref{i: passing to APs} of Lemma \ref{L: properties of box norms} allow us to replace the progression $d!\beta_d\cdot[\pm H^d]$ with $\beta_d\cdot[\pm N/\beta_d]$ upon taking $H\asymp_{d,\ell} \delta^{O_{d,\ell}(1)}K\asymp_{d,\ell} \delta^{O_{d,\ell}(1)}N/V$, giving the claimed result.

\end{proof}

  \bibliography{library}
\bibliographystyle{plain}
\end{document}